\documentclass[sn-mathphys-num]{sn-jnl}

\usepackage{graphicx}%
\usepackage{multirow}%
\usepackage{amsmath,amssymb,amsfonts}%
\usepackage{amsthm}%
\usepackage{mathrsfs}%
\usepackage[title]{appendix}%
\usepackage{xcolor}%
\usepackage{textcomp}%
\usepackage{manyfoot}%
\usepackage{booktabs}%
\usepackage{algorithm}%
\usepackage{algorithmicx}%
\usepackage{algpseudocode}%
\usepackage{listings}%

\DeclareMathOperator{\cl}{cl}
\DeclareMathOperator{\conv}{conv}
\DeclareMathOperator{\cor}{cor}
\DeclareMathOperator{\intt}{int}
\DeclareMathOperator{\bd}{bd}

\DeclareMathOperator{\R}{\mathbb{R}}
\DeclareMathOperator{\Rpos}{\mathbb{P}}

\DeclareMathOperator{\uB}{\mathbb{B}}
\DeclareMathOperator{\uS}{\mathbb{S}}
\DeclareMathOperator{\A}{\mathbb{A}}

\DeclareMathOperator{\Eff}{Eff}
\DeclareMathOperator{\WEff}{WEff}
\DeclareMathOperator{\PEff}{PEff}

\theoremstyle{thmstyleone}%
\newtheorem{theorem}{Theorem}
\newtheorem{proposition}[theorem]{Proposition}%
\newtheorem{lemma}[theorem]{Lemma}%
\newtheorem{corollary}[theorem]{Corollary}%

\theoremstyle{thmstyletwo}%
\newtheorem{example}{Example}%
\newtheorem{remark}{Remark}%

\theoremstyle{thmstylethree}%
\newtheorem{definition}{Definition}%

\begin{document}

\title[Bishop-Phelps Type Scalarization for Vector Optimization]{Bishop-Phelps Type Scalarization for Vector Optimization in Real Topological-Linear Spaces}


\author[1]{\fnm{Christian} \sur{G\"unther}}\email{c.guenther@ifam.uni-hannover.de}
\equalcont{These authors contributed equally to this work.}

\author*[2]{\fnm{Bahareh} \sur{Khazayel}}\email{Bahareh.Khazayel@mathematik.uni-halle.de}
\equalcont{These authors contributed equally to this work.}

\author[3]{\fnm{Radu} \sur{Strugariu}}\email{rstrugariu@tuiasi.ro}
\equalcont{These authors contributed equally to this work.}

\author[2]{\fnm{Christiane} \sur{Tammer}}\email{Christiane.Tammer@mathematik.uni-halle.de}
\equalcont{These authors contributed equally to this work.}

\affil[1]{\orgdiv{Institut f\"ur Angewandte Mathematik}, \orgname{Leibniz Universit\"at Hannover}, \orgaddress{\street{Welfengarten 1}, \city{Hannover}, \postcode{30167}, \country{Germany}}}

\affil[2]{\orgdiv{Institute of Mathematics}, \orgname{Martin Luther University Halle-Wittenberg, Faculty
		of Natural Sciences II}, \city{Halle (Saale)}, \postcode{06099},  \country{Germany}}

\affil[3]{\orgdiv{Department of Mathematics and Informatics}, \orgname{Gheorghe Asachi Technical University of Ia\c si}, \orgaddress{\city{Ia\c si}, \postcode{700506}, \country{Romania}}}


\abstract{It is well-known that scalarization techniques (e.g., in the sense of Gerstewitz; Kasimbeyli; Pascoletti and Serafini; Zaffaroni) are useful for generating (weakly, properly) efficient solutions of vector optimization problems. One recognized approach is the conic scalarization method in vector optimization in real normed spaces proposed by Kasimbeyli (2010, SIAM J Optim 20), which is based on augmented dual cones and Bishop-Phelps type (norm-linear) scalarizing functions. In this paper, we present new results on cone separation in real topological-linear spaces by using Bishop-Phelps type separating cones / separating seminorm-linear functions. Moreover, we study some extensions of known scalarization results in vector optimization (in the sense of Eichfelder; Gerstewitz; Jahn; Kasimbeyli; Pascoletti and Serafini). On this basis, we propose a Bishop-Phelps type scalarization method for vector optimization problems in real topological-linear spaces, which is based on Bishop-Phelps type cone-representing and cone-monotone scalarizing functions (e.g., Gerstewitz scalarizing functions or seminorm-linear scalarizing functions). Thus, our method also extends Kasimbeyli's conic scalarization method from real normed spaces to real topological-linear spaces. Within this framework, we derive new Bishop-Phelps type scalarization results for the concepts of weak efficiency and different types of proper efficiency.}




\maketitle
\section{Introduction} \label{sec:introduction}


The field of vector optimization in real topological-linear spaces is concerned with the minimization of a vector-valued function $f: X \to Y$ acting between real topological-linear spaces $X$ and $Y$ with respect to a convex cone $C \subseteq Y$ over a nonempty feasible set $\Omega \subseteq X$. A review of the existing literature reveals that many papers focus on the (finite dimensional) normed setting and, in particular, the analysis includes real normed spaces $Y$ where the convex cone $C$ is of Bishop-Phelps type, see e.g. \cite{CastanoEtAl2023}, \cite{DST2013, DST2017}, \cite{Eichfelder2008, Eichfelder2009}, \cite{Gasimov}, \cite{GJN2006}, \cite{Kasimbeyli2010, Kasimbeyli2013}, \cite{PascolettiSerafini1984}, \cite{Zaffaroni2003}, \cite[Ch. 7 and Sec. 15.2]{TammerWeidner2020}, \cite{EichfelderKasimbeyli2014}, \cite{GuenKhaTam23a, GuenKhaTam23b}, \cite{HaJahn2017, HaJahn2021}, \cite{Ha2022},\cite{Jahn2009, Jahn2023}, and references therein. The importance of the class of Bishop-Phelps cones can be emphasized by the fact that a nontrivial, closed, convex cone in a real normed space is representable as a Bishop-Phelps cone if and only if it is well-based.
We are primarily interested in the more general setting of a real (possibly infinite-dimensional and not normed) topological-linear space $Y$  motivated by optimization problems under uncertainty and problems from mathematical finance. Especially, robust counterparts to optimization problems under uncertainty can be formulated as vector optimization problems with an infinite-dimensional topological-linear image space (see G\"opfert et al. \cite[Ch. 6]{GoeRiaTamZal2023} and references therein). Furthermore, in arbitrage theory (see the class of financial market models in Marohn and Tammer \cite{MarohnTammer2021, MarohnTammer2022}, and Riedel \cite{Riedel2015}), problems in the setting of linear spaces are considered.
Some further works in vector optimization that focus on this setting are \cite{BotGradWanka}, \cite{Gerstewitz1983, Gerstewitz1984}, \cite{GerthWeidner}, \cite{GoeRiaTamZal2023}, \cite{GunKobSchmoTam2023}, \cite{GJN2015}, \cite{Grad2015}, \cite{Jahn2011, Jahn2023}, \cite{TammerWeidner2020}.

\par

The main focus in our paper is on the derivation of scalarization results for (weakly, properly) efficient solutions of the given vector optimization problem, where the proper efficiency concepts considered in this paper cover certain known concepts of proper efficiency from the literature (e.g., in the sense of Benson \cite{Benson1979}; Borwein \cite{Borwein1977, Borwein1980}; Henig \cite{Henig2}; Hurwicz \cite{Hurwicz}). 
It is well established that scalarization methods are useful for solving vector optimization problems. A notable feature of such methods is the ability to compute weakly efficient solutions by solving scalar optimization problems, which are often parameter-dependent, instead of the original vector optimization problem. 
Moreover,  using scalarization, one can derive necessary and sufficient optimality conditions for the initial vector optimization problem, using generalized differentiation concepts for the new, scalarized objective function (see, e.g.,  \cite{GM2004}, \cite{GJN2006}, \cite{DT2006}, \cite{GJN2010}, \cite{DST2013}, \cite{GJN2015}, \cite{DST2017}). 
The application of the scalarization methods for Henig dilating cones (enlarging convex cones) of the convex cone $C$ yields solutions that are properly efficient in the sense of Henig \cite{Henig2}. The relationships of the Henig proper efficiency concept to other known proper efficiency concepts have been well studied in the literature (see, e.g., \cite[Sec. 2.4.3]{BotGradWanka}, \cite[Ch. 3]{Grad2015}, \cite{GuerraggioMolhoZaffaroni}, \cite[Sec. 2.4]{Khanetal2015} and references therein). 

One recognized approach in vector optimization in a real normed space $(Y, ||\cdot||)$ is the \textbf{conic scalarization method} in vector optimization proposed by 
Kasimbeyli \cite{Kasimbeyli2010}, which is based on the following parameterized scalar optimization problem:

\begin{equation}
    \label{scal_problem_Kas_normed}
	\begin{cases}
		x^*(f(x) - a) + \alpha ||f(x) - a||\to \min \\
		x \in \Omega
	\end{cases}
\end{equation}
with a pair $(x^*, \alpha) \in Y^* \times \mathbb{R}_+$ (respectively, $(x^*, \alpha)$ belongs to the so-called augmented dual cone of $C$) and a given point $a \in Y$. Examining \eqref{scal_problem_Kas_normed}, one can see that Kasimbeyli's method is based on Bishop-Phelps type norm-linear scalarizing functions (i.e., functions that are a nonnegative weighted sum of the norm $||\cdot||$ with a linear continuous functional $x^*$, see also Zaffaroni \cite{Zaffaroni2022}), where the involving norm $||\cdot||$ is exactly the underlying norm in the real normed space $Y$. 
Several relationships between (weakly, properly) efficient solutions of the original vector problem and solutions of scalar problems of type \eqref{scal_problem_Kas_normed} are known (see \cite{CastanoEtAl2023}, \cite{Gasimov}, \cite{GuoZhang2017}, \cite{Kasimbeyli2010, Kasimbeyli2013}, \cite{Kasimbeyli2019}). 

In this paper, we propose a \textbf{Bishop-Phelps type scalarization method} for vector optimization problems in real topological-linear spaces, which is designed to complement and extend the existing scalarization frameworks. Our method is built around cone-representing and cone-monotone scalarizing functions, which can be chosen in a very broad way according to the underlying ordered topological vector space. In this sense, classical Gerstewitz-type scalarizing functions, as well as Bishop–Phelps type seminorm-linear scalarizing functions -- given by a nonnegative weighted sum of a continuous seminorm and a continuous linear functional -- appear as particular instances of the general construction. This flexibility is especially well-suited to the topological-linear space setting considered here, since it allows one to work for instance with continuous seminorms rather than a single norm and therefore adapts naturally to locally convex, and even nonnormable, frameworks. The underlying scalarizations of the vector problem will be of the types (in the sense of Gerstewitz  \cite{Gerstewitz1983,Gerstewitz1984}; Jahn \cite{Jahn2023}; Pascoletti and Serafini \cite{PascolettiSerafini1984}):
\begin{equation*}
	\label{scal_problem_phi_1}
    \tag{\ensuremath{P_{\varphi}^{a}}}
	\begin{cases}
		\varphi(f(x) - a)\to \min \\
		x \in \Omega
	\end{cases}
\end{equation*}
and 
\begin{equation*}
	\label{scal_problem_phi_2}
	\tag{\ensuremath{P_{\varphi}^{a,k}}}
	\begin{cases}
		\lambda \to \min \\
		\varphi(f(x) - a - \lambda k) \leq 0,\\
		(x, \lambda) \in \Omega \times \mathbb{R},
	\end{cases}
\end{equation*}
where $a \in Y$, $k \in Y \setminus \{0\}$, and $\varphi: Y \to \mathbb{R}$ is a Bishop-Phelps (type) cone-representing function, i.e., the level set 
$$
C_{\rm BP} := -\{y \in Y \mid \varphi(y) \leq 0\}
$$
is a Bishop-Phelps (type) cone in $Y$.
In Kasimbeyli's conic scalarization method based on the scalar problem \eqref{scal_problem_Kas_normed}, the set
$$
C_{\rm BP} = - \{y \in Y \mid x^*(y) + \alpha ||y|| \leq 0\}
$$
with $(x^*, \alpha) \in (Y^* \setminus \{0\}) \times (\mathbb{R}_+ \setminus \{0\})$ is indeed a Bishop-Phelps cone in the real normed space $Y$. Hence, our method can be seen as an extension of Kasimbeyli's scalarization method from real normed spaces to real topological-linear spaces.
\par

Under the usual conditions, the scalar surrogate problems based on the Gerstewitz function generate weakly efficient solutions with respect to the cone involved in the definition of the Gerstewitz function (see \cite[Theorem 6.2.31]{TammerWeidner2020}). Since the scalar optimization problem by Pascoletti and Serafini (see \cite{PascolettiSerafini1984}) is a special case of the scalar surrogate problem described by the Gerstewitz function (see \cite[Section 5.1]{TammerWeidner2020}), this statement also applies to the scalar surrogate problem by Pascoletti and Serafini. Furthermore, the Hiriart-Urruty scalarization function introduced in \cite{HiriartUrruty1979} coincides with the Gerstewitz function if the distance function involved in the Hiriart-Urruty function is induced by the Minkowski functional of the order interval (see \cite[Theorem 9.3.5]{TammerWeidner2020}). Under this condition, the scalar surrogate problem based on the Hiriart-Urruty function also generates weakly efficient solutions.

Properly efficient solutions are very important from the theoretical as well as computational point of view. These solutions can be generated using the classical Gerstewitz function too, but the set involved in the definition of the function must be an enlargement cone of the ordering cone such that one has to construct such an enlargement cone. In our paper, we are doing this by using a Bishop–Phelps type cone based on a seminorm, in the definition of the Gerstewitz function. This leads us to a scalar surrogate problem where a Bishop-Phelps (type) cone-representing function is involved. The advantage of the new Bishop–Phelps type scalarization in \eqref{scal_problem_phi_1}  is that it can be used to generate directly properly efficient solutions to vector optimization problems. 

\par

The rest of the paper is organized as follows: In Section \ref{sec:preliminaries}, after presenting some basics in real topological-linear spaces including properties of seminorms, we study properties of cones and their bases, focusing in particular on augmented dual cones, Bishop-Phelps type cones and zero level sets of Bishop-Phelps type (seminorm-linear) functions. We then introduce the solution concepts of vector optimization necessary for our work. Section \ref{sec:scalarization} is devoted to the study of scalarization methods in vector optimization which are based on cone-monotone and cone-representing scalarizing functions. Besides classical scalarization results (in the sense of Eichfelder \cite{Eichfelder2008, Eichfelder2009}; Gerstewitz \cite{Gerstewitz1983, Gerstewitz1984}; Jahn \cite{Jahn2023}; Pascoletti and Serafini \cite{PascolettiSerafini1984}), we establish some new scalarization results (namely Propositions \ref{prop:result_scalarization_Jahn_2_new}, \ref{prop:result_scalarization_Jahn_3_new} and \ref{prop:Eichfelder}) for vector optimization problems in real topological-linear spaces, which extend the classical ones.
The main Section \ref{sec:conic_scalarization} is dedicated to Bishop-Phelps type scalarization for vector optimization in real topological-linear spaces. The parameterized scalarization problems \eqref{scal_problem_phi_1} and \eqref{scal_problem_phi_2} are formulated using 
Bishop-Phelps type cone-representing  and cone-monotone scalarizing functions (e.g., Gerstewitz \cite{Gerstewitz1983, Gerstewitz1984} scalarizing functions or seminorm-linear scalarizing functions).
In addition to the basic scalarization results for the original vector optimization problem (roughly speaking, solutions of the scalar problems are solutions of the vector problem), we also point out (for the concept of efficiency in Proposition \ref{prop:scal_result_Eff_BP} and for the concept of weak efficiency in Proposition \ref{prop:scal_result_WEff_BP}) the relationships between the original vector problem (based on the convex cone $C$), an updated vector problem based on an enlarging / dilating cone of Bishop-Phelps type, and scalar problems involving Bishop-Phelps type cone-representing and cone-monotone scalarizing functions.
In Section \ref{sec:cone_separation}, we present the key tool to prove our main Bishop-Phelps type scalarization results in vector optimization, namely the technique of nonlinear separation of two (not necessarily convex) cones by using Bishop-Phelps type separating cones / separating (seminorm-linear) functions. Besides presenting some classical nonlinear cone separation results, we also reflect on the separation conditions involved in both real topological-linear spaces and real normed spaces.
Propositions \ref{prop:separation_conditions_1} and \ref{prop:separation_conditions_2}
give some new insights into the relationships
between the classical separation condition (the intersection of two cones is the singleton set $\{0\}$) 
and the separation conditions used in the nonlinear separation approach based on Bishop-Phelps type cones / functions 
discussed in \cite{CastanoEtAl2023}, \cite{GuenKhaTam23a, GuenKhaTam23b}, \cite{Kasimbeyli2010, Kasimbeyli2013}. 
On this basis, in Corollaries \ref{cor:new_sep_1} and \ref{cor:new_sep_2} we present new variants of strict cone separation results involving the classical separation condition  in real locally convex spaces and real normed spaces by using Bishop-Phelps type separating cones / separating seminorm-linear functions.
Then, using nonlinear cone separation results (from Section \ref{sec:cone_separation}) and Bishop-Phelps type scalarization results (Propositions \ref{prop:scal_result_Eff_BP} and  \ref{prop:scal_result_WEff_BP}), we present new scalarization theorems (Theorem \ref{th:main_scalarization_result_WEff} for weak efficiency, 
Theorems \ref{th:main_scalarization_result_PEff}, \ref{th:main_scalarization_result_PEff_Henig_1} and \ref{th:main_scalarization_result_PEff_Henig_2} for proper efficiency) and their corollaries, which show that any (weakly, properly) efficient solution of the original  vector problem is a solution of the above mentioned (parameterized) scalar problems for certain choices of parameters. 
Finally, in Section \ref{sec:conclusions} we present a conclusion and an outlook for further research. 

\section{Preliminaries} \label{sec:preliminaries}

Throughout the paper, we assume that $Y$ is a real topological-linear space. 
As usual, $Y^*$ denotes the topological dual space of $Y$. 
In certain situations we will also consider a more specific setting where $(Y, ||\cdot||)$ is a real normed space. Consistently, 
the sets of nonnegative and positive real numbers are denoted by $\mathbb{R}_+$ and $\mathbb{P}$. Given a set $\Omega \subseteq Y$, we denote by ${\rm conv}\,\Omega$ the convex hull of $\Omega$. Moreover, we denote by ${\rm cl}\,\Omega$, ${\rm bd}\,\Omega$, and ${\rm int}\,\Omega$ the closure, the boundary, and the interior of $\Omega$ with respect to the underlying topology in $Y$. The algebraic interior (the core) of $\Omega$ is defined by $\cor\, \Omega := \{x \in \Omega \mid \forall\, v \in Y\; \exists\, \varepsilon > 0: \; x + [0, \varepsilon] \cdot v \subseteq \Omega\}.$ 

Given a seminorm $\psi: Y \to \R$ (i.e., $\psi$ is sublinear and symmetric), we define the $\psi$-unit ball by
$
\uB_{\psi} := \{y \in Y \mid \psi(y) \leq 1\}
$
and the $\psi$-unit sphere by
$
\uS_{\psi} := \{y \in Y \mid \psi(y) = 1\}.
$
In what follows, we will suppose that $\psi$ is nontrivial (i.e., $\psi \not\equiv 0$).
It is well-known that if $Y$ is a real topological-linear space and $\psi : Y \to \mathbb{R}$ is a seminorm, then $\psi $ is upper semicontinuous $\Longleftrightarrow$ $\{y\in
Y\mid \psi (y) < 1\}$ is open  $\Longleftrightarrow$ $\psi $ is continuous at the origin $\Longleftrightarrow$ $\psi $ is continuous. Moreover, $\psi$ is lower semicontinuous  $\Longleftrightarrow$ $\uB_{\psi}$ is closed. Remark also that $\{y\in Y\mid \psi (y) = 0\} = \{0\}$ $\Longleftrightarrow$ $\psi$ is a norm. The singleton set $\{0\}$ is closed $\Longleftrightarrow$ the topological-linear space $Y$ is Hausdorff. Moreover, it is known (see, e.g.,  Gelfand \cite{Gelfand1935, Gelfand1938}, Eberlein \cite{Eberlein1946}) that if $Y$ is a Banach space and $\psi: Y \to \mathbb{R}$ is a lower semicontinuous seminorm, then $\psi$ is continuous. 

A set $K \subseteq Y$ is called a cone if $0 \in K = \R_+ \cdot K$. A cone $K \subseteq Y$ is said to be convex if $K + K = K$ (or equivalently, $K$ is a convex set); proper or nontrivial if $\{0\} \neq K \neq Y$; pointed if $\ell(K) : = K \cap (-K) = \{0\}$. 
Given a convex cone $K \subseteq Y$ in the real topological-linear space $Y$, the set $\ell(K)$ is known as the lineality space of $K$.  In this case, $K$ is a linear subspace of $Y$ $\Longleftrightarrow$ $K = \ell(K)$, and moreover, $\intt\,K = K + \intt\,K$ and $K \setminus \ell(K) = K + (K \setminus \ell(K))$.
If $K \subseteq Y$ is a cone, then $K \neq Y \iff 0 \notin \cor\,K \iff 0 \notin \intt\,K,$
and if further $K$ is convex, then (see also \cite[Lem. 2.10]{Khazayel2021a}) $K \neq Y \iff \ell(K) \cap \intt\,K = \emptyset.$

The topological dual cone of a cone $K \subseteq Y$ is given by $K^+ := \{y^* \in Y^* \mid \forall\, k \in K:\; y^*(k) \geq 0\}.$
Furthermore, the subset $K^\# := \{y^* \in Y^* \mid \forall\, k \in K \setminus \{0\}:\; y^*(k) > 0\}$
of $K^+$ is of special interest.
It is important to mention that both sets $K^+$ and $K^\#$ are convex for any (not necessarily convex) cone $K \subseteq Y$. Moreover, one has 
\begin{equation}
	\label{eq:K+=convK+=clvonvK+}
	K^{+} = (\conv\,K)^+ = (\cl(\conv\,K))^+ \text{ and } (\cl(\conv\,K))^\# \subseteq K^{\#} = (\conv\,K)^\#,   
\end{equation}
but the inclusion $(\cl(\conv\,K))^\# \subseteq K^{\#}$ can be strict (see G\"opfert et al.  \cite[p. 55]{GoeRiaTamZal2023}).

In the following definition, we recall a general base concept for cones in real topological-linear spaces (cf. G\"opfert et al.  \cite[Def. 2.1.42]{GoeRiaTamZal2023} for the case of a convex cone). 

\begin{definition}[Base] \label{def:top_base_cone}
	Consider a nontrivial cone $K \subseteq Y$. A set $B \subseteq K$ is called a \textbf{base} for $K$, if
	$B$ is a nonempty set,  and $K = \R_+ \cdot B$ with $0 \notin \cl\, B$.
	Moreover, $K$ is said to be well-based if there exists a bounded, convex base of $K$.
\end{definition}

\begin{remark} \label{rem:clconvKwell-based}
	Taking into account \eqref{eq:K+=convK+=clvonvK+}, for any nontrivial cone $K \subseteq Y$ in a real normed space $Y$, it is well-known (see \cite[Prop. 2.2.23 and 2.2.32]{GoeRiaTamZal2023}) that
	\begin{equation} \label{eq:KbasedKSharpneqEmpty}
		\conv\,K \text{ has a convex base} \iff  K^\# \neq \emptyset,\end{equation} 
	and 
	\begin{equation} \label{eq:KwellbasedIntK+neqEmpty}
		\conv\,K \text{ is well-based}\iff \cl(\conv\,K) \text{ is well-based} \iff \intt\,K^{+}\neq\emptyset .
	\end{equation} 
	Moreover, if $Y$ has finite dimension and $K$ is closed, then $K^\# = \intt\,K^{+}$ (see \cite[Th. 2.1(4), Rem. 2.6]{GuenKhaTam23b}), hence all the conditions involved in \eqref{eq:KbasedKSharpneqEmpty} and \eqref{eq:KwellbasedIntK+neqEmpty} are equivalent.
\end{remark}

\begin{example}[Bishop-Phelps cone] \label{ex:Bishop-Phelps_cone}
	Assume that $Y$ is a real normed space with the underlying norm $||\cdot||: Y \to \mathbb{R}$. For any $(x^*, \alpha) \in Y^* \times \Rpos$, a \textbf{Bishop-Phelps cone} (in the sense of Bishop and Phelps \cite{BP1962}) is defined by
    \begin{equation*}
        C_{||\cdot||}(x^*,\alpha) := \{ y \in Y \mid x^*(y) \geq \alpha ||y|| \}.
    \end{equation*}
	It is known that Bishop-Phelps cones have a lot of useful properties and there are interesting applications in variational analysis and optimization (see, e.g., Ha and Jahn \cite{HaJahn2017, HaJahn2021}, Kasimbeyli \cite{Kasimbeyli2010}, Phelps \cite{Phe93}). 
	Any Bishop-Phelps cone $C_{||\cdot||}(x^*,\alpha)$ is a closed, pointed, convex cone. If $||x^*||_* > \alpha$ (where $||\cdot||_*: Y^* \to \mathbb{R}$ denotes the dual norm of $||\cdot||$), then $C_{||\cdot||}(x^*,\alpha)$ is nontrivial and 
	$C_{||\cdot||}^>(x^*,\alpha) := \{ y \in Y \mid x^*(y) > \alpha ||y||\} = \cor\, C_{||\cdot||}(x^*,\alpha) = \intt\, C_{||\cdot||}(x^*,\alpha) \neq \emptyset$.
	Note that for $(x^*, \alpha) \in Y^* \times \Rpos$ with $||x^*||_* <  \alpha$ we have $C_{||\cdot||}(x^*,\alpha) = \{0\}$, while for $||x^*||_* \leq  \alpha$ we have $C_{||\cdot||}^>(x^*,\alpha) = \emptyset$.
	
    It is worth noting that every nontrivial, closed, convex cone $K \subseteq Y$ in a normed space $Y$ is \textbf{representable as a Bishop-Phelps cone} (i.e., there is a pair $(x^*, \alpha) \in Y^* \times \Rpos$ and an equivalent norm $||\cdot||' \sim ||\cdot||$ such that $K = C_{||\cdot||'}(x^*,\alpha)$) if and only if $K$ is well-based (see the proof of Proposition \ref{prop:separation_conditions_2}). Note that any nontrivial, closed, pointed, convex cone $K$ in a finite-dimensional normed space $Y$ is well-based, since in this setting $K$ is well-based if and only if $K^\# \neq \emptyset$ (taking into account Remark \ref{rem:clconvKwell-based}), and the latter condition $K^\# \neq \emptyset$ is true by the well-known Krein-Rutman theorem (see, e.g., \cite[Th. 2.4.7]{Khanetal2015}).
    Further properties of Bishop-Phelps cones are studied in \cite[Th. 4.5, Lem. 4.2]{KasKas17}, particularly another characterization for cones which are representable as a Bishop-Phelps cone (a so-called representation theorem). See also Lemma  \ref{lem:BPcone} and Remark \ref{rem:Representation_Theorem}.
\end{example}

Moreover, we recall another (algebraic) base concept for cones in linear spaces (cf. G\"opfert et al. \cite[Def. 2.1.14]{GoeRiaTamZal2023}, Jahn \cite[Def. 1.10 (d)]{Jahn2011} for the case of a convex cone).

\begin{definition}[Algebraic base] \label{def:alg_base_cone}
	Consider a nontrivial cone $K \subseteq Y$. A set $B \subseteq K \setminus \{0\}$ is called an \textbf{algebraic base} for $K$, if
	$B$ is a nonempty set,  and every $x \in K \setminus \{0\}$ has a unique representation of the form 
	$x = \lambda b$ for some $\lambda > 0$ and some $b \in B$.
\end{definition}

\begin{remark} \label{rem:bases}
	If $B$ is base in the sense of Definition \ref{def:top_base_cone} or Definition \ref{def:alg_base_cone} for the nontrivial cone $K \subseteq Y$, then $K = \R_+ \cdot B$ and $K \setminus \{0\} = \Rpos \cdot B$. 
\end{remark}

Our focus will be on a specific type of base in real topological-linear spaces, which we will recall in the following definition.

\begin{definition}[Normlike-base] \label{def:normlikebase}
	Consider a seminorm $\psi: Y \to \R$ and a cone $K \subseteq Y$. Define 
	\begin{center}
		$B_K := B_{K}(\psi) := \{x \in K \mid \psi(x) = 1\} = K \cap \uS_{\psi}$.
	\end{center}
	If $K$ is nontrivial and $\psi$ is positive on $K \setminus \{0\}$, then $B_K$ is called a \textbf{normlike-base} for $K$.
	If $\psi$ is a norm $||\cdot||: Y \to \R$, then $B_K(||\cdot||)$ is called a \textbf{norm-base} for $K$. 
	Note that $B_{K}(\psi)$ is a base in the sense of Definition \ref{def:alg_base_cone}, and if $\psi$ is continuous, then it is a base also in the sense of Definition \ref{def:top_base_cone}. 
\end{definition}

\begin{example}[Normlike-bases] \label{ex:normlike-bases}{\color{white}.}
\begin{itemize}
\item[a)] 
Consider a nontrivial cone $K\subseteq Y$ and some $x^*\in K^{\#}$.
Define a seminorm $\psi$ by $\psi(x):=|x^*(x)|$ for all $x\in Y$. Notice
that for dimension two or higher $\psi$ is not a norm. Since $x^*\in
K^{\#}$, we have $\psi(x)>0$ for all $x\in K\setminus\{0\}$, hence $B_{K}%
(\psi)$ is a normlike-base for $K$. Notice that for $x^*\in K^{\#}$, we
have
\[
B_{K}(\psi)=K\cap\mathbb{S}_{\psi}=\{x\in K\mid|x^*(x)|=1\}=\{x\in K\mid
x^*(x)=1\},
\]
which is a usual type of bases for convex cones.
\item[b)]  
From \cite[Lem. 2.16]{GuenKhaTam23a} we have the following results:
    
Consider a function $\varphi: Y \to \mathbb{R}$, a function $\psi_{\max}: Y \to \mathbb{R}$ defined by
\begin{equation}
\label{eq:psimax}
\psi_{\max}(y) := \max\{\varphi(y), \varphi(-y)\}
\quad \text{
for all } y \in Y,
\end{equation}
and assume that $K, A \subseteq Y$ are nontrivial cones.
Then, the following assertions hold:
\begin{itemize}
    \item[$1^\circ$] If $\varphi$ is sublinear  (i.e., $\varphi$ is positive homogeneous and subadditive), then $\psi_{\max}$ is a seminorm.
    \item[$2^\circ$] Assume that $\varphi$ is $-K$ representing. Then,  $\ell(K) = \{y \in Y \mid \psi_{\max}(y) \leq 0\}.$
    Moreover, we have:
    
    $\psi_{\max}$ is nonnegative on $Y$ $\iff$ $\ell(K) = \{y \in Y \mid \psi_{\max}(y) = 0\}$.
    
    $\psi_{\max}$ is positive on $Y \setminus \{0\}$ $\iff$  $K$ is pointed. 
            
    \item[$3^\circ$] Assume that $\varphi$ is sublinear and $-K$ representing. Then, $\psi_{\max}$ is a seminorm. Moreover, we have:
    
    $\ell(K) \cap A = \{0\} \iff$ $B_{A}(\psi_{\max})$ is a normlike-base for $A$.

    \item[$4^\circ$] If $K$ is pointed, $\varphi$ is sublinear and $-K$ representing, then $\psi_{\max}$ is a norm, hence        
    $B_{K}(\psi_{\max})$ and 
    $B_{A}(\psi_{\max})$ are norm-bases for $K$ and $A$, respectively.
\end{itemize} 

\item[c)] Consider a nontrivial cone $K\subseteq Y$ and some $x^*\in K^{\#}$. Assume that $A=\bigcup_{i=1}^{l}A_{i}$ for some nontrivial cones
$A_{1},\ldots,A_{l}\subseteq Y$, and consider $y_{i}^*\in(A_{i})^{\#}$
for all $i=1,\ldots,l$. Define a continuous seminorm $\psi$ (which is not necessarily a norm) by
\begin{equation} \label{eq:seminorm_1}
    \psi(x):=\max\{|x^*(x)|,|y_{1}^*(x)|,\ldots,|y_{l}^*%
(x)|\}\quad\text{for all }x\in Y,
\end{equation}
or by
\begin{equation} \label{eq:seminorm_2}
\psi(x):=|x^*(x)|+\sum_{i=1}^{l}\,|y_{i}^*(x)|\quad\text{for all
}x\in Y.
\end{equation}
Then, $\psi(x)>0$ for all $x\in(K\cup A)\setminus\{0\}$, hence $B_{K}(\psi)$
and $B_{A}(\psi)$ are normlike-bases for $K$ and $A$. Note that $\psi$ is a norm if and only if $\{x \in Y \mid x^*(x) = 0 = y_{1}^*(x) = \ldots = y_{m}^*(x)\} = \{0\}$. The latter condition is only valid if $Y$ has finite dimension.
\end{itemize}
\end{example}

Assume that $Y$ is a topological-linear space,  $K \subseteq Y$ is a nontrivial cone, and $\psi: Y \to \mathbb{R}$ is a seminorm. For any $(x^*, \alpha) \in Y^* \times \mathbb{R}_+$, we consider the \textbf{Bishop-Phelps type (seminorm-linear) function} $\varphi_{x^*, \alpha} : Y \to \mathbb{R}$ defined by
\begin{equation}\label{eq:seminorm_lin_fcn}
    \varphi_{x^*, \alpha}(y)  := x^*(y) + \alpha \psi(y) \quad \text{for all } y \in Y,
\end{equation}
as used in \cite{GuenKhaTam23a}.
Note that the seminorm $\psi$  must not have anything to do with the underlying topology of the real (locally convex) topological-linear space. Several examples of seminorms $\psi$, which are not necessarily norms, that can be used in the definition of the function $\varphi_{x^*, \alpha}$ in \eqref{eq:seminorm_lin_fcn} as well as in the subsequent study presented below, are discussed in \cite[Sec. 2]{GuenKhaTam23a}. Bishop-Phelps type (norm-linear) functions in a normed setting (i.e., $\psi$ is given by the underlying norm of the space) were studied for instance by Garc\'{i}a-Casta\~{n}o, Melguizo-Padial and Parzanese \cite{CastanoEtAl2023}, Gasimov (Kasimbeyli) \cite{Gasimov}, G\"unther, Khazayel and Tammer \cite{GuenKhaTam23b}, Guo and Hang \cite{GuoZhang2017}, Ha \cite{Ha2022}, Jahn \cite{Jahn2023}, Kasimbeyli \cite{Kasimbeyli2010, Kasimbeyli2013}, and Zaffaroni \cite{Zaffaroni2022}.

Following the definitions of so-called augmented dual cones by Gasimov \cite{Gasimov}, Kasimbeyli \cite{Kasimbeyli2010} (see also \cite{GuenKhaTam23a} for some generalizations), we consider a (generalized) \textbf{augmented dual cone} by
$$
K^{a+}(\psi)  = \{(x^*, \alpha) \in Y^* \times \mathbb{R}_+ \mid \forall\, y \in K:\; x^*(y) - \alpha \psi(y) \geq 0\} \quad (\subseteq K^+ \times  \mathbb{R}_+)
$$
as well as the following sets
\begin{align*}
	K^{a\circ}(\psi) & = \{(x^*, \alpha) \in Y^* \times \mathbb{R}_+ \mid \forall\, y \in \intt\, K:\;  x^*(y) - \alpha \psi(y)  > 0\} &&\quad (\subseteq (K^+ \setminus \{0\}) \times  \mathbb{R}_+ \text{ if } \intt\, K  \neq \emptyset),\\
	K^{a\#}(\psi) & = \{(x^*, \alpha) \in Y^* \times \mathbb{R}_+ \mid \forall\, y \in K \setminus \{0\}:\;  x^*(y) - \alpha \psi(y) > 0\}  &&\quad (\subseteq K^\# \times  \mathbb{R}_+).
\end{align*}
If no danger of confusion arises, we will use the notations $K^{a+}$, $K^{a\circ}$ and $K^{a\#}$. 

\begin{remark} 
	Note that $(\conv\,K)^{a+} = K^{a+}$ and $(\conv\,K)^{a\#} = K^{a\#}$, and if $\psi$ is lower semicontinuous, then 
	$(\cl(\conv\,K))^{a+} = K^{a+}$.
	It is a simple observation that for a normlike-base $B_K = B_K(\psi)$ of $K$ we have
	\begin{align*}
		K^{a+} & = \{(x^*, \alpha) \in Y^* \times \mathbb{R}_+ \mid \forall\, y \in B_K:\; x^*(y) \geq  \alpha\},\\
		K^{a\#} & = \{(x^*, \alpha) \in  Y^* \times \mathbb{R}_+ \mid \forall\, y \in B_K:\;  x^*(y) > \alpha\}.
	\end{align*}
	Furthermore, it is easy to observe that 
	$K^{a\#} \subseteq K^{a\circ} \cup K^{a+}$, and if $K$ is solid and convex, and $\psi$ is lower semicontinuous, then $K^{a\circ} \subseteq K^{a+}$.  
\end{remark}

In the following definition, we introduce two classes of functions that satisfy cone (interior) representation properties (cf. Ha \cite[Sec. 3.1]{Ha2022}, Jahn \cite[Th. 3.1]{Jahn2023}, see also Kasimbeyli \cite[Th. 3.8 and 3.9]{Kasimbeyli2010}).

\begin{definition}[Representing functionals] 
	The functional $\varphi: Y \to \mathbb{R}$ is 
	\begin{itemize}
		\item \textbf{$-K$ representing} if 
		$
		-K = \{y \in Y \mid \varphi(y) \leq 0\}.
		$
		\item\textbf{strictly $-K$ representing} if  $\varphi$ is $-K$ representing and
		$
		-\intt\, K = \{y \in Y \mid \varphi(y) < 0\}.
		$
	\end{itemize}
\end{definition}

For any $(x^*, \alpha) \in Y^* \times \R_+$, the sublevel set of the Bishop-Phelps type function $\varphi_{x^*, \alpha}$ at level zero is defined by
\begin{align*}
	S^\leq_\psi(x^*, \alpha) &:= \{ y \in Y \mid  \varphi_{x^*, \alpha}(y) = x^*(y) + \alpha \psi(y) \leq 0 \},
\end{align*}	
while a \textbf{Bishop-Phelps type cone} is defined by
$$C_\psi(x^*, \alpha) := \{ y \in Y \mid x^*(y) \geq \alpha \psi(y) \}.$$ 

Remark that, in case $\psi = || \cdot ||$ and $\alpha > 0$, this reduces to the usual Bishop-Phelps cone (see Example \ref{ex:Bishop-Phelps_cone}). Moreover, the sets
\begin{align*}
	S^{<}_\psi(x^*, \alpha) &:=  \{ y \in Y \mid  \varphi_{x^*, \alpha}(y) = x^*(y) + \alpha \psi(y) < 0 \},\\
	C^{>}_\psi(x^*, \alpha) & := \{y \in Y \mid x^*(y) >\alpha \psi(y) \}
\end{align*}	
are of interest. In our work we will focus in particular on Bishop-Phelps type functions, since they are (strictly) $-K$ representing functions if $K$ is a Bishop-Phelps type cone. Let us first note that the sets $K^{a+}$, $K^{a\circ}$ and $K^{a\#}$ can be written as
\begin{align*}
	K^{a+} & = \{(x^*, \alpha) \in Y^* \times \mathbb{R}_+ \mid K \subseteq C_\psi(x^*, \alpha)\},\\
	K^{a\circ} & = \{(x^*, \alpha) \in Y^*\times \mathbb{R}_+ \mid \intt\, K \subseteq C^{>}_\psi(x^*, \alpha)\},\\
	K^{a\#} & = \{(x^*, \alpha) \in Y^* \times \mathbb{R}_+ \mid K \setminus \{0\} \subseteq C^{>}_\psi(x^*, \alpha)\}.
\end{align*}

We collect next some important properties of Bishop-Phelps type cones and functions. For the first lemma, the proof is straightforward and we omit it.

\begin{lemma} 
	\label{lem:properties_BP}
	Consider $(x^*, \alpha) \in Y^* \times \R_+$. The following assertions hold:
	\begin{itemize}
		\item[$1^\circ$] Cone property: 
		$0 \in C_\psi(x^*, \alpha),  \Rpos \cdot C_\psi(x^*, \alpha) = C_\psi(x^*, \alpha)$ and $\Rpos \cdot C^>_\psi(x^*, \alpha) = C^>_\psi(x^*, \alpha)$.	
		\item[$2^\circ$] Convexity: 
		$C_\psi(x^*, \alpha)$ and $C^>_\psi(x^*, \alpha)$ are convex.	   
		\item[$3^\circ$] Closedness: 
		If $\varphi_{x^*, \alpha}$ (respectively, $\psi$) is lower semicontinuous, then $C_\psi(x^*, \alpha)$ is closed.  
		\item[$4^\circ$] Pointedness: 
		$\ell(C_\psi(x^*, \alpha)) = \{y \in Y \mid x^*(y) = 0, \alpha \psi(y) = 0\}$. 
		If $\alpha > 0$, then $C_\psi(x^*, \alpha)$ is pointed if and only if $\max\{|x^*|, \psi\}$ is a norm. 
		If $(x^*, \alpha) \in K^{a\#}$, then $K \cap \ell(C_\psi(x^*, \alpha)) = \{0\}$. 
		\item[$5^\circ$] Nontriviality:
		If $(x^*, \alpha) \in K^{a+}$, then $C_\psi(x^*, \alpha) \neq \{0\}$.
		If $x^* \neq 0$ or $\alpha \psi(y) > 0$ for some $y \in Y \setminus \{0\}$, then $C_\psi(x^*, \alpha) \neq Y$.     
		\item[$6^\circ$] Solidness:
		Assume that $C^{>}_\psi(x^*, \alpha) \neq \emptyset$ (e.g. if either $(x^*, \alpha) \in K^{a\#}$ or $(x^*, \alpha) \in K^{a\circ}$ and $K$ is solid). Then, 
		$
		\cor\,C_\psi(x^*, \alpha) \neq \emptyset,
		$
		and if either $(x^*, \alpha) \in K^{a\circ}$ and $K$ is solid or $\varphi_{x^*, \alpha}$ (respectively, $\psi$) is continuous, then  
		$
		\intt\,C_\psi(x^*, \alpha) \neq \emptyset.
		$
		\item[$7^\circ$] Lineality:
		If $C^{>}_\psi(x^*, \alpha) \neq \emptyset$ (e.g. if $(x^*, \alpha) \in K^{a\#}$), and either $x^* \neq 0$ or $\alpha \psi(y) > 0$ for some $y \in Y \setminus \{0\}$, then $C_\psi(x^*, \alpha)$ is not a linear subspace (or equivalently, $C_\psi(x^*, \alpha) \neq \ell(C_\psi(x^*, \alpha))$).   
		\item[$8^\circ$] Dilating cone for $K$: 
		If $(x^*, \alpha) \in K^{a\#}$, then $C_\psi(x^*, \alpha)$ is a dilating cone for $K$, i.e., $C_\psi(x^*, \alpha)$ is a nontrivial, convex cone with $K \setminus \{0\} \subseteq  \cor\,C_\psi(x^*, \alpha)$, and if further  $\intt\,C_\psi(x^*, \alpha) \neq \emptyset$, then  $K \setminus \{0\} \subseteq  \intt\,C_\psi(x^*, \alpha)$.
	\end{itemize}
\end{lemma}

\begin{lemma}  \label{lem:properties_BP_fcns}
	For any $(x^*, \alpha) \in Y^* \times \R_+$, the following assertions are valid:
	\begin{itemize}
		\item[$1^\circ$] $-C_\psi(x^*, \alpha) = S^\leq_\psi(x^*, \alpha)$ and $-C^{>}_\psi(x^*, \alpha) = S^{<}_\psi(x^*, \alpha)$.
		\item[$2^\circ$] $\varphi_{x^*, \alpha}$ is sublinear (convex) and  $-C_\psi(x^*, \alpha)$ representing.
		\item[$3^\circ$] $C^{>}_\psi(x^*, \alpha) \subseteq \cor\,C_\psi(x^*, \alpha)$, and if one has $C^{>}_\psi(x^*, \alpha) \neq \emptyset$, then $C^{>}_\psi(x^*, \alpha) = \cor\,C_\psi(x^*, \alpha)$.
		\item[$4^\circ$] If $C^{>}_\psi(x^*, \alpha) \neq \emptyset$ and $\intt\,C_\psi(x^*, \alpha) \neq \emptyset$, then
		$\varphi_{x^*, \alpha}$ is strictly $-C_\psi(x^*, \alpha)$ representing.
	\end{itemize}
\end{lemma}
\begin{proof}
	$1^\circ$ Both equalities follow from the definitions.
	
	$2^\circ$ The sublinearity (convexity) of $\varphi_{x^*, \alpha}$  follows from the sublinearity (convexity) of both $x^*$ and $\psi$. In view of $1^\circ$, $\varphi_{x^*, \alpha}$ is $-C_\psi(x^*, \alpha)$ representing.
	
	$3^\circ$ This assertion follows from Jahn \cite[Prop. 3.6]{Jahn2023} taking into account $1^\circ$.
	
	$4^\circ$ Of course, $\varphi_{x^*, \alpha}$ is $-C_\psi(x^*, \alpha)$ representing by $2^\circ$. Assume that $C^{>}_\psi(x^*, \alpha) \neq \emptyset$ and $\intt\,C_\psi(x^*, \alpha) \neq \emptyset$. By $1^\circ$ and $3^\circ$ we get $-S^{<}_\psi(x^*, \alpha) = C^{>}_\psi(x^*, \alpha) = \cor\,C_\psi(x^*, \alpha)$. Since $C_\psi(x^*, \alpha)$ is a convex set with $\intt\,C_\psi(x^*, \alpha) \neq \emptyset$ it follows that
	$\intt\,C_\psi(x^*, \alpha) = \cor\,C_\psi(x^*, \alpha)$. Thus, we conclude
	$
	-\intt\,C_\psi(x^*, \alpha) = -\cor\,C_\psi(x^*, \alpha) =  -C^{>}_\psi(x^*, \alpha) = S^{<}_\psi(x^*, \alpha),
	$
	which shows that $\varphi_{x^*, \alpha}$ is strictly $-C_\psi(x^*, \alpha)$ representing.
\end{proof}

Note that the properties of $S^\leq_\psi(x^*, \alpha)$ and $S^<_\psi(x^*, \alpha)$, and the relationships between $S^\leq_\psi(x^*, \alpha)$ and $S^<_\psi(x^*, \alpha)$ follow immediately from the one related to $C_\psi(x^*, \alpha)$ and $C^>_\psi(x^*, \alpha)$.
Furthermore, note that the conditions
$C^{>}_\psi(x^*, \alpha) \neq \emptyset$ and $\intt\,C_\psi(x^*, \alpha) \neq \emptyset$ (as given in the assumptions of Lemma \ref{lem:properties_BP_fcns} ($4^\circ$)) do not imply each other, as can be seen below: 
	
Let us first give an example with $C^{>}_\psi(x^*, \alpha) =\emptyset$ and $\intt\,C_\psi(x^*, \alpha) \neq \emptyset$. According to \cite[Ex. 2.1]{HaJahn2021}, for $Y := \R^2$, $\psi(\cdot) := ||\cdot||_1$, $x^* := (1,1)$ with $||x^*||_\infty = 1, \alpha := 1$, one has 
$C_{\psi}(x^*, 1) = \{y\in \R^2 \mid x^*(y)  = ||y||_1\} = \R^2_+$ 
and in this case $C^{>}_\psi(x^*, \alpha) = \emptyset$, but $\intt\,C_\psi(x^*, \alpha) \neq \emptyset$. 

Next, we give an example with $C^{>}_\psi(x^*, \alpha) \neq \emptyset$ and $\intt\,C_\psi(x^*, \alpha) = \emptyset$. We endow a real infinite-dimensional (normed) space $Y$ with the weak topology, we put $\psi(\cdot) := ||\cdot||$ and we assume $||x^*||_* > \alpha > 0$. In this case, it is known that $C^{>}_{||\cdot||}(x^*, \alpha) = \intt_{||\cdot||}\,C_{||\cdot||}(x^*, \alpha) \neq \emptyset$. Since $C_{||\cdot||}(x^*, \alpha)$ is a closed, pointed cone in $(Y, ||\cdot||)$ by Lemma \ref{lem:properties_BP} ($1^\circ$, $3^\circ$, $4^\circ$), we conclude that
$\intt_w\,C_{||\cdot||}(x^*, \alpha) = \emptyset$, as the following proposition shows. Recall that, for a set $\Omega \subseteq Y$ in a real normed space $(Y, ||\cdot||)$, one denotes the weak interior and weak closure of $\Omega$ by ${\rm int}_{w}\, \Omega$ and ${\rm cl}_{w}\, \Omega$, respectively.
\begin{proposition} \label{prop:ClosedPointedConeHasEmptyWeakInterior}
	Consider a closed, pointed cone $C \subseteq Y$ in a real  infinite-dimensional normed space $(Y, ||\cdot||$). Then, we have $\intt_w\,C = \emptyset$.
\end{proposition}
\begin{proof}
	On the contrary, assume that $\intt_w\,C \neq \emptyset$.
	Take some $c \in \intt_w\,C$, i.e., there is a weakly open neighbourhood $U(c) \subseteq Y$ around $c$ which satisfies $U(c) \subseteq C$.
	Since $Y$ has infinite dimension, the set $U(c)$ contains a line $c + \mathbb{R} \cdot v$ for some $v \in Y \setminus \{0\}$. This implies, for any $n \in \mathbb{N}$, we have
	$c n^{-1} \pm v = n^{-1} (c \pm n v) \in C$.
	Taking the limit $n \to \infty$, using the closedness of $C$, we conclude 
	$\pm v \in C$, or equivalently, $v \in C \cap (-C) = \ell(C) = \{0\}$, a contradiction.
\end{proof}
	
Remark also that, if $C^{>}_\psi(x^*, \alpha) \neq \emptyset$ and $\psi$ is continuous, then $\intt\,C_\psi(x^*, \alpha) \neq \emptyset$ by Lemma \ref{lem:properties_BP} ($6^\circ$).
In the second example from above, the weak upper semicontinuity of $||\cdot||$ is not satisfied, since $Y$ has infinite dimension (cf. \cite[Rem. 2.2]{GuenKhaTam23b}).

Let us present some examples of Bishop-Phelps type cones.
\begin{example}[Bishop-Phelps type cones] \label{ex:BPtypeCones}
{\color{white}.}
\begin{itemize}
    \item[a)]
    Assume that $Y$ has dimension two or higher. For given $y_i^* \in Y^* \setminus \{0\}$, $i = 1, \ldots, m$, we consider the polyhedral cone $C := \{x \in Y \mid \forall i = 1, \ldots, m:\; y_i^*(x) \geq 0\}$, which is closed and convex but not necessarily pointed (and therefore not necessarily well-based). Consider any $w := (w_1, \ldots, w_m) \in \intt\,\mathbb{R}^m_+$. Using the continuous linear functional $x^*_w := \sum_{i = 1}^m w_i y_i^* \in C^+$ and the continuous seminorm $\psi_w := \sum_{i = 1}^m w_i |y_i^*|$, we can write $C$ as a Bishop-Phelps type cone (depending on $w$):  
    \begin{align*}
        C = \left\{x \in Y \mid \sum_{i = 1}^m w_i y_i^*(x) \geq \sum_{i = 1}^m w_i |y_i^*(x)|\right\}
        = \{x \in Y \mid x_w^*(x) \geq \psi_w(x)\} 
        = C_{\psi_w}(x^*_w, 1).
    \end{align*}
    Note, if the polyhedral cone $C$ is pointed (i.e., $\{x \in Y \mid \forall i = 1, \ldots, m:\; y_i^*(x) = 0\} = \{0\}$; note that this can only be the case in finite dimension), then $\psi_w$ is a continuous norm and so $C = C_{\psi_w}(x^*, 1)$ is a Bishop-Phelps cone; otherwise, if $C$ is not pointed (e.g., if $Y$ has infinite dimension), then $C$ is not representable as a Bishop-Phelps cone $C_{||\cdot||}(y^*, \alpha)$ with $(y^*, \alpha) \in Y^* \times \Rpos$ and norm $||\cdot||$.
    
    \item[b)]
    Consider the real (Hilbert) sequence space $Y = \ell_2$ endowed with its canonical norm $||\cdot||_2  = \sqrt{\langle \cdot, \cdot\rangle}$, the natural ordering cone in $\ell_2$, namely $C := (\ell_2)_+ := \{x \in \ell_2 \mid \forall n \in \mathbb{N}:\; x_n \geq 0\}$, and the dual cone of $C$, namely $C^+ = (\ell_2)_+$  (i.e., $C$ is self-dual). Then, 
    $C$ is pointed, closed, convex with 
    $C^\# \neq \emptyset$, but $C$ is not solid and not well-based.
    Since $C$ is not well-based, it is known (by Petschke \cite{Petschke1990}) that $C$ is not representable as a Bishop-Phelps cone based on a continuous norm $\psi$ which is equivalent to $||\cdot||_2$. However, $C$ is representable as a Bishop-Phelps cone based on a continuous norm $\psi$ (which is not equivalent to $||\cdot||_2$). Let us show the latter fact: Consider $y = (y_n) \in \ell_2$ with $y_n > 0$ for all $n \in \mathbb{N}$ (e.g. $y_n = 2^{-n}$). Define $x_y^* \in C^\#$ by
    $x \mapsto x_y^*(x) := \langle y, x \rangle$ and $\psi_y: Y \to \mathbb{R}_{+} \cup \{+\infty\}$ by 
    $x \mapsto \psi_y(x) := \sum_{n = 1}^\infty y_n |x_n|$.
    Using $x^*_y$ and $\psi_y$, we can write $C$ as a Bishop-Phelps  cone (depending on $y$):  
    \begin{align*}
        C = (\ell_2)_+ 
        = \left\{x \in \ell_2 \mid \sum_{n = 1}^\infty y_n x_n \geq \sum_{n = 1}^\infty y_n |x_n|\right\} = \{x \in \ell_2 \mid x_y^*(x) \geq \psi_y(x)\} 
         = C_{\psi_y}(x^*_y, 1).
    \end{align*}
    We are going to prove that $\psi_y$ is indeed a continuous norm. 
    By the Cauchy-Schwarz inequality, for any $x \in \ell_2$, we have
    $$
    \psi_y(x) = \sum_{n = 1}^\infty y_n |x_n| \leq \underbrace{||y||_2}_{\in (0, +\infty)} \cdot ||x||_2.
    $$
    This shows that $\psi_y$ is a finite function.    
    Of course, $\psi_y$ is a seminorm, since $\psi_y$ is finite, nonnegative, absolute homogeneous, and subadditive. Note that 
    $\psi_y(x) = \sum_{n = 1}^\infty y_n |x_n| = 0 \iff  \forall n \in \mathbb{N}:\; x_n = 0 \iff x = 0 \in \ell_2.$
    We conclude that $\psi_y$ is a norm. In addition, $\psi_y$ is continuous, since the (semi)norm $\psi_y$ is bounded by a continuous (semi)norm $||\cdot||_2$.
    However, note that $\psi_y$ is not equivalent to the $||\cdot||_2$ norm; otherwise, $C$ would be well-based (this is known to be not the case).  
    
    \item[c)]
    Consider the function $\psi_{\max}: Y \to \mathbb{R}$ defined by \eqref{eq:psimax} in Example \ref{ex:normlike-bases}(b), where $\varphi: Y \to \mathbb{R}$ is a  sublinear function.
    For any $(x^*, \alpha) \in Y^* \times \Rpos$, the Bishop-Phelps type cone with respect to $\psi_{\max}$ is given by:
    \begin{align*}
        C_{\psi_{\max}}(x^*, \alpha) 
        & = \{y \in Y \mid x^*(y) \geq \alpha \max\{\varphi(y), \varphi(-y)\} \}\\
        & = \{y \in Y \mid x^*(y) \geq \alpha \varphi(y) \text{ and } x^*(y) \geq \alpha \varphi(-y)\} \}.
    \end{align*}
    Moreover, note that, for any function $\psi: Y \to \mathbb{R}$, the following statements are equivalent:
    \begin{itemize}
        \item[$1^\circ$] $\psi$ is a (continuous) seminorm.
        \item[$2^\circ$] There is a (continuous) sublinear function $\varphi: Y \to \mathbb{R}$ such that $\psi(y) = \max\{\varphi(y), \varphi(-y)\}$ for all $y \in Y$.
    \end{itemize}
    Hence, the family of Bishop-Phelps type cones (for a fixed pair $(x^*, \alpha) \in Y^* \times \Rpos$) given by
    $\{C_{\psi_{\max}}(x^*, \alpha)\mid \varphi: Y \to \mathbb{R} \text{ is (continuous and) sublinear}\}$ 
    is equal to the family of Bishop-Phelps type cones given by
    $\{C_{\psi}(x^*, \alpha)\mid \psi: Y \to \mathbb{R} \text{ is a (continuous) seminorm}\}$.
    
    \item[d)] Let $\mathcal{F}$ be the family of seminorms that generates the topology of the real locally convex space $Y$. Convex cones from the family of Bishop-Phelps type cones (for a fixed pair $(x^*, \alpha) \in Y^* \times \Rpos$) given by
    $\{C_{\psi}(x^*, \alpha)\mid \psi \in \mathcal{F}\}$
    could be of interest in order to define preorders in $Y$.
    In the case that the given space $Y$ is preordered but not partially ordered by a cone $K$, we get new ideas with considering Bishop-Phelps type type cones (since they are not necessarily pointed).
    
    Note that any real topological-linear space has a norm, which may not be continuous. Of course, normable spaces admit continuous norms. Furthermore, certain non-normable (Hausdorff) locally convex spaces admit continuous norms. For example, the space of smooth functions on a compact set with the Fréchet topology (the topology of uniform convergence on compact sets of all derivatives) is a non-normable Fréchet space, and thus a locally convex space, with a continuous norm, such as the supremum norm. However, there are infinite-dimensional non-normable (Hausdorff) locally convex spaces for which no continuous norms exist. For example, the sequence space $Y = \mathbb{R}^{\mathbb{N}}$ with product topology
    is a real Hausdorff locally convex space that is not normable. In particular, no continuous norm exists. Therefore, there are no Bishop-Phelps cones based on continuous norms. However, there are Bishop-Phelps type cones based on continuous seminorms, such as those from $\mathcal{F}$. This motivates the study of Bishop-Phelps type cones.
    
    We would also like to mention that pointed Bishop-Phelps type cones (based on a continuous seminorm/norm) only exist in spaces where continuous norms exist (due to the fact that $C_\psi(x^*, \alpha)$ is pointed if and only if $\max\{|x^*|, \psi\}$ is a (continuous) norm by Lemma \ref{lem:properties_BP} ($4^\circ$)).
    Consequently, non-pointed Bishop-Phelps type cones are particularly interesting in non-normable spaces, where no continuous norms exist.

\end{itemize}
\end{example}

\bigskip

In what follows, we will briefly introduce different solution concepts in vector optimization. Given two real topological-linear spaces $X$ and $Y$, a nonempty feasible set $\Omega \subseteq X$, and a vector-valued objective function $f: X \rightarrow Y$, we consider the following vector optimization problem:
\begin{equation}
	\label{vector_problem_P}
	\tag{\ensuremath{P}}
	\begin{cases}
		f(x) \to \min \mbox{ with respect to } K\\
		x \in \Omega,
	\end{cases}
\end{equation} 
where we assume throughout the paper that $K \subseteq Y$ is a cone. It is well-known that $K$ induces on $Y$ a reflexive binary relation $\leq_K$ defined, for any two points $y, \overline{y} \in Y$, by $y \leq_K \overline{y} :\Longleftrightarrow y \in \overline{y} - K.$
Note that if $K$ is convex, then $\leq_K$ is a preorder; if $K$ is pointed and convex, then $\leq_K$ is a partial order.
For notational convenience, we consider further the binary relations $\lneq_{K}$ and $<_K$ that are defined, for any two points $y, \overline{y} \in Y$, by
$y \lneq_{K} \overline{y} :\Longleftrightarrow y \in \overline{y} - K \setminus \{0\}, y <_K \overline{y} :\Longleftrightarrow y \in \overline{y} - \intt\, K$, respectively. 
Two well-known types of solutions to the vector problem \eqref{vector_problem_P} can be defined according to the next two definitions (see, e.g., \cite{BotGradWanka}, \cite{GoeRiaTamZal2023}, \cite{Jahn2011}, \cite{Khanetal2015}).

\begin{definition}[Efficiency] 
	\label{def:Eff}
	A point $\overline{x} \in \Omega$ is said to be an \textbf{efficient solution} of \eqref{vector_problem_P} if there is no $x \in \Omega$ such that $f(x) \lneq_{K} f(\overline{x})$. The set of all efficient solutions of \eqref{vector_problem_P} is denoted by
	${\rm Eff}(\Omega \mid f,K)$.	
\end{definition}

\begin{definition}[Weak efficiency]
	\label{def:wEff}
	A point $\overline{x} \in \Omega$ is said to be \textbf{a weakly efficient solution} of \eqref{vector_problem_P} if there is no $x \in \Omega$ such that $f(x) <_K f(\overline{x})$. The set of all weakly efficient solutions of \eqref{vector_problem_P} is denoted by
	${\rm WEff}(\Omega \mid f,K)$.	
\end{definition}

In addition to the efficiency concepts above, we also want to consider a proper efficiency with respect to a cone-valued map. 

\begin{definition}[Proper efficiency with respect to a cone-valued map] 
	\label{def:pEff_A}
	Consider a cone-valued map $\mathbb{A}: \Omega \to 2^Y$. A point $\overline{x} \in \Omega$ is said to be an \textbf{$\mathbb{A}$-properly efficient solution} of \eqref{vector_problem_P} if
	$\overline{x} \in {\rm Eff}(\Omega \mid f,K)$ and 
	$\mathbb{A}(\bar x) \cap (-K) = \{0\}$.
	The set of all $\mathbb{A}$-properly efficient solutions of \eqref{vector_problem_P} is denoted by ${\rm PEff}_{\mathbb{A}}(\Omega \mid f, K)$.	
\end{definition}

\begin{remark} \label{rem:PEFFvsWEffandEffwithA}
	In Definition \ref{def:pEff_A} one can remove the assumption $\overline{x} \in {\rm Eff}(\Omega \mid f,K)$ if $\mathbb{A}(\bar x)$ contains $f[\Omega] - f(\bar x)$. 
	Notice that ${\rm PEff}_{\mathbb{A}^1}(\Omega \mid f, K) = {\rm Eff}(\Omega \mid f,K)$
	and
	${\rm PEff}_{\mathbb{A}^1}(\Omega \mid f, \{0\} \cup \intt\,K) = {\rm WEff}(\Omega \mid f,K)$
	with
	$
	\mathbb{A}^1(\cdot) := \mathbb{R}_+ \cdot (f[\Omega] - f(\cdot)).
	$
	If $K$ is nontrivial, pointed and convex, then $\mathbb{A}^1(\cdot)$ can be replaced by
	$
	\mathbb{A}^2(\cdot) := \mathbb{R}_+ \cdot (f[\Omega] + K - f(\cdot)).
	$
\end{remark}

\begin{remark} \label{rem:proper_efficiency_A}
	For the particular choice of the cone-valued map $\A: \Omega \to 2^Y$, one can see that several well-known concepts of proper efficiency are special cases of Definition \ref{def:pEff_A}: Hurwicz proper efficiency (\cite{Hurwicz}, 1958), Borwein proper efficiency (\cite{Borwein1977}, 1977), Benson proper efficiency (\cite{Benson1979}, 1979), extended Borwein proper efficiency (\cite{Borwein1980}, 1980). For more details see, e.g., \cite[Sec. 2.4.3]{BotGradWanka}, \cite[Ch. 3]{Grad2015}, \cite{GuerraggioMolhoZaffaroni}, \cite[Sec. 2.4]{Khanetal2015} and references therein. 
\end{remark}

Besides the proper efficiency concept given in Definition \ref{def:pEff_A}, we would also like to consider a proper efficiency concept in the sense of Henig (\cite{Henig2}, 1982).
To formulate this concept in the next definition, consider the following family of dilating cones:
$$
\mathcal{D}(K) := \{C \subseteq Y \mid C \text{ is a nontrivial,  convex cone with } K\setminus \{0\} \subseteq \intt\, C\}.
$$

\begin{definition}[Proper efficiency in the sense of Henig]
	\label{def:pEff_Henig}
	A point $\bar x \in \Omega$ is said to be a \textbf{Henig properly efficient solution} of \eqref{vector_problem_P} if there is $C \in \mathcal{D}(K)$ such that $\bar x \in \Eff(\Omega \mid f,C)$. The set of all Henig properly efficient solutions of \eqref{vector_problem_P} is denoted by	$\PEff_{He}(\Omega \mid f,K)$.	
\end{definition}

\section{Scalarization in Vector Optimization} \label{sec:scalarization}

Under the above assumptions, consider the following scalarization of the vector optimization problem \eqref{vector_problem_P}:
\begin{equation}
	\label{scalarized_vector_problem}
	\tag{\ensuremath{P_{\varphi}}}
	\begin{cases}
		(\varphi \circ f)(x) \to \min \\
		x \in \Omega,
	\end{cases}
\end{equation} 
where $\varphi: Y \to \mathbb{R}$. Clearly, to get useful relationships between the problems \eqref{vector_problem_P} and \eqref{scalarized_vector_problem} one needs to impose certain properties on $\varphi$. By solving the scalar problem \eqref{scalarized_vector_problem}  (with a specific function $\varphi$), one can also get certain characterizations for the solutions of the original vector problem \eqref{vector_problem_P}. The application of such a strategy is called \textbf{scalarization method}, while the function $\varphi$ is called \textbf{scalarizing function} (see e.g. \cite{BotGradWanka}, \cite{GoeRiaTamZal2023}, \cite{Jahn2011}, \cite{Khanetal2015}).

In this paper, we would like to analyze the relationships between solutions of \eqref{scalarized_vector_problem} and the (weakly, properly) efficient solutions of \eqref{vector_problem_P} for the case that $\varphi$ satisfies certain monotonicity properties. For doing this, we recall monotonicity concepts for the function $\varphi$ (cf. Jahn \cite[Def. 5.1]{Jahn2011}). 

\begin{definition}[$K$- increasing functional]
	A scalarizing function $\varphi: Y \to \mathbb{R}$ is said to be
	\begin{itemize}
		\item \textbf{$K$-increasing} if 
		$\forall\, y, \overline{y} \in Y, \,y \leq_K \overline{y}:\; \varphi(y) \leq \varphi(\overline{y}).$
		\item\textbf{strictly $K$-increasing} if  
		$\forall\, y, \overline{y} \in Y, \,y <_K \overline{y}:\; \varphi(y) < \varphi(\overline{y}).$
		\item\textbf{strongly $K$-increasing} if  
		$\forall\, y, \overline{y} \in Y, \,y \lneq_K \overline{y}:\; \varphi(y) < \varphi(\overline{y}).$
	\end{itemize}
\end{definition}

As an immediate consequence, any strongly $K$-increasing function is $K$-increasing, and if $K \neq Y$, strictly $K$-increasing as well.

The following lemma collects several important relationships between cone-monotone functions (based on different cones) as well as cone-representing functions. The result will be useful to prove some basic scalarization results based on Bishop-Phelps type cone-monotone scalarizing functions in Section \ref{sec:basics_scalarization}. 

\begin{lemma} \label{lem:monotonicity_Krepresenting_fcn}
	Consider a function $\varphi: Y \to \mathbb{R}$ and two cones $K, C$ in $Y$. The following assertions hold:
	\begin{itemize}
		\item[$1^\circ$] If $\varphi$ is (strictly) $-K$ representing and subadditive, then $\varphi$ is (strictly) $K$-increasing.
		\item[$2^\circ$] If  $\intt K \subseteq \intt C$ and $\varphi$ is strictly $C$-increasing, then $\varphi$ is strictly $K$-increasing.
		\item[$3^\circ$] If $K \subseteq C$ and $\varphi$ is (strongly) $C$-increasing, then $\varphi$ is (strongly) $K$-increasing.    
	\end{itemize}
\end{lemma}

\begin{proof}
	$1^\circ$ We only prove the strict case, since the other case is similar. The case $\intt\, K = \emptyset$ is clear. Now assume that $\intt\, K \neq \emptyset$.
	Consider $y^1, y^2 \in Y$ with $y^2 <_K y^1$, i.e., $y^2 - y^1 \in -\intt\, K$. Since $\varphi$ is strictly $-K$ representing, this is equivalent to 
	$\varphi(y^2 - y^1) < 0$. Thus, we derive $\varphi(y^2) = \varphi(y^1 + y^2 - y^1) \leq  \varphi(y^1) + \varphi(y^2 - y^1) < \varphi(y^1).$    
	This shows that $\varphi$ is strictly $K$-increasing.
	The proofs of $2^\circ$ and $3^\circ$ are straightforward.
\end{proof}

The next scalarization results involving cone-monotone functions (cf. \cite[Lem. 5.6]{Khazayel2021a}, \cite[Lem. 5.2]{Khazayel2021b}) are important for deriving the result in Theorem \ref{th:argminPhiCircfisPWEff}. 

\begin{lemma} \label{lem:monotone_fcns_wpEff}
	Consider a function $\varphi: Y \to \mathbb{R}$ and a cone $K \subseteq Y$. Then, the following assertions hold:
	\begin{itemize}
		\item[$1^\circ$] If $\varphi$ is strictly $K$-increasing, then 	${\rm argmin}_{x \in \Omega} \; (\varphi \circ f)(x) \subseteq {\rm WEff}(\Omega \mid f,K)$. 
		\item[$2^\circ$] If $\varphi$ is strongly $K$-increasing, then 	${\rm argmin}_{x \in \Omega} \; (\varphi \circ f)(x) \subseteq {\rm Eff}(\Omega \mid f,K)$.
		\item[$3^\circ$] If $\varphi$ is $K$-increasing, and 
		${\rm argmin}_{x \in \Omega}\; (\varphi \circ f)(x) = \{\overline{x}\}$ for some $\overline{x} \in \Omega$, then $\overline{x}\in  {\rm Eff}(\Omega \mid f,K)$.	
		\item[$4^\circ$] If $\varphi$ is strictly $C$-increasing or strongly $C$-increasing for some $C \in \mathcal{D}(K)$, then 	
		${\rm argmin}_{x \in \Omega} \; (\varphi \circ f)(x) \subseteq {\rm PEff}_{He}(\Omega \mid f,K)$.	
	\end{itemize}
\end{lemma}
We omit the proof because it is similar to the proofs for \cite[Lem. 5.6]{Khazayel2021a} and \cite[Lem. 5.2]{Khazayel2021b}, which were established for convex cones in real linear spaces in an algebraic framework.

For a (strictly) $-K$ representing function $\varphi: Y \to \mathbb{R}$ with $\varphi(0) = 0$ we have the following scalarization result which will be used to prove some important characterizations in Propositions \ref{prop:scal_result_Eff_BP} and \ref{prop:scal_result_WEff_BP}.

\begin{lemma}\label{lem:result_scalaization_Jahn1}
	Suppose that $K \subseteq Y$ is a cone. Assume that $\varphi: Y \to \mathbb{R}$ is a function with $\varphi(0) = 0$. Then, for any $\bar x \in \Omega$, we have:
	\begin{itemize}
		\item[$1^\circ$] If $\varphi$ is $-K$ representing, then \\
		$\bar x \in {\rm Eff}(\Omega \mid f,K) \iff
		\{x \in \Omega \mid f(x) = f(\bar x)\} = {\rm argmin}_{x \in \Omega}\; \varphi(f(x) - f(\bar x)).
		$
		\item[$2^\circ$] If $\varphi$ is strictly $-K$ representing, then \\
		$\bar x \in {\rm WEff}(\Omega \mid f,K) \iff \bar x \in {\rm argmin}_{x \in \Omega}\; \varphi(f(x) - f(\bar x)).$
	\end{itemize}
\end{lemma}

\begin{remark}
    Note that Jahn \cite[Th. 4.1]{Jahn2023} stated a similar result as in $1^\circ$ of Lemma \ref{lem:result_scalaization_Jahn1} for the case of a pointed convex cone $K$ in a real linear space $Y$. The proof of Lemma \ref{lem:result_scalaization_Jahn1} is similar to \cite[Th. 4.1]{Jahn2023}.
\end{remark}

Next, we recall two well-known nonlinear scalarization methods in vector optimization which are due to Pascoletti and Serafini \cite{PascolettiSerafini1984}, and Gerstewitz \cite{Gerstewitz1983, Gerstewitz1984}.
For a-priori given parameters $a\in Y$ and $k \in Y \setminus \{0\}$, in \cite{PascolettiSerafini1984} the following scalar optimization problem is used
\begin{equation}
	\label{scalar_problem_Pascoletti_Serafini}
	\begin{cases}
		\lambda \to \min \\
		f(x) \in \lambda k + (a - K),\\
		(x, \lambda) \in \Omega \times \mathbb{R},
	\end{cases}
\end{equation}
while in \cite{Gerstewitz1983,Gerstewitz1984} the scalar optimization problem is given as
\begin{equation}
	\label{scalar_problem_Gerstewitz}
	\begin{cases}
		\varphi^G_{a-K, k}(f(x)) \to \min \\
		x\in \Omega,
	\end{cases}
\end{equation}
where the nonlinear scalarizing function
$\varphi^G_{a-K, k}: Y \to \overline{\mathbb{R}} := \mathbb{R} \cup \{-\infty, +\infty\}$ is defined (for $k \in K \setminus \ell(K)$, and $K$ closed and convex) by
\begin{equation}
	\label{eq:Gerstewitz_fcn}
	\varphi^G_{a-K, k}(y) := \inf\{t \in \mathbb{R}\mid y \in t k + (a - K)\} \quad \text{for all }y \in Y.
\end{equation}
Many facts about the problems \eqref{scalar_problem_Pascoletti_Serafini} and \eqref{scalar_problem_Gerstewitz}, as well as about the nonlinear scalarizing function
$\varphi^G_{a-K, k}$, are known (see, e.g., \cite{Gerstewitz1983, Gerstewitz1984}, \cite{PascolettiSerafini1984}, \cite{GerthWeidner}, \cite{Eichfelder2008, Eichfelder2009}, \cite{Khanetal2015}, \cite{TammerWeidner2020}, \cite{GoeRiaTamZal2023}).  

\begin{lemma} \label{lem:equivalencePSproblemAndGproblem}
	Suppose that $K \subseteq Y$ is a closed, convex cone and take $k \in K \setminus \ell(K)$. Then, the scalar problems \eqref{scalar_problem_Pascoletti_Serafini}  and \eqref{scalar_problem_Gerstewitz} are equivalent (i.e., each
	problem has an optimal solution if and only if the other one has an optimal
	solution, and the optimal solutions $x$ as well as the optimal value of both
	problems coincide).
\end{lemma}

\begin{proof}
    This fact is well-known (see \cite[Th. 5.1.3]{TammerWeidner2020}).
\end{proof}

In the next lemma, we state classical relationships between the vector problem \eqref{vector_problem_P} and the scalar problem \eqref{scalar_problem_Pascoletti_Serafini} (see, e.g., Eichfelder \cite[Th. 2.1 and Cor. 2.9]{Eichfelder2008} for the finite dimensional case).
\begin{lemma} \label{lem:scalarization_Gerstewitz_Pascoletti_Serafini}
	Suppose that $K \subseteq Y$ is a convex cone. Then, we have:
	\begin{itemize}    
		\item[$1^\circ$] If $\bar x \in {\rm Eff}(\Omega \mid f,K)$, then $(\bar x, 0)$ solves the problem \eqref{scalar_problem_Pascoletti_Serafini} for $a = f(\bar x)$ and any $k \in K \setminus \ell(K)$. More precisely, $\{(x, 0) \mid x \in \Omega, f(x) = f(\bar x)\}$ is the set of solutions of  the problem \eqref{scalar_problem_Pascoletti_Serafini}.    
		\item[$2^\circ$] If $\bar x \in {\rm WEff}(\Omega \mid f,K)$, then $(\bar x, 0)$ solves the problem \eqref{scalar_problem_Pascoletti_Serafini} for $a = f(\bar x)$ and any $k \in \intt\,K$.
		\item[$3^\circ$] Consider $(\bar x, \bar \lambda) \in \Omega \times \R$. If $\{(x, \bar \lambda) \mid x \in \Omega, f(x) = f(\bar x)\}$ is the set of solutions of the problem \eqref{scalar_problem_Pascoletti_Serafini} for some $a \in Y$ and $k \in Y \setminus\{0\}$, then $\bar x \in {\rm Eff}(\Omega \mid f,K)$.
		\item[$4^\circ$] If $(\bar x, \bar \lambda) \in \Omega \times \R$ solves the problem \eqref{scalar_problem_Pascoletti_Serafini} for some $a \in Y$ and $k \in Y \setminus\{0\}$, then $\bar x \in {\rm WEff}(\Omega \mid f,K)$ and $f(\bar x) \in \bar\lambda k + (a - \bd\,K)$.
	\end{itemize}
\end{lemma}

The scalar optimization problems of type \eqref{scalar_problem_Pascoletti_Serafini}  and \eqref{scalar_problem_Gerstewitz} involved in the classical scalarization results (see, e.g., Eichfelder \cite[Th. 2.1 and Cor. 2.9]{Eichfelder2008}) have the drawback that the parameter $a$ depends on the solution $\bar x$ (more precisely, we have $a = f(\bar x)$). In order to avoid this disadvantage, Jahn \cite[Th. 4.2]{Jahn2023} stated a scalarization result where the parameter $a$ does not depend on the solution $\bar x$. 

In the following proposition, we present a new scalarization result that contains the result of Jahn \cite[Th. 4.2]{Jahn2023} as a special case (see Remark \ref{rem:Jahn_result} for details).

\begin{proposition}
	\label{prop:result_scalarization_Jahn_2_new}
	Suppose that $K \subseteq Y$ is a convex cone, and take sets $\mathcal{P} \subseteq Y$, $\mathcal{T} \subseteq K\setminus \ell(K)$ and $\mathcal{R} \subseteq \mathbb{R}$ such that $f[{\rm Eff}(\Omega \mid f,K)] \subseteq \mathcal{P} + \mathcal{R} \cdot \mathcal{T}$.  Then, for any $\bar x \in \Omega$, the following assertions are equivalent:
	\begin{itemize}
		\item[$1^\circ$] $\bar x \in {\rm Eff}(\Omega \mid f,K)$. 
		\item[$2^\circ$] There are some $a \in \mathcal{P}$, $k \in \mathcal{T}$ and $s \in \mathcal{R}$ so that $\{(x, s)\mid x \in \Omega, f(x) = f(\bar x)\}$ is the set of solutions of \eqref{scalar_problem_Pascoletti_Serafini}.
	\end{itemize} 
	If, in addition, $K$ is closed, then $1^\circ$ and $2^\circ$ are equivalent to:
	\begin{itemize} 
		\item[$3^\circ$] There are some $a \in \mathcal{P}$, $k \in \mathcal{T}$ and $s \in \mathcal{R}$ so that $\varphi^G_{a - K, k}(f(\bar x)) = s$ and 
		$
		\{x \in \Omega \mid f(x) = f(\bar x)\} = {\rm argmin}_{x \in \Omega}\; \varphi^G_{a - K, k}(f(x)).
		$
	\end{itemize}   
\end{proposition}

\begin{proof} 
	The implication $2^\circ \Longrightarrow 1^\circ$ is a direct consequence of Lemma \ref{lem:scalarization_Gerstewitz_Pascoletti_Serafini} ($3^\circ$).
	
	Now, we show the implication $1^\circ \Longrightarrow 2^\circ$. Let $\bar x \in {\rm Eff}(\Omega \mid f,K)$.
	Take any $k \in \mathcal{T} \,(\subseteq K \setminus \ell(K))$. By Lemma \ref{lem:scalarization_Gerstewitz_Pascoletti_Serafini} ($1^\circ$), the set of solutions of the problem \eqref{scalar_problem_Pascoletti_Serafini} (with $f(\bar x)$ in the role of $a$) is given by $\{(x, 0) \mid x \in \Omega, f(x) = f(\bar x)\}$. Note that, for any $x \in \Omega$, the pair ($x, 0$) solves \eqref{scalar_problem_Pascoletti_Serafini} (with $f(\bar x)$ in the role of $a$) if and only if ($x, t$) solves \eqref{scalar_problem_Pascoletti_Serafini} with $f(\bar x) - t k$ in the role of $a$ for all $t \in \mathbb{R}$. 
	By our assumption $f[{\rm Eff}(\Omega \mid f,K)] \subseteq \mathcal{P} + \mathcal{R} \cdot \mathcal{T}$, there are $\bar a \in \mathcal{P}$, $\bar{k} \in \mathcal{T}$ and $s \in \mathcal{R}$ such that $f(\bar x) = \bar a + s \bar{k}$. Thus, $\{(x, s) \mid x \in \Omega, f(x) = f(\bar x)\}$ is the set of solutions of the problem \eqref{scalar_problem_Pascoletti_Serafini} with $\bar a = f(\bar x) - s \bar{k} \in \mathcal{P}$ and $\bar{k} \in \mathcal{T}$ in the role of $a$ and $k$.
	
	If $K$ is closed, the equivalence of $2^\circ$ and $3^\circ$ follows immediately from Lemma \ref{lem:equivalencePSproblemAndGproblem}.
\end{proof}

\begin{remark} \label{rem:Jahn_result}
	In Proposition \ref{prop:result_scalarization_Jahn_2_new}, one may take 
	$\mathcal{R} \in \{\mathbb{R}, \mathbb{R}_+, -\mathbb{R}_+, \mathbb{P}, -\mathbb{P}, \{1\}, \{0\},$ $\{-1\} \}$.
	If $\varphi$ is a $-K$ representing function, then the problem  \eqref{scalar_problem_Pascoletti_Serafini} can be equivalently stated by replacing the relation $f(x) \in \lambda k + (a - K)$ by $\varphi(f(x) - a - \lambda k) \leq 0$, as done by Jahn \cite{Jahn2023}. 
	From Proposition \ref{prop:result_scalarization_Jahn_2_new}, considering a pointed, convex cone $K \subseteq Y$, a singleton set $\mathcal{P}$,  $\mathcal{T} := K \setminus \{0\}$ and $\mathcal{R} := \{1\}$, we recover the result in Jahn \cite[Th. 4.2]{Jahn2023} (which provides the equivalence of $1^\circ$ and $2^\circ$).   
	In this setting, one needs the condition $f[{\rm Eff}(\Omega \mid f,K)] \subseteq a + K \setminus \{0\}$ for some $a \in Y$,
	which means that $f[{\rm Eff}(\Omega \mid f,K)]$ is strongly $K$-bounded from below (i.e., $a \lneq_K f(x)$ for all $x \in {\rm Eff}(\Omega \mid f,K)$).
	To give just one example, in the case of $Y = \mathbb{R}^n$ and $K$ as the standard ordering cone, one could take the so-called ideal point of \eqref{vector_problem_P}, namely $y^I := (\inf_{x \in \Omega}\, f_1(x), \ldots, \inf_{x \in \Omega}\, f_n(x))$, as the  point $a$ (if $y^I \in \mathbb{R}^n  \setminus f[\Omega]$). 
\end{remark}

For the concept of weak efficiency we have the following new result.

\begin{proposition}
	\label{prop:result_scalarization_Jahn_3_new}
	Suppose that $K \subseteq Y$ is a  convex cone, and take sets $\mathcal{P} \subseteq Y$, $\mathcal{T} \subseteq  \intt\,K$ and $\mathcal{R} \subseteq \mathbb{R}$ such that $f[{\rm WEff}(\Omega \mid f,K)] \subseteq \mathcal{P} + \mathcal{R} \cdot \mathcal{T}$. Then, for any $\bar x \in \Omega$, the following assertions are equivalent:
	\begin{itemize}
		\item[$1^\circ$] $\bar x \in {\rm WEff}(\Omega \mid f,K)$. 
		\item[$2^\circ$] There are some $a \in \mathcal{P}$, $k \in \mathcal{T}$ and $s \in \mathcal{R}$ so that $(\bar x, s)$ is a solution of       \eqref{scalar_problem_Pascoletti_Serafini}.
	\end{itemize} 
	If, in addition, $K$ is closed  and nontrivial, then $1^\circ$ and $2^\circ$ are equivalent to:
	\begin{itemize} 
		\item[$3^\circ$] There are some $a \in \mathcal{P}$, $k \in \mathcal{T}$ and $s \in \mathcal{R}$ so that 
		$
		s = \varphi^G_{a - K, k}(f(\bar x)) = {\rm min}_{x \in \Omega}\; \varphi^G_{a - K, k}(f(x)).
		$
	\end{itemize}   
\end{proposition}

\begin{proof}
	The implication $2^\circ \Longrightarrow 1^\circ$ is an immediate consequence of Lemma \ref{lem:scalarization_Gerstewitz_Pascoletti_Serafini} ($4^\circ$) taking into account that $\intt\,K \subseteq K \setminus \{0\}$, or equivalently $K \neq Y$ (since \eqref{scalar_problem_Pascoletti_Serafini} has a solution),  while the proof of the implication $1^\circ \Longrightarrow 2^\circ$ is similar to the proof of the corresponding implication in Proposition \ref{prop:result_scalarization_Jahn_2_new} by using Lemma \ref{lem:scalarization_Gerstewitz_Pascoletti_Serafini} ($2^\circ$).
	
	If $K$ is closed, the equivalence of $2^\circ$ and $3^\circ$ follows immediately from Lemma \ref{lem:equivalencePSproblemAndGproblem}.
\end{proof}

\begin{remark} Assuming that $K$ has a normlike-base $B_K(\psi)$ (which means that the seminorm $\psi: Y \to \mathbb{R}$ is positive on $K \setminus \{0\}$) and $f[{\rm Eff}(\Omega \mid f,K)] \subseteq a + K \setminus \ell(K)$ (respectively, $f[{\rm WEff}(\Omega \mid f,K)] \subseteq a + \intt\,K$) for some $a \in Y$,
	in Proposition \ref{prop:result_scalarization_Jahn_2_new} (respectively, Proposition \ref{prop:result_scalarization_Jahn_3_new}), one may take $\mathcal{P} := \{a\}$, $\mathcal{T} := B_{\{0\} \cup (K \setminus \ell(K))}(\psi)$ (respectively, $\mathcal{T} := B_{\{0\} \cup \intt\,K}(\psi)$ if $0 \notin \intt\,K$, or equivalently, $K \neq Y$) and $\mathcal{R} := \mathbb{P}$. 
\end{remark}

The following proposition provides possible choices for the sets $\mathcal{P} \subseteq Y$, $\mathcal{T} \subseteq K\setminus \ell(K)$ and $\mathcal{R} \subseteq \mathbb{R}$ to ensure 
$f[D] \subseteq \mathcal{P} + \mathcal{R} \cdot\mathcal{T}$
with $D = {\rm Eff}(\Omega \mid f,K)$, $D = {\rm WEff}(\Omega \mid f,K)$ and $D = \Omega$, respectively, in Propositions \ref{prop:result_scalarization_Jahn_2_new} and \ref{prop:result_scalarization_Jahn_3_new}.

\begin{proposition} \label{prop:Eichfelder}
	Let $K \subseteq Y$ be a nontrivial, convex cone and let $y^* \in Y^* \setminus \{0\}$, $b \in Y$ and $k \in K \setminus \{0\}$ be given. Suppose that $y^*(k) > 0$ (e.g. if $y^* \in K^\#$). Define
	\begin{align*}
		H_{y^*, b}^\sim := \{y \in Y \mid y^*(y) \sim y^*(b)\} \quad \text{for all }\sim \in \{=, <, >\}. 
	\end{align*}
	Then, the following assertions are valid:
	\begin{itemize}
		\item[$1^\circ$] $H_{y^*, b}^> = H_{y^*, b}^= + \mathbb{P} \cdot k$ and $H_{y^*, b}^< = H_{y^*, b}^= - \mathbb{P} \cdot k$.
		\item[$2^\circ$] If $y^* \in K^\#$, then 
		$b + K \setminus \{0\} \subseteq H_{y^*, b}^= + \mathbb{P} \cdot k$
		and
		$b - K \setminus \{0\} \subseteq H_{y^*, b}^= - \mathbb{P} \cdot k$.
		\item[$3^\circ$] If $y^* \in K^+$ and $k \in \intt\,K$, then 
		$b + \intt\,K \subseteq H_{y^*, b}^= + \mathbb{P} \cdot k$
		and
		$b - \intt\,K \subseteq H_{y^*, b}^= - \mathbb{P} \cdot k$.
		\item[$4^\circ$] $Y = \bigcup_{\sim \in \{<,=,>\}}\, H_{y^*, b}^\sim =  H_{y^*, b}^= + \mathbb{R} \cdot k$.
	\end{itemize}
\end{proposition}

\begin{proof} 
The proofs for $1^\circ$ and $2^\circ$ are straightforward, while $3^\circ$ and $4^\circ$ follow easily by using $1^\circ$ and $2^\circ$, taking into account that, for 
a nontrivial, solid, convex cone $K$ and $C := ({\rm int}\,K) \cup \{0\}$  we have ${\rm int}\,K = C \setminus \{0\}$ and $K^+ \setminus \{0\} = \{x^* \in Y^* \mid \forall x \in {\rm int}\,K: x^*(x) > 0\} = C^\#$ (see also \cite[Th. 4.8, Cor. 4.9]{Khazayel2021a}).
\end{proof}

\begin{remark} 
	Let the assumptions of Proposition \ref{prop:Eichfelder} be fulfilled.
	Eichfelder's adaptive scalarization method in multiobjective optimization proposed in \cite{Eichfelder2008, Eichfelder2009} is based on parameterized scalar optimization problems of type \eqref{scalar_problem_Pascoletti_Serafini} and \eqref{scalar_problem_Gerstewitz}, respectively, with different choices of the parameter $a \in H_{y^*, b}^=$, for a-priori fixed $b \in Y$ and direction $k \in K \setminus \{0\}.$ Note that Proposition \ref{prop:result_scalarization_Jahn_2_new}  generalizes a result by Eichfelder \cite[Th. 3.2]{Eichfelder2009} taking into account that $f[{\rm Eff}(\Omega \mid f,K)]  \subseteq Y =  \mathcal{P} + \mathcal{R} \cdot \mathcal{T}$ for $\mathcal{P} := H_{y^*, b}^=$, $\mathcal{T} := \{k\}$ and $\mathcal{R} := \mathbb{R}$ by Proposition \ref{prop:Eichfelder} ($4^\circ$).
\end{remark}

\section{Bishop-Phelps Type Scalarization in Vector Optimization} \label{sec:conic_scalarization}

Assume that the real topological-linear space $Y$ is endowed with a nontrivial cone $K \subseteq Y$. Consider a sublinear functional $\varphi: Y \to \mathbb{R}$ with $\varphi(0) = 0$.
As highlighted in Section \ref{sec:introduction}, for any given point $a \in Y$ and any direction $k \in Y \setminus \{0\}$, the following scalarizations of the vector problem \eqref{vector_problem_P} will play a key role in our studies:
\begin{equation}
	\label{scal_problem_Kas}
	\tag{\ensuremath{P_{\varphi}^{a}}}
	\begin{cases}
		\varphi(f(x) - a) \to \min \\
		x \in \Omega
	\end{cases}
\end{equation}
and
\begin{equation}
	\label{scal_problem_GerPasSer}
	\tag{\ensuremath{P_{\varphi}^{a,k}}}
	\begin{cases}
		\lambda \to \min \\
		\varphi(f(x) - a - \lambda k) \leq 0,\\
		(x, \lambda) \in \Omega \times \mathbb{R}.
	\end{cases}
\end{equation}
When dealing with the scalar problems \eqref{scal_problem_Kas} and \eqref{scal_problem_GerPasSer}, we will impose that $\varphi$ is $-C_\psi(x^*, \alpha)$ representing for some $(x^*, \alpha) \in Y^* \times \R_+$.  
Thus, the Bishop-Phelps type cone $C_\psi(x^*, \alpha)$ plays an important role in our scalarization approach due to the cone-representing property of $\varphi$. This is why our method is called \textbf{Bishop-Phelps type scalarization in vector optimization}.\\

\begin{remark} 
It is known that several scalarizing functions in vector optimization (e.g., in the sense of Gerstewitz \cite{Gerstewitz1983,Gerstewitz1984}; Hiriart-Urruty \cite{HiriartUrruty1979} / Zaffaroni \cite[Prop. 3.2]{Zaffaroni2003}) admit cone-representation properties. Such functions can therefore be used well in our approach. Let us describe two types of scalarizing functions in more detail:
\begin{itemize}
    \item The nonlinear scalarizing function $\varphi^G_{- C_\psi(x^*, \alpha), k}$ (i.e., the function defined in \eqref{eq:Gerstewitz_fcn} with $a = 0$ and with $C_\psi(x^*, \alpha)$ in the role of $K$)  is sublinear and $-C_\psi(x^*, \alpha)$ representing (respectively, strictly $-C_\psi(x^*, \alpha)$ representing) with $\varphi^G_{- C_\psi(x^*, \alpha), k}(0) = 0$, if $C_\psi(x^*, \alpha)$ is nontrivial and closed, and $k \in C_\psi(x^*, \alpha) \setminus \ell(C_\psi(x^*, \alpha)$ (respectively, $k \in \intt\,C_\psi(x^*, \alpha)$).
    \item 
    According to Lemma \ref{lem:properties_BP_fcns} ($2^\circ$, $4^\circ$), the Bishop-Phelps type (seminorm-linear) function $\varphi_{x^*, \alpha}$ introduced   
    in \eqref{eq:seminorm_lin_fcn} is sublinear and $-C_\psi(x^*, \alpha)$ representing with $\varphi_{x^*, \alpha}(0) = 0$, and if $C^{>}_\psi(x^*, \alpha) \neq \emptyset$ and $\intt\,C_\psi(x^*, \alpha) \neq \emptyset$, then
    $\varphi_{x^*, \alpha}$ is strictly $-C_\psi(x^*, \alpha)$ representing.
\end{itemize}
\end{remark}

\subsection{Basic scalarization results} \label{sec:basics_scalarization}

In this section, we apply the results derived in Section \ref{sec:scalarization}, which are based on the monotonicity and cone-representing properties of the scalarizing function. We apply these results to Bishop-Phelps type cones $C_\psi(x^*, \alpha)$, which act as dilating or enlarging cones for the original cone $K$. By doing so, we derive basic results related to the Bishop-Phelps type scalarization method. In Sections \ref{sec:scalarization_WEff}, \ref{sec:scalarization_PEff_A} and \ref{sec:scalarization_PEff_Henig}, we derive additional Bishop-Phelps type scalarization results that heavily rely on the nonlinear cone separation results from Section \ref{sec:cone_separation}. These results are not derived from standard monotonicity arguments alone. 

In the next theorem, we present first relationships between the vector problem \eqref{vector_problem_P} and the scalar problems \eqref{scal_problem_Kas} and \eqref{scal_problem_GerPasSer}. 

\begin{theorem} \label{th:argminPhiCircfisPWEff}
	For any $a \in Y$, the following assertions hold:
	\begin{itemize}
		\item[$1^\circ$] If $(x^*, \alpha) \in K^{a\circ}$, and $\varphi$ is strictly $C_\psi(x^*, \alpha)$-increasing (e.g., if $\varphi$ is strictly $-C_\psi(x^*, \alpha)$ representing), then 
        ${\rm argmin}_{x \in \Omega} \; \varphi(f(x) - a) \subseteq {\rm WEff}(\Omega \mid f,K)$.    
		\item[$2^\circ$] If $(x^*, \alpha) \in K^{a\#}$ and $\varphi$ is strongly $(C_\psi^>(x^*, \alpha) \cup \{0\})$-increasing, then 	${\rm argmin}_{x \in \Omega} \; \varphi(f(x) - a) \subseteq {\rm Eff}(\Omega \mid f,K)$.    
		\item[$3^\circ$] Assume that $\overline{x} \in \Omega$. If $(x^*, \alpha) \in K^{a+}$, $\varphi$ is $C_\psi(x^*, \alpha)$-increasing, and 
		${\rm argmin}_{x \in \Omega}\; \varphi(f(x) - a) = \{\overline{x}\}$, then $\overline{x}\in  {\rm Eff}(\Omega \mid f,K)$.	  
		\item[$4^\circ$] If $(x^*, \alpha) \in K^{a\#}$, $\intt\,C_\psi(x^*, \alpha) \neq \emptyset$,  and $\varphi$ is strictly $C_\psi(x^*, \alpha)$-increasing or strongly $C_\psi(x^*, \alpha)$-increasing, then 	${\rm argmin}_{x \in \Omega} \; \varphi(f(x) - a) \subseteq {\rm PEff}_{He}(\Omega \mid f,K)$.
	\end{itemize}
\end{theorem}
\begin{proof}
	$1^\circ$  For $\intt\,K = \emptyset$ the result is clear. Assume that $\intt\,K \neq \emptyset$. Consider $(x^*, \alpha) \in K^{a\circ}$. Obviously, if $\varphi$ is strictly $C_\psi(x^*, \alpha)$-increasing, then $\varphi(\cdot - a)$ is strictly $C_\psi(x^*, \alpha)$-increasing as well. Because $(x^*, \alpha) \in K^{a\circ}$ we have $\emptyset\neq \intt\,K \subseteq  C_\psi^>(x^*, \alpha) = \intt\,C_\psi(x^*, \alpha)$.
    Finally, Lemma \ref{lem:monotonicity_Krepresenting_fcn} ($2^\circ$) and Lemma \ref{lem:monotone_fcns_wpEff} ($1^\circ$) yield the conclusion.
	
	$2^\circ$ and $3^\circ$ The proof of the assertions is similar to the proof of assertion $1^\circ$ by using Lemma \ref{lem:monotonicity_Krepresenting_fcn} ($3^\circ$) and Lemma \ref{lem:monotone_fcns_wpEff} ($2^\circ, 3^\circ$).
	
	$4^\circ$ Take some $(x^*, \alpha) \in K^{a\#}$. Since $\varphi$ is strictly (respectively, strongly) $C_\psi(x^*, \alpha)$-increasing we have that $\varphi(\cdot - a)$ is strictly (respectively, strongly) $C_\psi(x^*, \alpha)$-increasing as well. Because $(x^*, \alpha) \in K^{a\#}$ we get $K \setminus \{0\} \subseteq C^>_\psi(x^*, \alpha) = \intt\, C_\psi(x^*, \alpha)$. Taking into account that $C_\psi(x^*, \alpha)$ is a nontrivial convex cone, we conclude that $C_\psi(x^*, \alpha) \in \mathcal{D}(K)$. Finally, with the aid of Lemma \ref{lem:monotone_fcns_wpEff} ($4^\circ$) we get the desired inclusion. 
\end{proof}

In particular, we like to point out monotonicity properties of the seminorm-linear function $\varphi_{x^*, \alpha}$ with respect to a Bishop-Phelps type cone $C_\psi(x^*, \alpha)$.

\begin{lemma} \label{lem:monotonicity_normlinear_fcn_BP}
	Consider $(x^*, \alpha) \in Y^* \times \mathbb{R}_+$. Then, the following assertions hold:
	\begin{itemize}
		\item[$1^\circ$] $\varphi_{x^*, \alpha}$ is $C_\psi(x^*, \alpha)$-increasing.
		\item[$2^\circ$] If $C^>_\psi(x^*, \alpha) \neq \emptyset$, then $\varphi_{x^*, \alpha}$ is strictly $C_\psi(x^*, \alpha)$-increasing.
		\item[$3^\circ$] If $K$ is solid and $(x^*, \alpha) \in K^{a\circ}$, then $\varphi_{x^*, \alpha}$ is strictly $C_\psi(x^*, \alpha)$-increasing.	
		\item[$4^\circ$] If $(x^*, \alpha) \in K^{a\#}$, then $\varphi_{x^*, \alpha}$ is strictly $C_\psi(x^*, \alpha)$-increasing.
		\item[$5^\circ$] $\varphi_{x^*, \alpha}$ is strongly $(C_\psi^>(x^*, \alpha) \cup \{0\})$-increasing.
	\end{itemize}
\end{lemma}

\begin{proof}
	$1^\circ$ Noting that $\varphi_{x^*, \alpha}$ is $-C_\psi(x^*, \alpha)$ representing, the assertion follows from Lemma \ref{lem:monotonicity_Krepresenting_fcn} ($1^\circ$).
	
	$2^\circ$ If $\intt\, C_\psi(x^*, \alpha) = \emptyset$, there is nothing to prove. If $\intt\, C_\psi(x^*, \alpha) \neq \emptyset$, using the assumption $C^>_\psi(x^*, \alpha) \neq \emptyset$ we have that $\varphi_{x^*, \alpha}$ is strictly $-C_\psi(x^*, \alpha)$ representing, and the result follows from Lemma \ref{lem:monotonicity_Krepresenting_fcn} ($1^\circ$). 
	
	$3^\circ$ and $4^\circ$ Using the assumptions it follows that $C^>_\psi(x^*, \alpha) \neq \emptyset$, so by assertion $2^\circ$ we get the conclusion.
	
	$5^\circ$ It follows from the definitions. 
\end{proof}

The next lemma, which studies the monotonicity properties of $\varphi_{x^*, \alpha}$ with respect to the original cone $K$, extends a corresponding result derived by Kasimbeyli \cite[Th. 3.5]{Kasimbeyli2010} in a normed setting.

\begin{lemma} \label{lem:monotonicity_normlinear_fcn}
	Consider $(x^*, \alpha) \in Y^* \times \mathbb{R}_+$. The following assertions hold:
	\begin{itemize}
		\item[$1^\circ$] $(x^*, \alpha) \in K^{a+}$ $\iff$ $\varphi_{x^*, \alpha}$ is $K$-increasing.
		\item[$2^\circ$] $(x^*, \alpha) \in K^{a\circ}$ $\iff$ $\varphi_{x^*, \alpha}$ is strictly $K$-increasing. 
		\item[$3^\circ$] $(x^*, \alpha) \in K^{a\#}$ $\iff$ $\varphi_{x^*, \alpha}$ is strongly $K$-increasing.	
	\end{itemize}
\end{lemma}

\begin{proof}
	$1^\circ$ By Lemma \ref{lem:monotonicity_normlinear_fcn_BP} ($1^\circ$), the function $\varphi_{x^*, \alpha}$ is $C_\psi(x^*, \alpha)$-increasing. If $(x^*, \alpha) \in K^{a+}$, then $K \subseteq C_\psi(x^*, \alpha)$, hence $\varphi_{x^*, \alpha}$ is $K$-increasing by Lemma \ref{lem:monotonicity_Krepresenting_fcn} ($3^\circ$). 
	Conversely, assuming $\varphi_{x^*, \alpha}$ is $K$-increasing, for $k \in K$ we have $-k \leq_K 0$, hence $\varphi_{x^*, \alpha}(-k) \leq \varphi_{x^*, \alpha}(0) = 0$, which means that $x^*(k) - \alpha \psi(k) \geq 0$. Noting that $x^* \in K^+$, we conclude $(x^*, \alpha) \in K^{a+}$.

	$2^\circ$ If $\intt\,K = \emptyset$, we have $K^{a\circ} = Y^* \times \mathbb{R}_+$ and $\varphi_{x^*, \alpha}$ is strictly $K$-increasing, hence the equivalence is valid. Assume that $\intt\,K \neq \emptyset$.
	If $(x^*, \alpha) \in K^{a\circ}$, then $\intt\,K \subseteq C^>_\psi(x^*, \alpha) = \intt\, C_\psi(x^*, \alpha)$, and so $\varphi_{x^*, \alpha}$ is strictly $K$-increasing by Lemmata \ref{lem:monotonicity_normlinear_fcn_BP} ($2^\circ$) and \ref{lem:monotonicity_Krepresenting_fcn} ($2^\circ$). The proof of the reverse implication is similar to $1^\circ$.
	
	$3^\circ$ If $(x^*, \alpha) \in K^{a\#}$, then $K \subseteq C_\psi^>(x^*, \alpha) \cup \{0\}$, so $\varphi_{x^*, \alpha}$ is strongly $K$-increasing by Lemmata \ref{lem:monotonicity_normlinear_fcn_BP} ($5^\circ$) and \ref{lem:monotonicity_Krepresenting_fcn} ($3^\circ$). Also here the proof of the reverse implication is similar to $1^\circ$.  
\end{proof}

Let us shift our focus to relationships between the parameterized scalar problem \eqref{scal_problem_GerPasSer} and the vector problem \eqref{vector_problem_P} involving a Bishop-Phelps type cone. 

\begin{lemma} \label{lem:scalarization_Gerstewitz_Pascoletti_Serafini_BPCones}
	Consider $(x^*, \alpha) \in Y^* \times \mathbb{R}_+$ and assume that $\varphi$ is $-C_\psi(x^*, \alpha)$ representing. Then, the following assertions are valid:
	\begin{itemize}
		\item[$1^\circ$] If $\bar x \in {\rm Eff}(\Omega \mid f, C_\psi(x^*, \alpha))$, then $(\bar x, 0)$ solves the problem \eqref{scal_problem_GerPasSer} for $a = f(\bar x)$ and any $k \in C_\psi(x^*, \alpha) \setminus \ell(C_\psi(x^*, \alpha))$. More precisely, $\{(x, 0)\mid x \in \Omega, f(x) = f(\bar x)\}$ is the set of solutions of \eqref{scal_problem_GerPasSer}.
		\item[$2^\circ$] If $\bar x \in {\rm WEff}(\Omega \mid f, C_\psi(x^*, \alpha))$, then $(\bar x, 0)$ solves the problem \eqref{scal_problem_GerPasSer} for $a = f(\bar x)$ and any $k \in \intt\,C_\psi(x^*, \alpha)$.
		\item[$3^\circ$] Consider $(\bar x, \bar \lambda) \in \Omega \times \R$. If $\{(x, \bar \lambda) \mid x \in \Omega, f(x) = f(\bar x)\}$ is the set of solutions of the problem \eqref{scal_problem_GerPasSer} for some $a \in Y$ and $k \in Y\setminus \{0\}$, then $\bar x \in {\rm Eff}(\Omega \mid f, C_\psi(x^*, \alpha))$.
		\item[$4^\circ$] If $(\bar x, \bar \lambda)$ solves the problem \eqref{scal_problem_GerPasSer} for some $a \in Y$ and $k \in Y \setminus\{0\}$, then $\bar x \in {\rm WEff}(\Omega \mid f, C_\psi(x^*, \alpha))$.
	\end{itemize}
\end{lemma}

\begin{proof} Follows immediately from Lemma \ref{lem:scalarization_Gerstewitz_Pascoletti_Serafini} taking into account that $\varphi$ is $-C_\psi(x^*, \alpha)$ representing.
\end{proof}

Relationships between the parameterized scalar problem \eqref{scal_problem_GerPasSer} and the vector problem \eqref{vector_problem_P} are given in the next theorem.

\begin{theorem} \label{th:solution_GPS_problem_in_PWEff_new}
	Consider $\bar x \in \Omega$. Take some $(x^*, \alpha) \in Y^* \times \R_+$, $a \in Y, k \in Y \setminus\{0\}$ and assume that $\varphi$ is $-C_\psi(x^*, \alpha)$ representing and that $(\bar x, \bar  \lambda)$ is a solution of the optimization problem \eqref{scal_problem_GerPasSer}. Then:
	\begin{itemize}
		\item[$1^\circ$]
		If $(x^*, \alpha) \in K^{a\#}$ and $C_\psi(x^*, \alpha)$ is solid, then $\bar x \in {\rm PEff}_{He}(\Omega \mid f,K)$.
		\item[$2^\circ$] 
		If $(x^*, \alpha) \in K^{a\circ}$, then $\bar x \in {\rm WEff}(\Omega \mid f,K)$.
	\end{itemize}
\end{theorem}
\begin{proof} 
	$1^\circ$ Remark that $\bar x \in {\rm WEff}(\Omega \mid f, C_\psi(x^*, \alpha))$ in view of Lemma \ref{lem:scalarization_Gerstewitz_Pascoletti_Serafini_BPCones} ($4^\circ$).
	The latter set is contained in ${\rm PEff}_{He}(\Omega \mid f, K)$ (since $C_\psi(x^*, \alpha) \in \mathcal{D}(K)$).
	
	$2^\circ$ For the case $\intt\, K = \emptyset$ the statement is clear. Assume that $\intt\, K \neq \emptyset$.
	Again we have $\bar x \in {\rm WEff}(\Omega \mid f, C_\psi(x^*, \alpha))$. Here the latter set is contained in ${\rm WEff}(\Omega \mid f, K)$ (since $\emptyset \neq \intt\, K \subseteq C^>_\psi(x^*, \alpha) = \intt\,C_\psi(x^*, \alpha)$).
\end{proof}

Let us summarize our findings related to vector problems involving a Bishop-Phelps type cone in the next two propositions.

\begin{proposition}
	\label{prop:scal_result_Eff_BP}
	Let $(x^*, \alpha) \in Y^* \times \mathbb{R}_+$ such that  $C_\psi(x^*, \alpha) \neq \ell(C_\psi(x^*, \alpha))$ and $\varphi$ is $-C_\psi(x^*, \alpha)$ representing.  Then, for any $\bar x \in \Omega$, the following assertions are equivalent:
	\begin{itemize}
		\item[$1^\circ$] $\bar x \in {\rm Eff}(\Omega \mid f, C_\psi(x^*, \alpha))$.
		\item[$2^\circ$]  For any $k \in C_\psi(x^*, \alpha) \setminus \ell(C_\psi(x^*, \alpha))$, the set $\{(x, 0) \mid x \in \Omega, f(x) = f(\bar x)\}$ is the set of solutions of \eqref{scal_problem_GerPasSer} for $a = f(\bar x)$.
		\item[$3^\circ$] $\{x \in \Omega \mid f(x) = f(\bar x)\} = {\rm argmin}_{x \in \Omega} \; \varphi(f(x) - f(\bar x))$.
	\end{itemize}
	If there are sets $\mathcal{P} \subseteq Y$, $\mathcal{T} \subseteq C_\psi(x^*, \alpha)\setminus \ell(C_\psi(x^*, \alpha))$ and $\mathcal{R} \subseteq \mathbb{R}$ such that $f[{\rm Eff}(\Omega \mid f, C_\psi(x^*, \alpha))] \subseteq \mathcal{P} + \mathcal{R} \cdot \mathcal{T}$, then  $1^\circ-3^\circ$ are equivalent to: 
	\begin{itemize}
		\item[$4^\circ$] There are $a \in \mathcal{P}$, $k \in \mathcal{T}$ and $s \in \mathcal{R}$ so that $\{(x, s) \mid x \in \Omega, f(x) = f(\bar x)\}$ is the set of solutions of \eqref{scal_problem_GerPasSer}.
	\end{itemize}
	If further $C_\psi(x^*, \alpha)$ is closed, then $1^\circ-4^\circ$ are equivalent to: 
	\begin{itemize}
		\item[$5^\circ$]  There are $a \in \mathcal{P}$, $k \in \mathcal{T}$ and $s \in \mathcal{R}$ so that $\varphi^G_{a - C_\psi(x^*, \alpha), k}(f(\bar x)) = s$ and 
		$
		\{x \in \Omega \mid f(x) = f(\bar x)\} = {\rm argmin}_{x \in \Omega}\; \varphi^G_{a - C_\psi(x^*, \alpha), k}(f(x)).
		$
	\end{itemize}
\end{proposition}

\begin{proof}
	Note that $C_\psi(x^*, \alpha)$ is a convex cone, $\varphi$ is $-C_\psi(x^*, \alpha)$-representing with $\varphi(0) = 0$ and $C_\psi(x^*, \alpha) \setminus \ell(C_\psi(x^*, \alpha)) \neq \emptyset$.
	
	$1^\circ \iff 2^\circ$ The equivalence follows immediately from Lemma \ref{lem:scalarization_Gerstewitz_Pascoletti_Serafini_BPCones} ($1^\circ$, $3^\circ$).
	
	$1^\circ \iff 3^\circ$ The result follows directly from Lemma \ref{lem:result_scalaization_Jahn1} ($1^\circ$).
	
	$1^\circ \iff 4^\circ \iff 5^\circ$ The result follows from Proposition \ref{prop:result_scalarization_Jahn_2_new} taking into account that \eqref{scal_problem_GerPasSer} is the same problem as \eqref{scalar_problem_Pascoletti_Serafini} with $C_\psi(x^*, \alpha)$ in the role of $K$.	
\end{proof}

\begin{proposition} \label{prop:scal_result_WEff_BP}
	Consider $(x^*, \alpha) \in Y^* \times \mathbb{R}_+$ such that $\intt\,C_\psi(x^*, \alpha) \neq \emptyset$ and $\varphi$ is $-C_\psi(x^*, \alpha)$ representing. Then, for any $\bar x \in \Omega$, the following assertions are equivalent:
	\begin{itemize}
		\item[$1^\circ$] $\bar x \in {\rm WEff}(\Omega \mid f, C_\psi(x^*, \alpha))$.
		\item[$2^\circ$]  For any $k \in \intt\,C_\psi(x^*, \alpha)$, the point $(\bar x, 0)$ is a solution of  the problem \eqref{scal_problem_GerPasSer} for $a = f(\bar x)$.
	\end{itemize}
	\noindent If $\varphi$ is strictly $-C_\psi(x^*, \alpha)$ representing (e.g., if $\varphi := \varphi_{x^*, \alpha}$ and $C^>_\psi(x^*, \alpha) \neq \emptyset$), then $1^\circ$ and $2^\circ$ are equivalent to:
	\begin{itemize}
		\item[$3^\circ$]  $\bar x \in {\rm argmin}_{x \in \Omega} \; \varphi(f(x) - f(\bar x))$.
	\end{itemize}
	\noindent If there are sets $\mathcal{P} \subseteq Y$, $\mathcal{T} \subseteq \intt\,C_\psi(x^*, \alpha)$ and $\mathcal{R} \subseteq \mathbb{R}$ such that $f[{\rm WEff}(\Omega \mid f, C_\psi(x^*, \alpha))] \subseteq \mathcal{P} + \mathcal{R}  \cdot \mathcal{T}$, then $1^\circ$ and $2^\circ$ are equivalent to: 
	\begin{itemize}
		\item[$4^\circ$]  There are $a \in \mathcal{P}$, $k \in \mathcal{T}$ and $s \in \mathcal{R}$ so that $(\bar x, s)$ is a solution of \eqref{scal_problem_GerPasSer}.
	\end{itemize}
	\noindent If further $C_\psi(x^*, \alpha)$ is closed and nontrivial, then $1^\circ, 2^\circ$ and $4^\circ$ are equivalent to: 
	\begin{itemize}
		\item[$5^\circ$]  There are $a \in \mathcal{P}$, $k \in \mathcal{T}$ and $s \in \mathcal{R}$ so that 
		$
		s = \varphi^G_{a - C_\psi(x^*, \alpha), k}(f(\bar x)) = {\rm min}_{x \in \Omega}\; \varphi^G_{a - C_\psi(x^*, \alpha), k}(f(x)).
		$
	\end{itemize}
\end{proposition}

\begin{proof}
	First, note that $C_\psi(x^*, \alpha)$ is a solid, convex cone and $\varphi$ is $-C_\psi(x^*, \alpha)$-representing with $\varphi(0) = 0$.
	
	$1^\circ \iff 2^\circ$ The equivalence follows immediately from Lemma \ref{lem:scalarization_Gerstewitz_Pascoletti_Serafini_BPCones} ($2^\circ$, $4^\circ$).
	
	$1^\circ \iff 3^\circ$  The result follows from Lemma \ref{lem:result_scalaization_Jahn1} ($2^\circ$). 
	
	$1^\circ \iff 4^\circ \iff 5^\circ$  The result follows directly from the Proposition \ref{prop:result_scalarization_Jahn_3_new}.
\end{proof}

\begin{remark} 
	The results derived in Theorem \ref{th:argminPhiCircfisPWEff} (combined with Lemma \ref{lem:monotonicity_normlinear_fcn_BP}) and Propositions \ref{prop:scal_result_Eff_BP} and \ref{prop:scal_result_WEff_BP} for the case $\varphi = \varphi_{x^*, \alpha}$ extend some known results in the literature, which are given in a normed space $(Y, \psi)$, $\psi(\cdot) :=||\cdot||$, preordered by a convex cone $K \subseteq Y$:
	\begin{itemize}
		\item Gasimov \cite[Th. 1]{Gasimov} proved a corresponding statement as in Theorem \ref{th:argminPhiCircfisPWEff} ($4^\circ$) for Benson proper efficiency (instead of Henig proper efficiency).
		\item Kasimbeyli \cite[Th. 5.4]{Kasimbeyli2013} stated similar assertions as in Theorem \ref{th:argminPhiCircfisPWEff} for a closed, pointed, convex cone $K$ and Benson proper efficiency (instead of Henig proper efficiency).
		\item Gasimov \cite[Th. 2]{Gasimov} (see also Guo and Zhang \cite[Th. 4.5]{GuoZhang2017}) stated that $\bar x \in {\rm PEff}_\mathbb{A}(\Omega \mid f, C_{||\cdot||}(x^*, \alpha))$ (with $\mathbb{A}$ considered in Benson proper efficiency) implies $\{x \in \Omega \mid f(x) = f(\bar x)\} = {\rm argmin}_{x \in \Omega} \; \varphi_{x^*, \alpha}(f(x) - f(\bar x))$.
	\end{itemize}
\end{remark}

\begin{remark} \label{rem:EffBPConeSubsetWEffBPConeSubsetPEffHenig}
	Assume that $K$ is a nontrivial cone and $(x^*, \alpha) \in K^{a\#}$. Then, $K \setminus \{0\} \subseteq C_\psi^>(x^*, \alpha) \subseteq C_\psi(x^*, \alpha) \setminus \{0\}$, hence
	$\Eff(\Omega \mid f, C_\psi(x^*, \alpha)) \subseteq {\rm Eff}(\Omega \mid f,K)$.  
	If further $K$ has the normlike-base $B_K(\psi)$, and $C_\psi(x^*, \alpha)$ is solid (e.g. if $K$ is solid or $\psi$ is continuous), then $C_\psi(x^*, \alpha)$ is a nontrivial, convex cone with $K \setminus \{0\} \subseteq C_\psi^>(x^*, \alpha) = \intt\, C_\psi(x^*, \alpha)$, which means that $C_\psi(x^*, \alpha) \in \mathcal{D}(K)$, and so  
	$\Eff(\Omega \mid f, C_\psi(x^*, \alpha)) \subseteq \WEff(\Omega \mid f, C_\psi(x^*, \alpha)) \subseteq {\rm PEff}_{He}(\Omega \mid f,K)$.
	
	If $K$ is a nontrivial, solid cone and $(x^*, \alpha) \in K^{a\circ}$, then $\intt\,K \subseteq C_\psi^>(x^*, \alpha) = \intt\, C_\psi(x^*, \alpha)$, hence
	$\WEff(\Omega \mid f, C_\psi(x^*, \alpha)) \subseteq {\rm WEff}(\Omega \mid f,K)$.  
\end{remark}

\subsection{Nonlinear Cone Separation Conditions} \label{sec:cone_separation}

The key tool to prove our main Bishop-Phelps type scalarization results in vector optimization (see Sections \ref{sec:scalarization_WEff}, \ref{sec:scalarization_PEff_A} and \ref{sec:scalarization_PEff_Henig}) will be the separation of two (not necessarily convex) cones by using Bishop-Phelps type separating cones / separating (seminorm-linear) functions. 

Consider a real topological-linear space $Y$ and a seminorm $\psi: Y \to \R$. For a cone $C \subseteq Y$, we use the notations:
$$
S_C(\psi) := \conv(B_C(\psi)) \quad \text{and} \quad S_C^0(\psi) := \conv(\{0\} \cup B_C(\psi)),
$$
where $B_{C}(\psi) = C \cap \uS_{\psi}$ (as defined in Definition \ref{def:normlikebase}).
If the underlying seminorm (norm) $\psi$ is clear in the context, we will simply use $S_C$ and $S_C^0$ instead of $S_C(\psi)$ and $S_C^0(\psi)$.

Throughout this section, we consider a nontrivial cone $K \subseteq Y$ with the normlike-base $B_K = B_K(\psi)$, and a nontrivial cone $A \subseteq Y$ with the normlike-base $B_A = B_A(\psi)$. 
\begin{remark}
    It is important to mention that this normlike-base property
    is essential for utilizing Bishop-Phelps type cone separation results (presented in this Section \ref{sec:cone_separation}) and Bishop-Phelps type scalarization results (presented in Sections \ref{sec:scalarization_WEff}, \ref{sec:scalarization_PEff_A} and \ref{sec:scalarization_PEff_Henig}).
    In fact, this normlike-base property asserts that the cone $K \cup A$ has the normlike-base $B_{K \cup A} = B_{K \cup A}(\psi)$ (respectively, a norm-base if $K \cup A = Y$), i.e., the seminorm $\psi$ is positive on $(K \cup A) \setminus \{0\}$.     
    Our upcoming Remark \ref{rem:normlike-base-property} in Section \ref{sec:scalarization_WEff} provides additional information about the normlike-base property for cones and the construction of adequate seminorms.
\end{remark}

The following conditions will be of interest in our upcoming auxiliary results:
\begin{align}
	&(\cl\,S_{A}^0(\psi)) \cap (\cl\, S_{-K}(\psi))  = \emptyset, \label{eq:sep_condition_1}\\
	&A \cap (\cl(\conv(-K))  = \{0\},  \label{eq:sep_condition_2}\\
	&0  \notin \cl\, S_{-K}(\psi).  \label{eq:sep_condition_3}
\end{align}

\begin{remark}
    We would like to point out that a separation condition for two (not necessarily convex) cones, 
    \begin{equation}
    \label{eq:sep_cond_Kasimbeyli}
       (\cl\,S_{{\rm bd}\,A}^0(\psi)) \cap (\cl\, S_{-K}(\psi))  = \emptyset, 
    \end{equation}
    which is a similar condition to the one in \eqref{eq:sep_condition_1} (more precisely, $A$ is replaced by ${\rm bd}\,A$),     
    was initially studied by 
    Kasimbeyli \cite[Def. 4.1, Th. 4.3]{Kasimbeyli2010}, \cite{, Kasimbeyli2013} and more recently by Garc\'{i}a-Casta\~{n}o, Melguizo-Padial and Parzanese \cite{CastanoEtAl2023} in the
    context of cone separation based on Bishop-Phelps separating cones / functions in real normed spaces (with $\psi$ given by the underlying norm of the space). Moreover, in the same framework, 
    G\"unther, Khazayel and Tammer \cite{GuenKhaTam23b} studied the separation condition \eqref{eq:sep_condition_1} and compared their results with those results (from \cite{Kasimbeyli2010} and \cite{CastanoEtAl2023}) obtained for the separation condition \eqref{eq:sep_cond_Kasimbeyli}.
    In \cite{GuenKhaTam23a}, the study of the separation condition \eqref{eq:sep_condition_1} and cone separation  based on Bishop–Phelps type separating cones / functions is conducted within the more general framework of real topological-linear spaces.
\end{remark}

Some relationships between the conditions \eqref{eq:sep_condition_1}, \eqref{eq:sep_condition_2} and \eqref{eq:sep_condition_3} are studied in \cite[Sec. 4 and 5]{GuenKhaTam23b} and \cite[Sec. 5]{GuenKhaTam23a}.
The following proposition, which provides new assumptions for the validity of the implication 
$[\eqref{eq:sep_condition_2} \text{ and } \eqref{eq:sep_condition_3}] \Longrightarrow   \eqref{eq:sep_condition_1},$
will later play an important role in proving a Bishop-Phelps type scalarization result (namely Corollary \ref{cor:main_scalarization_result_PEff_A_BP}) for vector optimization problems in real locally convex spaces.

\begin{proposition}
	\label{prop:separation_conditions_1}
    Assume that $Y$ is a real topological-linear space, $K, A \subseteq Y$ are nontrivial cones with normlike-bases $B_{K}(\psi)$ and $B_{A}(\psi)$. Moreover, suppose that \eqref{eq:sep_condition_2}  and \eqref{eq:sep_condition_3} are valid, and one of the following statements is satisfied:
	\begin{enumerate}
		\item[a)] $A$ is closed and convex.
		\item[b)] There is $(x^*, \alpha) \in Y^* \times \Rpos$ such that 
		\begin{align}
			\cl\, S_{-K}(\psi) & = \{x \in \cl\,\uB_{\psi} \mid x^*(x) \geq \alpha \} \label{eq:sep_condition_9}
		\end{align}
		is valid, and one of the following conditions is satisfied:
		\begin{align}
			&  B_{A}(\psi) \text{ is compact}, \label{ass:com1}\\
			&  S_{A}^0(\psi) \text{ is closed}, \label{ass:com2}\\
			&  A \text{ is closed and } \cl\,B_A(\psi) \text{ is compact}, \label{ass:com3}\\
			&  A \text{ is closed and } \conv(\cl(B_{A}(\psi)) \cup \{0\}) \text{ is closed}. \label{ass:com4}
		\end{align}
	\end{enumerate}
	Then, \eqref{eq:sep_condition_1} is valid.
\end{proposition}

\begin{proof} 
	a) Assuming $A$ is closed and convex, by \eqref{eq:sep_condition_2} we get $\cl\,S_A^0\cap \cl\, S_{-K} \subseteq A \cap \cl(\conv(-K)) = \{0\}$, and so by \eqref{eq:sep_condition_3} we conclude \eqref{eq:sep_condition_1}.
	
	b) Obviously, \eqref{eq:sep_condition_2} and \eqref{eq:sep_condition_3} imply $A \cap \cl\, S_{-K} = \emptyset$. 
	Since there is $(x^*, \alpha) \in Y^* \times \Rpos$ with \eqref{eq:sep_condition_9}, we get that $x^*(x) < \alpha$ for all $x \in   A \cap (\cl\,\uB_{\psi})$. Note that $B_{A} \cup \{0\} \subseteq  A \cap \cl\,\uB_{\psi}$, and if $A$ is closed, then $\cl\,B_{A} \cup \{0\} \subseteq A \cap \cl\,\uB_{\psi}$. Consider cases:
	\par 
	Case 1: \eqref{ass:com1} is valid. By the continuity of $x^*$ and the compactness of $B_{A} \cup \{0\}$, we get some $\beta \in (0, \alpha)$ such that $B_{A} \cup \{0\} \subseteq \{x \in \cl\,\uB_{\psi} \mid x^*(x) \leq \beta \} \subseteq \{x \in \cl\,\uB_{\psi} \mid x^*(x) < \alpha \}$. Hence, $\cl\,S_A^0 \subseteq \{x \in \cl\,\uB_{\psi} \mid x^*(x) \leq \beta \}$.
	\par 
	Case 2: \eqref{ass:com2} is valid. Then, 
	$\cl\,S_A^0 = S_A^0 \subseteq \{x \in \cl\,\uB_{\psi} \mid x^*(x) < \alpha \}$.
	\par
	Case 3: \eqref{ass:com3} is valid. By the continuity of $x^*$ and the compactness of $\cl\,B_{A} \cup \{0\}$ we get some $\beta \in (0, \alpha)$ such that $\cl\,B_{A} \cup \{0\} \subseteq \{x \in \cl\,\uB_{\psi} \mid x^*(x) \leq \beta \} \subseteq \{x \in \cl\,\uB_{\psi} \mid x^*(x) < \alpha \}$. Thus, $\cl\,S_A^0 \subseteq \{x \in \cl\,\uB_{\psi} \mid x^*(x) \leq \beta \}$.
	\par
	Case 4: \eqref{ass:com4} is valid. Then, $\cl\,S_A^0 \subseteq \cl(\conv(\cl(B_{A}) \cup \{0\})) = \conv(\cl(B_{A}) \cup \{0\}) \subseteq \{x \in \cl\,\uB_{\psi} \mid x^*(x) < \alpha \}$.
	\par
	Recalling that \eqref{eq:sep_condition_9} is valid, we conclude \eqref{eq:sep_condition_1}.   
\end{proof}

\begin{remark} \label{rem:examples}
    Note that the equality ``$=$'' in \eqref{eq:sep_condition_9} of Proposition \ref{prop:separation_conditions_1} is essential and can not be replaced by ``$\subseteq$''. On the left-hand side of Figure \ref{fig:conditions_BPtypecone}, conditions \eqref{eq:sep_condition_2}, \eqref{eq:sep_condition_3}, \eqref{eq:sep_condition_9}, \eqref{ass:com1} are valid and the expected conclusion given by \eqref{eq:sep_condition_1} holds.
    On the right-hand side of the Figure \ref{fig:conditions_BPtypecone}, the conditions \eqref{eq:sep_condition_2}, \eqref{eq:sep_condition_3}, \eqref{ass:com1} are valid, there is no $(x^*, \alpha) \in Y^* \times \Rpos$ such that condition \eqref{eq:sep_condition_9} is  valid, but there is a pair $(x^*, \alpha) \in Y^* \times \Rpos$ such that
    $$\cl\, S_{-K}(\psi) \subsetneq \{x \in \cl\,\uB_{\psi} \mid x^*(x) \geq \alpha \}$$ holds. In this case, the conclusion given by \eqref{eq:sep_condition_1} fails. 
    This example encourages studying equivalent norms in the contexts of cone separation (see the results in this section) and scalarization (see Section \ref{sec:scalarization_PEff_A}).

    \begin{figure}[h!]
    \resizebox{1\hsize}{!}{
    \includegraphics{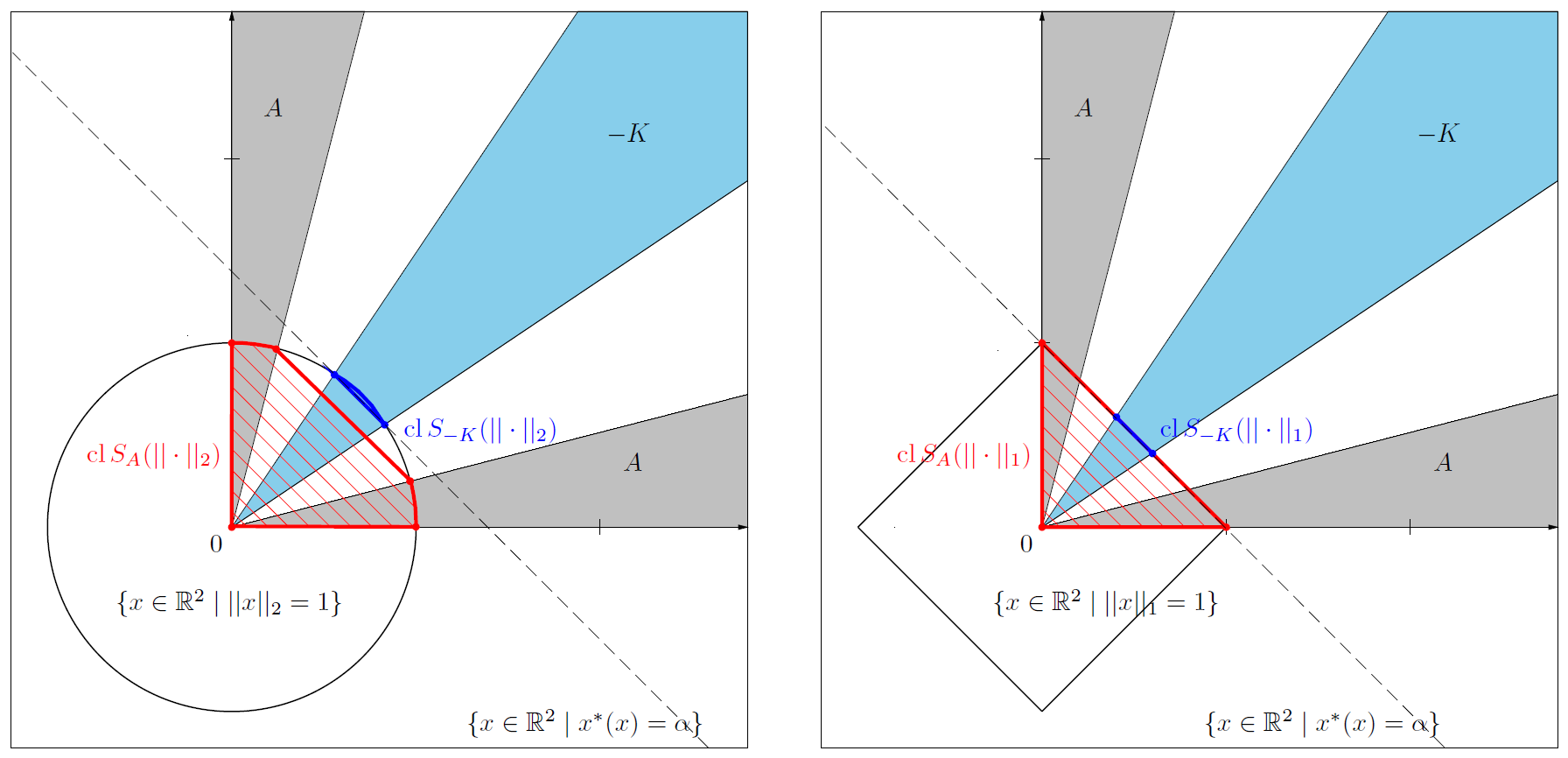}
    }
    \caption{Examples given in Remark \ref{rem:examples}.}
    \label{fig:conditions_BPtypecone}
    \end{figure}
\end{remark}

\begin{remark}
The result in Proposition \ref{prop:separation_conditions_1} (regarding assumption b)) is a generalization of a result by Kasimbeyli \cite[Lem. 3]{Kasimbeyli2013}, where the author studied the Euclidean space $Y = \R^n$, a closed cone $A \subseteq Y$ and a closed convex cone $K \subseteq Y$ with $A \cap (-K) = \{0\}$. Moreover, it is assumed that there is $(x^*, \alpha) \in Y^* \times \Rpos$ such that \eqref{eq:sep_condition_9} is valid (hence $0 \notin \cl\, S_{-K}$ is valid). The conclusion in \cite[Lem. 3]{Kasimbeyli2013} is given as the separation condition \eqref{eq:sep_cond_Kasimbeyli} (i.e., \eqref{eq:sep_condition_1} with $\bd\, A$ in the role of $A$).  
\end{remark}

The condition \eqref{eq:sep_condition_9} involved in b) of Proposition \ref{prop:separation_conditions_1} can be fulfilled by certain Bishop-Phelps type cones, as the following lemma shows.

\begin{lemma} \label{lem:BPcone}  Consider a real topological-linear space $Y$, a continuous seminorm $\psi: Y \to \mathbb{R}$, a nontrivial closed convex cone $K \subseteq Y$ and $(x^*, \alpha) \in Y^* \times \Rpos$. Define $C := C_{\psi}(x^*, \alpha)$. Then, the following assertions are valid:
\begin{itemize}
    \item[$1^{\circ}$] If $K = C$, then 
    $$\cl\, B_{K}(\psi)  = B_{K}(\psi)  =  \{x \in \uS_{\psi}\mid x^*(x) \geq \alpha\},$$
    and if
    \begin{equation}
        \label{eq:BPconeAssumption}
        \{e \in Y \setminus \{0\} \mid \psi(e) = 1, x^*(e) = 0\} \neq \emptyset
    \end{equation}
    (e.g., if $Y$ has dimension two or higher, and $K$ is pointed (or equivalently, $\max\{|x^*|, \psi\}$ is a norm)), 
    then
    \begin{equation}
        \label{eq:BPconeClSC}
        \cl\, S_{K}(\psi)  = S_{K}(\psi) = \{x \in \uB_{\psi}\mid x^*(x) \geq \alpha \}.
    \end{equation}
    
    \item[$2^{\circ}$] If \eqref{eq:BPconeClSC} is valid, and $B_K(\psi)$ is a normlike-base for $K$ (e.g. if $\psi$ is a norm), then $K \subseteq C$ and $S_{K}(\psi)$ is a topological base for $K$, and if $B_C(\psi)$ is a normlike-base for $C$ (e.g. if $\psi$ is a norm), then $K = C$.
\end{itemize}
\end{lemma}

\begin{proof}
    $1^\circ$ The equality $B_{K} = \{x \in \uS_{\psi}\mid x^*(x) \geq \alpha \}$ is obvious. Since $\psi$ is continuous, the sets $C$ and $\uS_{\psi}$ are closed, hence $B_{K}$ is closed. 
	Define  $D := \{x \in \uB_{\psi} \mid x^*(x) \geq \alpha \}$.
	Since $D$ is a convex set with $B_K \subseteq D$, the inclusion $S_K \subseteq D$ is clear. 
	
	Let us show that $S_K \supseteq D$. Take some $d \in D$. We consider two cases:
	\par
	{\rm Case 1}: Assume that $\psi(d) = 1$. Then, $d \in B_K \subseteq S_K$.
	\par
	{\rm Case 2}: Assume that $\psi(d) < 1$.  Put $\gamma := x^*(d)\, (\geq \alpha)$. By our assumption \eqref{eq:BPconeAssumption}, there is $e \in Y \setminus \{0\}$ with $\psi(e) = 1$ and $x^*(e) = 0$. Then, 
	$x^*(d \pm t e) = x^*(d) \pm t x^*(e) = \gamma$ for all $t \geq 0$, 
	and $\psi(d \pm t e) \geq |\,\psi(d) - \psi(t e)\,| = |\,\psi(d) - |t|\,| \to \infty$ for $t \to \infty$.
	Taking into account the continuity of $t \mapsto h_+(t) := \psi(d + t e)$ and $t \mapsto h_-(t) := \psi(d - t e)$, as well as the  conditions $h_+(0) = h_-(0) = \psi(d) < 1$, $\lim_{t \to \infty}\,h_+(t) = \lim_{t \to \infty}\,h_-(t) = \infty$, by the well-known intermediate value theorem there are $t_+, t_- \in \Rpos$ with $h_+(t_+) = h_-(t_-) = 1$. Since $x^*(d + t_+ e) = x^*(d - t_- e) = \gamma \geq \alpha$ we infer $d - t_- e \in B_K$ and $d + t_+ e \in B_K$, hence 
	$d \in \conv(\{d - t_- e,\, d + t_+ e\}) \subseteq S_K.$
	
	Of course, $S_K = D$ and the closedness of $D$ imply $\cl\, S_K = S_K$.

    Note that, in view of Lemma \ref{lem:properties_BP} ($4^\circ$), if $Y$ has dimension two or higher, then $K$ is pointed $\Longleftrightarrow$ $\max\{|x^*|, \psi\}$ is a norm $\Longleftrightarrow$ $\{e \in Y \setminus \{0\} \mid \psi(e) = 0, x^*(e) = 0\} = \emptyset$ $\Longrightarrow$ \eqref{eq:BPconeAssumption} is valid.

    $2^\circ$ Suppose that \eqref{eq:BPconeClSC} is valid. We have
     $B_{K} \subseteq S_{K} = \{x \in \uB_{\psi}\mid x^*(x) \geq \alpha \} \subseteq C_{\psi}(x^*, \alpha) = C$, 
     hence $K = \mathbb{R}_+ \cdot B_{K} \subseteq C$. It is easy to check that $K = \mathbb{R}_+ \cdot B_{K} = \mathbb{R}_+ \cdot S_{K}$, and from \eqref{eq:BPconeClSC} we get $0 \notin \cl\,S_{K}$, hence $S_{K}$ is a topological base for $K$. It remains to show that $K \supseteq C$, if $B_C$ is a normlike-base for $C$. Take some $x \in C$. The assertion is immediate for $x = 0$. Let $x$ belong to $C \setminus \{0\}$. Then there exist $\beta > 0$ and $b \in B_{C}$ such that $x = \beta b$ (since $B_C$ is a normlike-base for $C$), hence
     $
     \beta x^*(b)  = x^*(x) \geq \alpha \psi(x) =  \alpha \beta.
     $
     From this we get that
     $b \in S_{K} \subseteq K$, and so $x = \beta b \in K$. 
     We conclude that $K = C$.  
    \end{proof}

\begin{remark} \label{rem:Representation_Theorem}
	Kasimbeyli and Kasimbeyli \cite[Th. 4.5]{KasKas17} stated a similar result as in Lemma \ref{lem:BPcone} ($1^\circ$) in a reflexive Banach space setting (with $\psi(\cdot) := ||\cdot||$)  and as in Lemma \ref{lem:BPcone} ($2^\circ$) in a normed space setting. Note that the assumption given in \eqref{eq:BPconeAssumption}, which in this setting means that $Y$ has dimension two or higher, is essential. In the one dimensional space ($Y := \mathbb{R}$, $|\cdot|$), the condition  \eqref{eq:BPconeAssumption} fails  and the conclusion of Lemma \ref{lem:BPcone} ($1^\circ$) may fail. To see this, consider $Y := \mathbb{R}$, $x^* := 1$ and $\alpha \in (0,1]$. Then, $C = C(x^*, \alpha) = \mathbb{R}_+$ and $B_{C}(|\cdot|) = \{1\}$ as well as $\cl\, S_{C}(|\cdot|) = \{1\}$ while $\{x \in \uB_{\psi}\mid x^*(x) \geq \alpha \} = [\alpha, 1]$.
\end{remark}

\begin{remark}
    Consider the initial assumptions of Lemma \ref{lem:BPcone}. Suppose that $B_K(\psi)$ is a normlike-base for $K$. The validity of \eqref{eq:sep_condition_9} for some $(x^*, \alpha) \in Y^* \times \Rpos$ implies that $K$ is pointed. If $B_C(\psi)$ is a normlike-base for $C = C_{\psi}(x^*, \alpha)$, then $K = C$ by Lemma \ref{lem:BPcone} ($2^\circ$), hence $\max\{|x^*|, \psi\}$ is a continuous norm. Thus, in non-normable spaces where no continuous norms exist, the condition \eqref{eq:sep_condition_9} together with $B_C(\psi)$ is a normlike-base for $C = C_{\psi}(x^*, \alpha)$ (for some $(x^*, \alpha) \in Y^* \times \Rpos$) can never be true.
\end{remark}

To prove some more specific results in the normed space setting in the upcoming Proposition \ref{prop:separation_conditions_2} and Corollary \ref{cor:main_scalarization_result_PEff_A_normed} (see Section \ref{sec:scalarization_PEff_A}), we need  additional lemmata.

\begin{lemma} \label{lem:WeakCompactnessOmega1AndOmega2normed}  
	Consider a real normed space $(Y, ||\cdot||)$, a (weakly) compact set $\Omega_1 \subseteq Y$ with $0 \notin \Omega_1$, and a (weakly) closed, bounded  set $\Omega_2 \subseteq Y$ with $\Omega_2 \subseteq \R_+ \cdot \Omega_1$. Then, $\Omega_2$ is (weakly) compact.
\end{lemma}

\begin{proof} 
    Note that a set $\Omega_2$ in a normed space is weakly bounded if and only if it is (norm) bounded. 
	First, we prove that there exists $t > 0$ such that 
	\begin{equation}
		\label{eq:Omega2subset[0,t]Omega1}
		\Omega_2 \subseteq [0, t] \cdot \Omega_1.
	\end{equation}
	On the contrary, assume that there are sequences $(x_n^1)$, $(x_n^2)$ and $(t_n)$ with $x_n^1 \in \Omega_1$, $x_n^2 \in \Omega_2$, $t_n \in \mathbb{P}$ such that $x_n^2 = t_n \cdot x_n^1$ with $t_n \to \infty$ for $n \to \infty$. 
	By the boundedness of $\Omega_2$ there is $\lambda > 0$ such that $||x|| \leq \lambda$ for all $x \in \Omega_2$.
	Consequently, we have
	$||x_n^1|| = ||t_n^{-1}x_n^2|| = t_n^{-1} ||x_n^2|| \leq t_n^{-1} \lambda \to 0$ for $n \to \infty$.
	This means that $0 \in \cl\,\Omega_1 \subseteq \cl_w\,\Omega_1$. Noting that the compact (respectively, weakly compact) set $\Omega_1$ is closed (respectively, weakly closed), we get a contradiction.
	
	Define the function $h: \R \times Y \to Y$ by $h(s, \omega) := s \cdot \omega$ for all $(s, \omega) \in \R \times Y$. Since $[0,t] \times \Omega_1$ is (weakly) compact in $\R \times Y$, and the function $h$ is (weakly) continuous by the topological vector space properties, we get that $[0, t] \cdot \Omega_1$ is (weakly) compact. From \eqref{eq:Omega2subset[0,t]Omega1} we derive that the (weakly) closed subset $\Omega_2$ of $[0, t] \cdot \Omega_1$ is (weakly) compact as well.
\end{proof}

\begin{lemma} \label{lem:WeakCompact}
	Consider a real normed space $(Y, ||\cdot||)$ and a nontrivial cone $C \subseteq Y$. Then, the following assertions are valid:
	\begin{itemize}
		\item[$1^\circ$] If $B_C(||\cdot||)$ is compact, then $B_C(||\cdot||')$ is compact for all $||\cdot||' \sim ||\cdot||$.
		\item[$2^\circ$] If $\cl_w\,B_C(||\cdot||)$ is weakly compact with $0 \notin \cl_w\,B_C(||\cdot||)$, then $\cl_w\,B_C(||\cdot||')$ is weakly compact for all $||\cdot||' \sim ||\cdot||$.
		\item[$3^\circ$] If $\cl\,S_C(||\cdot||)$ is (weakly) compact with $0 \notin \cl\,S_C(||\cdot||)$, then $\cl\,S_C(||\cdot||')$ is (weakly) compact for all $||\cdot||' \sim ||\cdot||$.
		\item[$4^\circ$] If $Y$ is reflexive (or equivalently, $\uB_{||\cdot||}$ is weakly compact), then $\cl_w\,B_C(||\cdot||')$ and $\cl\,S_C(||\cdot||')$ are weakly compact for all $||\cdot||' \sim ||\cdot||$.
	\end{itemize}
\end{lemma}

\begin{proof} Consider any $||\cdot||' \sim ||\cdot||$.
	$1^\circ$  Obviously, $B_C(||\cdot||') \subseteq C = \R_+\cdot B_C(||\cdot||)$. Since $||\cdot||' \sim ||\cdot||$, there is $s > 0$ such that $B_C(||\cdot||') \subseteq \uB_{||\cdot||'} \subseteq s \cdot \uB_{||\cdot||}$, which shows that $B_C(||\cdot||')$ is bounded. Moreover, according to \cite[Lemma 2.1.43 (iii)]{GoeRiaTamZal2023}, for the bounded and  closed norm-base $B_C(||\cdot||)$ with $0 \notin B_C(||\cdot||)$ we have $C = \mathbb{R}_+ \cdot B_C(||\cdot||)$ is closed, hence $B_C(||\cdot||')$ is closed as well.
	Then, the compactness of $B_C(||\cdot||')$ follows from Lemma \ref{lem:WeakCompactnessOmega1AndOmega2normed} applied for the compact set $\Omega_1 := B_C(||\cdot||)$ and the closed, bounded set $\Omega_2 := B_C(||\cdot||')$. Note that $0 \notin B_C(||\cdot||) = \Omega_1$.
	
	$2^\circ$  Since $C = \R_+\cdot B_C(||\cdot||) \subseteq \R_+\cdot \cl_w\,B_C(||\cdot||) \subseteq  \cl_w\, C$, we have 
	$\cl_w\, C = \cl_w(\R_+\cdot \cl_w\,B_C(||\cdot||))$. Taking into account that $\cl_w\,B_C(||\cdot||)$ is a weakly closed, bounded set with $0 \notin \cl_w\,B_C(||\cdot||)$, applying \cite[Lemma 2.1.43 (iii)]{GoeRiaTamZal2023} we conclude that $\R_+\cdot \cl_w\,B_C(||\cdot||)$ is weakly closed (remark that the weak topology of the normed space $Y$ is Hausdorff).
	Consequently, 
	$\cl_w\,B_C(||\cdot||') \subseteq \cl_w\, C = \R_+\cdot \cl_w\,B_C(||\cdot||)$.  Since $||\cdot||' \sim ||\cdot||$, there is $s > 0$ such that $\cl_w\,B_C(||\cdot||') \subseteq \uB_{||\cdot||'} \subseteq s \cdot \uB_{||\cdot||}$, which shows that $\cl_w\,B_C(||\cdot||')$ is bounded. Then, the weak compactness of $\cl_w\,B_C(||\cdot||')$ follows from Lemma \ref{lem:WeakCompactnessOmega1AndOmega2normed} applied for the weakly compact set $\Omega_1 := \cl_w\,B_C(||\cdot||)$ and the weakly closed, bounded set $\Omega_2 := \cl_w\,B_C(||\cdot||')$.
	
	$3^\circ$ We have
	$\cl\,S_C(||\cdot||') \subseteq \cl(\conv\,C) = \R_+\cdot \cl\,S_C(||\cdot||)$, where the latter equality is discussed in \cite[Rem. 2.5]{GuenKhaTam23b}. Since $||\cdot||' \sim ||\cdot||$, there is $s > 0$ such that $\cl\,S_C(||\cdot||') \subseteq \uB_{||\cdot||'} \subseteq s \cdot \uB_{||\cdot||}$, hence the set $\cl\,S_C(||\cdot||')$ is bounded.
	Finally, the (weak) compactness of $\cl\,S_C(||\cdot||')$ follows from Lemma \ref{lem:WeakCompactnessOmega1AndOmega2normed} applied for the (weakly) compact set $\Omega_1 := \cl\,S_C(||\cdot||)$ and the (weakly) closed, bounded set $\Omega_2 := \cl\,S_C(||\cdot||')$.
	
	$4^\circ$ Since $||\cdot||' \sim ||\cdot||$, there is $s > 0$ such that $\cl_w\,B_C(||\cdot||') \subseteq \cl\,S_C(||\cdot||') \subseteq \uB_{||\cdot||'} \subseteq s \cdot \uB_{||\cdot||}$, hence both weakly closed sets $\cl_w\,B_C(||\cdot||')$ and $\cl\,S_C(||\cdot||')$ are weakly compact if $\uB_{||\cdot||}$ is weakly compact.
\end{proof}

To later prove a Bishop-Phelps type scalarization result (Corollary \ref{cor:main_scalarization_result_PEff_A_normed}) for vector optimization problems in real normed spaces, we need the following proposition.

\begin{proposition}
	\label{prop:separation_conditions_2}
	Consider a real normed space $(Y, ||\cdot||)$ with dimension two or higher, a nontrivial cone $K \subseteq Y$ with the norm-base $B_K(||\cdot||)$, a nontrivial cone $A \subseteq Y$ with the norm-base $B_A(||\cdot||)$. Suppose that \eqref{eq:sep_condition_2} and \eqref{eq:sep_condition_3} with $\psi(\cdot) := ||\cdot||$ are valid. Moreover, assume that one of the following statements is valid:
	\begin{enumerate}
		\item[a)] $B_A(||\cdot||)$ is compact.
		\item[b)] $A$ is weakly closed, and either $Y$ is reflexive or $\cl_w\,B_A(||\cdot||)$ is weakly compact with $0 \notin \cl_w\,B_A(||\cdot||)$.
	\end{enumerate}
	Then, there is $||\cdot||' \sim ||\cdot||$ such that
	\eqref{eq:sep_condition_1} with $\psi(\cdot) := ||\cdot||'$ is valid. 
\end{proposition}

\begin{proof} Suppose that \eqref{eq:sep_condition_2} and \eqref{eq:sep_condition_3} (with $\psi(\cdot) := ||\cdot||$) are valid. 
	First, note that the following assertions are known to be equivalent (see \cite[Lem. 1.25 and 1.26]{Eichfelder2014}, \cite[Lem. 3.4]{GuenKhaTam23b} and \cite[Th. 3.2]{Petschke1990}):
	\begin{itemize}
		\item $0  \notin \cl\, S_{-K}(||\cdot||)$\quad (i.e., \eqref{eq:sep_condition_3} with $\psi(\cdot) := ||\cdot||$ is valid);
		\item $0  \notin \cl\, S_{-\cl(\conv\,K)}(||\cdot||)$;
		\item $(-\cl(\conv\,K))^{a+}(||\cdot||) \cap (Y^* \times \mathbb{P}) \neq \emptyset$;
		\item $-\cl(\conv\,K)$ is well-based;
		\item $-\cl(\conv\,K)$ is representable as a Bishop-Phelps cone (cf. Example \ref{ex:Bishop-Phelps_cone}).
	\end{itemize}
	Thus, under assumption \eqref{eq:sep_condition_3} (with $\psi(\cdot) := ||\cdot||$), there is a pair $(x^*, \alpha) \in Y^* \times \Rpos$ and an equivalent norm $||\cdot||' \sim ||\cdot||$ such that $-\cl(\conv\,K) = C_{||\cdot||'}(x^*, \alpha)$. By Lemma \ref{lem:BPcone}, in the real normed space $(Y, ||\cdot||')$ with dimension two or higher we have $\cl\, S_{-\cl(\conv\,K)}(||\cdot||') = \{x \in \uB_{||\cdot||'}\mid x^*(x) \geq \alpha \}$,
	hence \eqref{eq:sep_condition_3} and \eqref{eq:sep_condition_9} with $\psi(\cdot) := ||\cdot||'$ and $\cl(\conv\,K)$ in the role of $K$ are valid.
	Let us finish the proof by assuming a) and b), respectively:
	
	a) In view of Lemma \ref{lem:WeakCompact} ($1^\circ$), $B_A(||\cdot||')$ is compact if $B_A(||\cdot||)$ is compact. Proposition \ref{prop:separation_conditions_1} b) applied for the space $(Y, ||\cdot||')$, $\psi(\cdot) := ||\cdot||'$ and $\cl(\conv\,K)$ in the role of $K$ ensures 
	$(\cl\,S_{A}^0(||\cdot||')) \cap (\cl\, S_{-\cl(\conv\,K)}(||\cdot||')) = \emptyset$, hence
	\eqref{eq:sep_condition_1} with $\psi(\cdot) := ||\cdot||'$ is valid, taking into account that $\cl\, S_{{-\rm cl}(\conv\,K)}(||\cdot||') \supseteq \cl\, S_{-K}(||\cdot||')$. 
	
	b) In view of Lemma \ref{lem:WeakCompact} ($2^\circ$, $4^\circ$), $\cl_w\,B_A(||\cdot||')$ is weakly compact if either $Y$ is reflexive or $\cl_w\,B_A(||\cdot||)$ is weakly compact with $0 \notin \cl_w\,B_A(||\cdot||)$.
	Proposition \ref{prop:separation_conditions_1} b) applied for the space $Y$ endowed with the weak topology, $\psi(\cdot) := ||\cdot||'$ and $\cl(\conv\,K)$ in the role of $K$ ensures \eqref{eq:sep_condition_1} with $\psi(\cdot) := ||\cdot||'$, taking into account that $\cl_w\,S_{A}^0(||\cdot||') = \cl\,S_{A}^0(||\cdot||')$ and $\cl_w\, S_{-\cl_w(\conv\,K)}(||\cdot||') = \cl\, S_{{-\rm cl}(\conv\,K)}(||\cdot||') \supseteq \cl\, S_{-K}(||\cdot||')$.    
\end{proof}

Furthermore, we recall some recent results on cone separation in real topological-linear (locally convex) spaces established in \cite{GuenKhaTam23a}.

\begin{proposition}[{\cite[Th. 5.4]{GuenKhaTam23a}}] \label{prop:weak_cone_separation_if}
	Assume that $Y$ is a real topological-linear space, $K, A \subseteq Y$ are nontrivial cones with normlike-bases $B_{K}(\psi)$ and $B_{A}(\psi)$, and $K$ is solid. Suppose that $S_{-K}(\psi)$ is solid and 
	$
	S_{A}^0(\psi) \cap \intt \,S_{-K}(\psi) = \emptyset.
	$
	Then, there exists $(x^*, \alpha) \in K^{a\circ}(\psi)$ such that
	\begin{align*}
		x^*(a) + \alpha \psi(a) & \geq  0 \geq x^*(k) + \alpha \psi(k)   \quad \mbox{for all } a \in A \mbox{ and }k \in -K,\\
		x^*(a) + \alpha \psi(a) & \geq  0 > x^*(k) + \alpha \psi(k)   \quad \mbox{for all } a \in A \mbox{ and }k \in \intt(-K),
	\end{align*}
	or equivalently,
	\begin{align}
		\label{sep_conditions_weak}
		& A \cap -\intt\, C_\psi(x^*, \alpha) = \emptyset \quad \text{and} \quad  K \subseteq C_\psi(x^*, \alpha)\quad \\
		& \text{and} \quad  \intt\,K \subseteq C_\psi^>(x^*, \alpha) = \intt\, C_\psi(x^*, \alpha). \notag
	\end{align}
\end{proposition}

\begin{proposition}[{\cite[Th. 5.8]{GuenKhaTam23a}}] \label{prop:weak_cone_separation_iff_convex}
	Assume that $Y$ is a real topological-linear space, $K, A \subseteq Y$ are nontrivial, convex cones with normlike-bases $B_{K}(\psi)$ and $B_{A}(\psi)$, and $K$ is solid. 
	Suppose that $S_{-K}(\psi)$ is solid.	
	Then, the following assertions are equivalent:	
	\begin{itemize}
		\item[$1^\circ$] $S_{A}^0(\psi) \cap \intt \,S_{-K}(\psi) = \emptyset$.	
		\item[$2^\circ$] $A \cap \intt(-K) = \emptyset$.
		\item[$3^\circ$] There exists $(x^*, \alpha) \in K^{a\circ}(\psi)$ such that \eqref{sep_conditions_weak} is valid.
	\end{itemize}
\end{proposition}

\begin{proposition}[{\cite[Th. 5.9]{GuenKhaTam23a}}] \label{prop:strict_cone_separation_if}
	Assume that $Y$ is a real locally convex space,  $K, A \subseteq Y$ are nontrivial cones with normlike-bases $B_{K}(\psi)$ and $B_{A}(\psi)$. Suppose that one of the sets ${\rm cl}\, S_{-K}(\psi)$ and ${\rm cl}\,S_{A}^0(\psi)$ is compact, and
	\eqref{eq:sep_condition_1} is valid.
	Then, there exists $(x^*, \alpha) \in K^{a\#}(\psi) \cap (Y^* \times \mathbb{P})$ such that
	\begin{equation*}
		x^*(a) + \alpha \psi(a) >  0 > x^*(k) + \alpha \psi(k)   \quad \mbox{for all } a \in A\setminus \{0\} \mbox{ and }k \in -K \setminus \{0\}, 
	\end{equation*}
	or equivalently,
	\begin{equation}
		A \cap -C_\psi(x^*, \alpha) = \{0\} \quad \text{and} \quad  K \setminus \{0\} \subseteq C_\psi^>(x^*, \alpha) = \cor\, C_\psi(x^*, \alpha). \label{sep_conditions_strict}
	\end{equation}
\end{proposition}

From the general strict cone separation result in Proposition \ref{prop:strict_cone_separation_if}, and auxiliary results in Propositions \ref{prop:separation_conditions_1} and \ref{prop:separation_conditions_2} as well as Lemma \ref{lem:WeakCompact}, we derive some new variants of strict cone separation results involving the classical separation condition \eqref{eq:sep_condition_2}.

\begin{corollary} \label{cor:new_sep_1}
	Assume that $Y$ is a real locally convex space,  $K, A \subseteq Y$ are nontrivial cones with normlike-bases $B_{K}(\psi)$ and $B_{A}(\psi)$. Suppose that one of the sets ${\rm cl}\, S_{-K}(\psi)$ and ${\rm cl}\,S_{A}^0(\psi)$ is compact, and
	\eqref{eq:sep_condition_2} and  \eqref{eq:sep_condition_3} are valid. Moreover, assume that the assumptions in a) or in b) of Proposition \ref{prop:separation_conditions_1} are satisfied. 
	Then, there is $(x^*, \alpha) \in K^{a\#}(\psi) \cap (Y^* \times \mathbb{P})$ such that \eqref{sep_conditions_strict} is valid.
\end{corollary}
\begin{proof}
    The conditions \eqref{eq:sep_condition_2} and  \eqref{eq:sep_condition_3} together with the assumptions in a) or in b) of Proposition \ref{prop:separation_conditions_1} ensure that \eqref{eq:sep_condition_1} is valid (by applying Proposition \ref{prop:separation_conditions_1}). 
    Then, using 
    Proposition \ref{prop:strict_cone_separation_if} 
    there exists $(x^*, \alpha) \in K^{a\#} \cap (Y^* \times \mathbb{P})$ such that \eqref{sep_conditions_strict} is valid.
\end{proof}

\begin{corollary} \label{cor:new_sep_2}
    Consider a real normed space $(Y, ||\cdot||)$ with dimension two or higher, a nontrivial cone $K \subseteq Y$ with the norm-base $B_K(||\cdot||)$, a nontrivial cone $A \subseteq Y$ with the norm-base $B_A(||\cdot||)$. Suppose that 
    ${\rm cl}\, S_{-K}(||\cdot||)$ 
    is weakly compact (e.g., if $Y$ is reflexive), and \eqref{eq:sep_condition_2} and \eqref{eq:sep_condition_3} with $\psi(\cdot) := ||\cdot||$ are valid. Moreover, assume that the assumptions in a) or in b) of Proposition \ref{prop:separation_conditions_2} are satisfied. 
	Then, there are $||\cdot||' \sim ||\cdot||$ and $(x^*, \alpha) \in K^{a\#}(||\cdot||') \cap (Y^* \times \mathbb{P})$ such that \eqref{sep_conditions_strict} with $\psi(\cdot) := ||\cdot||'$ is valid.
\end{corollary}
\begin{proof}
    By Proposition \ref{prop:separation_conditions_2} there is $||\cdot||' \sim ||\cdot||$ such that
	\eqref{eq:sep_condition_1} with $\psi(\cdot) := ||\cdot||'$ is valid. Of course, we have $0 \notin {\rm cl}\, S_{-K}(||\cdot||')$, hence ${\rm cl}\, S_{-K}(||\cdot||')$ is weakly compact by Lemma \ref{lem:WeakCompact} ($3^\circ$). Since $(\cl\,S_{A}^0(||\cdot||')) \cap (\cl\, S_{-K}(||\cdot||'))$ is the same set in $(Y, ||\cdot||)$ and $(Y, w)$ (where $w$ is the weak topology on $Y$), the separation condition \eqref{eq:sep_condition_1} is valid in $(Y, w)$.  Applying Proposition \ref{prop:strict_cone_separation_if} to the real locally convex space $(Y, w)$ and the norm $\psi(\cdot) = ||\cdot||'$, there is $(x^*, \alpha) \in K^{a\#}(||\cdot||') \cap (Y^* \times \mathbb{P})$ such that \eqref{sep_conditions_strict} is valid.
\end{proof}

\subsection{Scalarization results for the concept of weak efficiency} \label{sec:scalarization_WEff}

Throughout the remaining Sections \ref{sec:scalarization_WEff}, \ref{sec:scalarization_PEff_A} and \ref{sec:scalarization_PEff_Henig}, we consider a real topological-linear space $Y$, a seminorm $\psi: Y \to \R$, a nontrivial cone $K \subseteq Y$ with the normlike-base $B_K = B_K(\psi)$ and a nontrivial cone $A \subseteq Y$ with the normlike-base $B_A = B_A(\psi)$ (hence $\psi$ is positive on $(K \cup A) \setminus \{0\}$). The normlike-base property is essential to our Bishop-Phelps type scalarization approach, as it employs key tools from nonlinear cone separation (as derived in Section \ref{sec:cone_separation}). 

In this section, we focus on the solution concept of weak efficiency. More precisely, we will present 
conditions that ensure that a weakly efficient solution of the vector problem \eqref{vector_problem_P} is a solution of the scalar problems \eqref{scal_problem_Kas} and \eqref{scal_problem_GerPasSer} for some parameters $a \in Y$ and $k \in C_\psi(x^*, \alpha) \setminus \ell(C_\psi(x^*, \alpha))$.

\begin{theorem} \label{th:main_scalarization_result_WEff}
	Assume that $Y$ is a  real topological-linear space, $K$ is a nontrivial, solid cone, $\mathbb{A}(\cdot) := \mathbb{R}_+ \cdot (f[\Omega] - f(\cdot))$ or $\mathbb{A}(\cdot) := \mathbb{R}_+ \cdot (f[\Omega] + K - f(\cdot))$, $\bar x \in \Omega$, and $A := \mathbb{A}(\bar x)$ is nontrivial. Moreover, suppose that $B_{K}$ and $B_A$ are normlike-bases for $K$ and $A$, respectively. 
    Suppose $S_{-K}$ is solid and
	\begin{equation}
		\label{eq:intersection SAandS-K}
		S_{A}^0 \cap \intt\, S_{-K} = \emptyset.
	\end{equation}
	Then, there is $(x^*, \alpha) \in K^{a\circ}$ such that $\bar x \in \WEff(\Omega \mid f, C_\psi(x^*, \alpha))$. Moreover, if $\varphi$ is $-C_\psi(x^*, \alpha)$ representing, then the other equivalent assertions given in Proposition \ref{prop:scal_result_WEff_BP} are valid whenever the corresponding assumptions in Proposition \ref{prop:scal_result_WEff_BP} are satisfied.
\end{theorem}

\begin{proof} 
	Under our assumptions, by Proposition \ref{prop:weak_cone_separation_if} there exists $(x^*, \alpha) \in K^{a\circ}$ such that 
	$A \cap -\intt\, C_\psi(x^*, \alpha) = \emptyset$ and $\intt\,K \subseteq C_\psi^>(x^*, \alpha) = \intt\, C_\psi(x^*, \alpha)$.
	Taking into account that $f[\Omega] - f(\bar x) \subseteq A$, we have 
	$(f[\Omega] - f(\bar x)) \cap -\intt\, C_\psi(x^*, \alpha) \subseteq A \cap -\intt\, C_\psi(x^*, \alpha) = \emptyset,$
	hence  $\bar x \in \WEff(\Omega \mid f, C_\psi(x^*, \alpha))$. Since $C^>_\psi(x^*, \alpha) \neq \emptyset$ and $\intt\,C_\psi(x^*, \alpha) \neq \emptyset$, Proposition \ref{prop:scal_result_WEff_BP} yields the rest of the conclusion.
\end{proof}

In the convex case (i.e., $K$ and $\mathbb{A}(\bar x)$ are convex cones), using Proposition \ref{prop:weak_cone_separation_iff_convex} we are able to replace the separation condition \eqref{eq:intersection SAandS-K} by $\bar x \in {\rm WEff}(\Omega \mid f, K)$, as shown in the following result.

\begin{corollary}[convex case] \label{cor:main_scalarization_result_WEff_convex}
	Assume that $Y$ is a real topological-linear space, $K$ is a nontrivial, solid, convex cone, $\mathbb{A}(\cdot) := \mathbb{R}_+ \cdot (f[\Omega] - f(\cdot))$ or $\mathbb{A}(\cdot) := \mathbb{R}_+ \cdot (f[\Omega] + K - f(\cdot))$, $\bar x \in {\rm WEff}(\Omega \mid f, K)$, and $A := \mathbb{A}(\bar x)$ 
    is nontrivial and convex. Moreover, suppose that $B_{K}$ and $B_A$ are normlike-bases for $K$ and $A$, respectively, and $S_{-K}$ is solid.\\
	Then, the conclusion of Theorem \ref{th:main_scalarization_result_WEff} holds.
\end{corollary}

\begin{proof}
Taking into account that $K$ is assumed to be nontrivial, solid and convex, we have
$$
\bar x \in {\rm WEff}(\Omega \mid f, K) \iff \bar x \in {\rm PEff}_\mathbb{A}(\Omega \mid f, (\intt\,K)\cup\{0\}) \iff A \cap (-\intt\,K) = \emptyset.
$$
Proposition \ref{prop:weak_cone_separation_iff_convex} yields
$$S_{A}^0 \cap \intt \,S_{-K} = \emptyset \iff A \cap (-\intt\,K) = \emptyset.$$
Thus, the conclusion follows from Theorem \ref{th:main_scalarization_result_WEff}.
\end{proof}

\begin{remark} \label{rem:opt_cond_weak}
In view of Remark \ref{rem:EffBPConeSubsetWEffBPConeSubsetPEffHenig}, since the point $\bar x \in \WEff(\Omega \mid f, C_\psi(x^*, \alpha))$  in the conclusion of Theorem \ref{th:main_scalarization_result_WEff} 
belongs to ${\rm WEff}(\Omega \mid f,K)$, the result provides sufficient conditions for weak efficiency. Under the assumptions of Corollary \ref{cor:main_scalarization_result_WEff_convex} (including convexity of both cones) and taking into account Remark \ref{rem:EffBPConeSubsetWEffBPConeSubsetPEffHenig}, we derive the equivalence of the following assertions:
	\begin{itemize}
		\item[$1^\circ$] $\bar x \in {\rm WEff}(\Omega \mid f, K)$.
		\item[$2^\circ$] There is $(x^*, \alpha) \in K^{a\circ} \cap (Y^* \times \Rpos)$ such that $\bar x \in \WEff(\Omega \mid f, C_\psi(x^*, \alpha))$ (or any other equivalent assertion given in Proposition \ref{prop:scal_result_WEff_BP} is valid whenever the corresponding assumptions in Proposition \ref{prop:scal_result_WEff_BP} are satisfied).
	\end{itemize}   
    From this perspective, $2^\circ$ is a necessary and sufficient condition for the weak efficiency of $\bar x \in \Omega$.
\end{remark}

\begin{remark} 
	Recall that, for any set $\Omega \subseteq Y,$ the quasi-interior of $\Omega$ is the set ${\rm qi}\, \Omega := \{x \in \Omega \mid \cl(\mathbb{R}_+ \cdot (\Omega - x)) = Y\}$. In order to ensure that the assumption $A$ is nontrivial and convex from Corollary \ref{cor:main_scalarization_result_WEff_convex} is satisfied it is sufficient to impose that $Y$ is a real Hausdorff locally convex space and $f$ is $K$-convexlike (i.e., $f[\Omega] + K$ is convex). 
	Indeed, assume that $\bar x \in {\rm WEff}(\Omega \mid f, K)$. 
	Since $f$ is $K$-convexlike, we infer that $A = \mathbb{R}_+\cdot (f[\Omega] + K - f(\bar x))$ is a convex cone in $Y$.  By \cite[Lem. 3.3 (4)]{Khazayel2021b} we know that $f(\bar x) \notin \intt(f[\Omega] + K)$. Since the interior and the quasi interior of a solid, convex set $f[\Omega] + K$ in a real  Hausdorff locally convex space coincide (see \cite[p. 66-67]{GoeRiaTamZal2023}), we have $A \subseteq {\rm cl}(\mathbb{R}_+ \cdot (f[\Omega] + K - f(\bar x))) \neq Y$. Moreover, since $K \neq \{0\}$ we have $A \neq \{0\}$ as well, hence $A$ is nontrivial.
\end{remark}

\begin{remark} \label{rem:normlike-base-property}
    In this and upcoming sections, we consider a seminorm $\psi: Y \to \mathbb{R}$ which is positive on $K \setminus \{0\}$ and on $\mathbb{A}(\bar x) \setminus \{0\}$ (i.e., $B_K = B_{K}(\psi)$ and $B_A = B_A(\psi)$ are normlike-bases for $K$ and $A = \mathbb{A}(\bar x)$). Note that $\psi$ does not have to be continuous in our approach.  In Example \ref{ex:normlike-bases} (b,c), we discussed ways to construct such seminorms. 
    Of course, changing the $\bar x$ may alter the cone $\mathbb{A}(\bar x)$, so one may require another seminorm $\psi$ (at least if $\psi$ is not a norm). In general, one must be careful when choosing a seminorm. Ideally, one would want a seminorm $\psi$ that is positive on $K\setminus\{0\}$ and on $\mathbb{A}(\bar x)\setminus \{0\}$ for all points $\bar x \in \Omega$ which are relevant in the results (e.g. $\bar x \in {\rm WEff}(\Omega \mid f, K)$ in Corollary \ref{cor:main_scalarization_result_WEff_convex}). This is because the scalarized problems appearing in the conclusions of the results are then defined with respect to the same seminorm $\psi$.

    It is easy to see that $K \subseteq \mathbb{R}_+ \cdot (f[\Omega] + K - f(\bar x)) = \mathbb{A}(\bar x)$. Hence, if $\mathbb{A}(\bar x)$ has the normlike-base $B_{\mathbb{A}(\bar x)}(\psi)$, then $K$ has the normlike-base $B_{K}(\psi)$. In general, it makes sense to search for a (not necessarily continuous) seminorm $\psi$ which is in fact a norm, and so $B_K(\psi)$ is a norm-base for $K$. Then, the normlike-base property for any other cone $A = \mathbb{A}(\bar x)$ (respectively, for any cone from a family of cones, e.g. $\{\mathbb{A}(\bar x) \mid \bar x \in {\rm WEff}(\Omega \mid f, K)\}$) is automatically satisfied. Of course, there are infinite-dimensional non-normable (real Hausdorff locally convex) spaces where no continuous norms exist (see also Example \ref{ex:BPtypeCones} (d)). In this case, the properties that $B_{K}(\psi)$ and $B_{A}(\psi)$ are normlike-bases for $K$ and $A$ (with continuous $\psi$) are more difficult to ensure. 
   
    Using the ideas in Example \ref{ex:normlike-bases} (c) one can also state the following two decomposition approaches (assume that $K \subseteq Y$ is a nontrivial solid convex cone):
    \begin{itemize}
        \item[$\bullet$] Suppose that $f[\Omega]$ is a finite union of nonempty convex sets $Y_i$, $i = 1, \ldots ,l$ (i.e., $f[\Omega] = \bigcup_{i = 1}^l Y_i$). Then, 
    $A := \mathbb{A}(\bar x) = \mathbb{R}_+ \cdot (f[\Omega] + K - f(\bar x)) = \bigcup_{i = 1}^l \mathbb{R}_+ \cdot (Y_i + K - f(\bar x))$. Define $A_i := \mathbb{R}_+ \cdot (Y_i + K - f(\bar x))$ for all $i = 1, \ldots, l$.
    If $K$ is convex, then $Y_i + K$ is convex as well, hence $A$ is a finite union of convex cones (this case is studied in Example \ref{ex:normlike-bases} (c)). If, in addition, $y_i^* \in A^\#_i \neq \emptyset$ for all $i = 1, \ldots, l$ (e.g. ensured by the Krein-Rutman theorem), then 
    one can define a continuous seminorm $\psi$ (which is not necessarily a norm) by 
    \eqref{eq:seminorm_1} 
    or by
    \eqref{eq:seminorm_2}. 
    In both cases, $B_A(\psi)$ is a normlike-base for $A$ (hence $B_K(\psi)$ for $K$ as well).
        \item[$\bullet$] For any $\bar x \in {\rm WEff}(\Omega \mid f, K)$ we know that $\mathbb{A}(\bar x) \cap (-\intt\,K) = \emptyset$. Suppose that the (reverse-convex) nontrivial closed solid cone $\widehat{A} := Y \setminus (-\intt\,K)$ can be decomposed into a finite union of nontrivial pointed closed convex cones, i.e., $\widehat{A} = \bigcup_{i = 1}^l A_i$. Note that $K \cap (-\intt\,K) = \ell(K) \cap (-\intt\,K) = \emptyset$, if $K \neq Y$ (since in this case $\intt\,K \subseteq K \setminus \ell(K)$, or equivalently, $-\intt\,K \subseteq (-K) \setminus \ell(K)$). Then, for any $\bar x \in {\rm WEff}(\Omega \mid f, K)$, we have
    $
    K \cup \mathbb{A}(\bar x) \subseteq \widehat{A} = \bigcup_{i = 1}^l A_i.
    $
    If, in addition, $y_i^* \in A^\#_i \neq \emptyset$ for all $i = 1, \ldots, l$, then one can again define a continuous seminorm $\psi$ by 
    \eqref{eq:seminorm_1} 
    or by
    \eqref{eq:seminorm_2}. 
    In both cases, for any $\bar x \in {\rm WEff}(\Omega \mid f, K)$, the sets $B_K(\psi)$ and $B_{\mathbb{A}(\bar x)}(\psi)$ are normlike-bases for $K$ and $\mathbb{A}(\bar x)$.

    \begin{example}
 Consider $Y := \mathbb{R} \times \mathbb{R}$ (endowed with the Euclidean norm), $K := \mathbb{R}_+ \times \mathbb{R}_+$, and
$\widehat{A} := Y \setminus (-{\rm int}\,K) = (\mathbb{R}_+ \times \mathbb{R}_+) \cup (-\mathbb{R}_+ \times \mathbb{R}_+) \cup (\mathbb{R}_+ \times -\mathbb{R}_+).
$
Hence, $\widehat{A} = A_1 \cup A_2 \cup A_3$ for some nontrivial pointed solid closed convex cones $A_1, A_2, A_3 \subseteq \widehat{A}$.  Take $x_1^* \in A_1^\#$,  $x_2^* \in A_2^\#$ and  $x_3^* \in A_3^\#$. Thus, $\psi := |x_1^*| + |x_2^*| + |x_3^*|$ (respectively, $\psi := \max\{|x_1^*|, |x_2^*|, |x_3^*|\}$) is a continuous seminorm (in fact, it is a norm, since $[x_{1}^*, x_2^*, x_3^*](x) = 0$ if and only if $x = (0,0)$).
Note that $\psi$ is positive on $\widehat{A} \setminus \{0\}$. Moreover, for $\bar x \in {\rm WEff}(\Omega \mid f, K)$ (i.e., $\mathbb{A}(\bar x) \cap (-{\rm int}\,K) = \emptyset$), we have $K \cup \mathbb{A}(\bar x) \subseteq \widehat{A}$. Consequently, $B_{\mathbb{A}(\bar x)}(\psi)$ is a normlike-base (in fact, it is a norm-base) for $\mathbb{A}(\bar x)$, while $B_{K}(\psi)$ is a normlike-base (norm-base) for $K$.
In this example, there is no seminorm on $Y$ with a nontrivial kernel that works, since any line through zero intersects $\widehat{A}$ in at least one nonzero point. However, as the example shows, we can construct a norm with a trivial kernel that works.
\end{example}
    \end{itemize}

Note that the second decomposition approach has the advantage (in comparison to the first one) that the constructed seminorm does not depend on the considered point $\bar x \in {\rm WEff}(\Omega \mid f, K)$. We would also like to mention that in infinite-dimensional spaces such decompositions approaches can be more problematic, since a seminorm $\psi$ constructed as in
\eqref{eq:seminorm_1} 
or 
\eqref{eq:seminorm_2}
are not norms (so they have a nontrivial linear subspace contained in the kernel of $\psi$, and this kernel has indeed infinite dimension, because $[x_{1}^*, \ldots, x_m^*](x) = 0_{\mathbb{R}^m}$ has an infinite-dimensional kernel; note that  $Y$ has infinite dimension but $\mathbb{R}^m$ has finite dimension). 
However, note that not always in infinite-dimensional (non-normable) spaces nontrivial kernels of seminorms have infinite dimension. This topic is also part of future research.
\end{remark}

\begin{remark}
   Note that the Bishop-Phelps type scalarization method in vector optimization for examples with a nonconvex cone $K \subseteq Y$ and a nonconvex set $f[\Omega] + K \subseteq Y$ is generally not expected to work properly. This scalarization method is based on Bishop-Phelps type cones, which are in fact convex cones. These cones typically serve as dilating/enlargement cones for a possibly nonconvex given cone $K$. The Bishop-Phelps type scalarization method aims to compute elements of
    $
    {\rm WEff}(\Omega \mid f, C_{\psi}(x^*, \alpha))
    $
    and ${\rm Eff}(\Omega \mid f, C_{\psi}(x^*, \alpha))$), respectively. Propositions \ref{prop:scal_result_Eff_BP} and \ref{prop:scal_result_WEff_BP} show how these sets are connected to minimizers of scalarized problems based on Bishop-Phelps type cone-representing functions.
\end{remark}

\subsection{Scalarization results for the concept of proper efficiency with respect to a cone-valued map} \label{sec:scalarization_PEff_A}

We are ready to present some further main scalarization results which ensure that an $\mathbb{A}$-properly efficient solution of the vector problem \eqref{vector_problem_P} is a solution of the scalar problems \eqref{scal_problem_Kas} and \eqref{scal_problem_GerPasSer} for some parameters $a \in Y$ and $k \in C_\psi(x^*, \alpha) \setminus \ell(C_\psi(x^*, \alpha))$. Remember that $\mathbb{A}: \Omega \to 2^Y$ is a cone-valued map.

\begin{theorem}\label{th:main_scalarization_result_PEff}
	Assume that $Y$ is a real locally convex space, $K \subseteq Y$ is a nontrivial cone, $\bar x \in \Omega$, and $A := \mathbb{A}(\bar x)$ is a nontrivial cone in $Y$ with $f[\Omega] - f(\bar x) \subseteq A$. Moreover, suppose that $B_{K}$ and $B_A$ are normlike-bases for $K$ and $A$, respectively. Suppose that one of the sets $\cl\, S_{-K}$ and $\cl\,S_{A}^0$ is compact, and
	\begin{equation}
		\label{eq:cl_intersection}
		({\rm cl}\,S_{A}^0) \cap ({\rm cl}\, S_{-K}) = \emptyset.
	\end{equation}
	Then, there is $(x^*, \alpha) \in K^{a\#} \cap (Y^* \times \Rpos)$ such that $\bar x \in \Eff(\Omega \mid f, C_\psi(x^*, \alpha))$. Moreover, if $\varphi$ is $-C_\psi(x^*, \alpha)$ representing, 
    then the other equivalent assertions given in Proposition \ref{prop:scal_result_Eff_BP} are valid whenever the corresponding assumptions in Proposition \ref{prop:scal_result_Eff_BP} are satisfied.
\end{theorem}

\begin{remark}
    If the stronger condition $f[\Omega] + K - f(\bar x) \subseteq A$ (instead of $f[\Omega] - f(\bar x) \subseteq A$) is satisfied, and $B_A(\psi)$ is a normlike-base for $A$, then $B_{K}(\psi)$ is a normlike-base for $K$ (since $K \subseteq A$).
\end{remark}

\begin{proof}
	By Proposition \ref{prop:strict_cone_separation_if} there exists $(x^*, \alpha) \in K^{a\#} \cap (Y^* \times \Rpos)$ such that 
	$A \cap -C_\psi(x^*, \alpha) = \{0\}$ and $K \setminus \{0\} \subseteq C_\psi^>(x^*, \alpha) = \cor\, C_\psi(x^*, \alpha)$.
	Then, taking into account that $f[\Omega] - f(\bar x) \subseteq A$, we have 
	$(f[\Omega] - f(\bar x)) \cap -C_\psi(x^*, \alpha) \subseteq A \cap -C_\psi(x^*, \alpha) = \{0\},$
	hence $\bar x \in \Eff(\Omega \mid f, C_\psi(x^*, \alpha))$. Finally, observing that $C_\psi(x^*, \alpha) \neq \ell(C_\psi(x^*, \alpha))$ by Lemma \ref{lem:properties_BP} ($7^\circ$),  Proposition \ref{prop:scal_result_Eff_BP} yields the rest of the conclusion.
\end{proof}

\begin{remark} \label{rem:EffBPConeIsInPEffHenig}
	In view of Remark \ref{rem:EffBPConeSubsetWEffBPConeSubsetPEffHenig}, if $K$ is solid or $\psi$ is continuous, then the point $\bar x \in \Eff(\Omega \mid f, C_\psi(x^*, \alpha))$ in the conclusion of Theorem \ref{th:main_scalarization_result_PEff} belongs to ${\rm PEff}_{He}(\Omega \mid f,K)$, hence the result provides sufficient conditions for proper efficiency in the sense of Henig.
\end{remark}

Using Proposition \ref{prop:separation_conditions_1} we are able to replace the separation condition \eqref{eq:cl_intersection} 
by $0 \notin \cl\, S_{-K}$ and $\bar x \in {\rm PEff}_{\mathbb{A}}(\Omega \mid f, K)$ as well as some more specific conditions that we formulate below.

\begin{corollary}\label{cor:main_scalarization_result_PEff_A_BP}
	Assume that $Y$ is a real locally convex space, $K \subseteq Y$ is a nontrivial, closed, convex cone with $0 \notin \cl\, S_{-K}$, $\bar x \in {\rm PEff}_{\mathbb{A}}(\Omega \mid f, K)$, and $A := \mathbb{A}(\bar x)$ is a cone in $Y$ with $\{0\} \neq f[\Omega] - f(\bar x) \subseteq A$. Suppose that $B_{K}$ and $B_A$ are normlike-bases for $K$ and $A$, and that one of the sets  ${\rm cl}\, S_{-K}$ and ${\rm cl}\,S_{A}^0$ is compact. Moreover, assume that one of the following statements is valid:
	\begin{enumerate}
		\item[a)] $A$ is closed and convex.
		\item[b)] There is $(x^*, \alpha) \in Y^* \times \Rpos$ such that \eqref{eq:sep_condition_9}
		is valid, and one of the conditions \eqref{ass:com1}--\eqref{ass:com4} holds.
	\end{enumerate}
	Then, the conclusion of Theorem \ref{th:main_scalarization_result_PEff} is valid.
\end{corollary}
\begin{proof}
	Using Proposition \ref{prop:separation_conditions_1} a) and b) respectively, under our standing assumptions a) and b), we get that $A \cap \cl(\conv(-K)) = A \cap (-K) = \{0\}$ and $0 \notin \cl\, S_{-K}$ imply  \eqref{eq:cl_intersection}. Note that the nontriviality of $A$ follows from $\{0\} \neq f[\Omega] - f(\bar x) \subseteq A$, $K$ is nontrivial and $A \cap (-K) = \{0\}$.
	Then, the result follows from Theorem \ref{th:main_scalarization_result_PEff}.
\end{proof}

In the case that the cone $K$ is given by a Bishop-Phelps cone based on a continuous norm (see the upcoming Example \ref{ex:l2spaceBPtypeCone} or Jahn and Ha \cite{HaJahn2021} for examples of spaces with underlying ordering $K$ in the normed setting), one can state the following result:
\begin{corollary}\label{cor:main_scalarization_result_PEff_A_BP_2}
	Assume that $Y$ is a real locally convex space with dimension two or higher, $||\cdot||: Y \to \mathbb{R}$ is a continuous norm, $K := C_{||\cdot||}(x^*, \alpha)$ for some $(x^*, \alpha) \in (Y^* \setminus \{0\}) \times \mathbb{P}$ is nontrivial, $\bar x \in {\rm PEff}_{\mathbb{A}}(\Omega \mid f, K)$, and $A := \mathbb{A}(\bar x)$ is a cone in $Y$ with $\{0\} \neq f[\Omega] - f(\bar x) \subseteq A$. Suppose that one of the sets  ${\rm cl}\, S_{-K}(||\cdot||)$ and ${\rm cl}\,S_{A}^0(||\cdot||)$ is compact. Moreover, assume that either $A$ is closed and convex or one of the conditions \eqref{ass:com1}--\eqref{ass:com4} holds.
	Then, the conclusion of Theorem \ref{th:main_scalarization_result_PEff} is valid.
\end{corollary}
\begin{proof}
    Follows immediately from Corollary \ref{cor:main_scalarization_result_PEff_A_BP}. Note that Lemma \ref{lem:BPcone} ($1^\circ$) ensures the validity of \eqref{eq:sep_condition_9}.
\end{proof}

In the remaining part of the section, we present some more particularized results for the case of a real normed space $Y$ by using Proposition \ref{prop:separation_conditions_2} and Lemma \ref{lem:WeakCompact}.

\begin{corollary}       \label{cor:main_scalarization_result_PEff_A_normed}
	Assume that ($Y, ||\cdot||$) is a real normed space with dimension two or higher, $K \subseteq Y$ is a nontrivial, closed, convex cone with $0 \notin \cl\, S_{-K}(||\cdot||)$, $\bar x \in {\rm PEff}_{\mathbb{A}}(\Omega \mid f, K)$, and $A := \mathbb{A}(\bar x)$ is a cone in $Y$ with $\{0\} \neq f[\Omega] - f(\bar x) \subseteq A$. Moreover, assume that one of the following statements is valid:
	\begin{enumerate}
		\item[a)] $B_A(||\cdot||)$ and ${\rm cl}\, S_{-K}(||\cdot||)$ are compact.
		\item[b)] $A$ is weakly closed, and either $Y$ is reflexive or ${\rm cl}\, S_{-K}(||\cdot||)$ and $\cl_w\,B_A(||\cdot||)$ are weakly compact and $0 \notin \cl_w\,B_A(||\cdot||)$.
	\end{enumerate}
	Then, there is a norm $||\cdot||' \sim ||\cdot||$ 
	and a pair $(x^*, \alpha) \in K^{a\#}(||\cdot||') \cap (Y^* \times \mathbb{P})$ such that $\bar x \in \Eff(\Omega \mid f, C_{||\cdot||'}(x^*, \alpha))$. 
	Moreover, if $\varphi$ is $-C_{||\cdot||'}(x^*, \alpha)$ representing, then the other equivalent assertions given in Proposition \ref{prop:scal_result_Eff_BP}, applied for $||\cdot||'$ in the role of $\psi$, are valid whenever the corresponding assumptions in Proposition \ref{prop:scal_result_Eff_BP} are satisfied.
\end{corollary}

\begin{proof} 
	Since $0 \notin \cl\, S_{-K}(||\cdot||)$ and $A \cap \cl(\conv(-K)) = A \cap (-K) = \{0\}$, using Proposition \ref{prop:separation_conditions_2} a) and b) respectively, under our standing assumptions a) and b), one gets that 
	$({\rm cl}\,S_{A}^0(||\cdot||')) \cap ({\rm cl}\, S_{-K}(||\cdot||')) = \emptyset$ 
	for some norm $||\cdot||' \sim ||\cdot||$.
	Applying Lemma \ref{lem:WeakCompact} ($3^\circ$, $4^\circ$), it follows that ${\rm cl}\, S_{-K}(||\cdot||')$ is compact under a), respectively, weakly compact under b). Consequently, the conclusion follows by applying Theorem \ref{th:main_scalarization_result_PEff} for $Y$ endowed with the norm $\psi(\cdot) := ||\cdot||'$ under a), respectively, endowed with the weak topology and $\psi(\cdot) := ||\cdot||'$ under b).
\end{proof}

For the finite-dimensional case, we obtain from Corollary \ref{cor:main_scalarization_result_PEff_A_normed} the following result. 

\begin{corollary} \label{cor:main_scalarization_result_PEff_A_normed_finite_dim}
	Assume that ($Y, ||\cdot||$) is a real finite-dimensional normed space with dimension two or higher, $K \subseteq Y$ is a nontrivial, closed, pointed, convex cone, $\bar x \in {\rm PEff}_{\mathbb{A}}(\Omega \mid f, K)$, and $A := \mathbb{A}(\bar x)$ is a closed cone in $Y$ with $\{0\} \neq f[\Omega] - f(\bar x) \subseteq A$.\\
	Then, the conclusion of Corollary \ref{cor:main_scalarization_result_PEff_A_normed} holds.
\end{corollary}

\begin{proof} By the well-known Krein-Rutman theorem (see, e.g., \cite[Th. 2.4.7]{Khanetal2015}) we know that $K^\# \neq \emptyset$, and by \cite[Th. 3.1 ($4^\circ$)]{GuenKhaTam23b} it follows that $0 \notin \cl\, S_{-K}(||\cdot||)$. So, the assumptions of Corollary \ref{cor:main_scalarization_result_PEff_A_normed} are satisfied.
\end{proof}

\begin{remark} \label{rem:PEFF_Henig_A}
    Note that Corollaries \ref{cor:main_scalarization_result_PEff_A_BP}, \ref{cor:main_scalarization_result_PEff_A_normed} and \ref{cor:main_scalarization_result_PEff_A_normed_finite_dim}  provide necessary conditions for $\mathbb{A}$-proper efficiency (in the sense of Definition \ref{def:pEff_A}), including the proper efficiency concepts in the senses of Benson, Borwein and Hurwicz (as mentioned in Remark \ref{rem:proper_efficiency_A}).
    Moreover, we gain scalarization results for the proper efficiency concept in the sense of Henig taking into account the known relationships between proper efficiency concepts in vector optimization.
	In order to emphasize this fact, consider a nontrivial cone $K \subseteq Y$ with normlike-base $B_K(\psi)$ and suppose that either $K$ is solid or $\psi$ is continuous. Assume that ${\rm PEff}_{He}(\Omega \mid f, K) \subseteq {\rm PEff}_{\mathbb{A}}(\Omega \mid f, K)$. Taking into account Remark \ref{rem:EffBPConeIsInPEffHenig}, in this subsection, we have provided sufficient conditions for the equivalence of the following assertions:
	\begin{itemize}
		\item[$1^\circ$] $\bar x \in {\rm PEff}_{He}(\Omega \mid f, K)$.
		\item[$2^\circ$] $\bar x \in {\rm PEff}_{\mathbb{A}}(\Omega \mid f, K)$.
		\item[$3^\circ$] There is $(x^*, \alpha) \in K^{a\#} \cap (Y^* \times \Rpos)$ such that $\bar x \in \Eff(\Omega \mid f, C_\psi(x^*, \alpha))$ (or any other equivalent assertion given in Proposition \ref{prop:scal_result_Eff_BP} is valid whenever the corresponding assumptions in Proposition \ref{prop:scal_result_Eff_BP} are satisfied),
	\end{itemize}
	where in a normed setting the third assertion is replaced by 
	\begin{itemize}
		\item[$3^\circ$'.] There is a norm $||\cdot||' \sim ||\cdot||$ 
		and a pair $(x^*, \alpha) \in K^{a\#}(||\cdot||') \cap (Y^* \times \mathbb{P})$ such that $\bar x \in \Eff(\Omega \mid f, C_{||\cdot||'}(x^*, \alpha))$
        (or any other equivalent assertion given in Proposition \ref{prop:scal_result_Eff_BP}, applied for $||\cdot||'$ in the role of $\psi$, is valid whenever the corresponding assumptions in Proposition \ref{prop:scal_result_Eff_BP} are satisfied).
	\end{itemize}
    From this perspective, $3^\circ$ (respectively, $3^\circ$') is a necessary and sufficient condition for the proper efficiency of $\bar x \in \Omega$.
\end{remark}

\begin{remark}
	Under the assumptions in Corollary \ref{cor:main_scalarization_result_PEff_A_BP} a) (in particular the convexity of $A$), we were able to replace the separation condition 
	\eqref{eq:cl_intersection}
	by
	$0 \notin {\rm cl}\, S_{-K}$ and $\bar x \in {\rm PEff}_{\mathbb{A}}(\Omega \mid f, K)$ (which implies the condition $A \cap (-K) = \{0\}$)
	in our scalarization results.
	In contrast, the scalarization result (for Henig and Benson proper efficiency concepts) by Garc\'{i}a-Casta\~{n}o, Melguizo-Padial and Parzanese \cite[Cor. 3.10]{CastanoEtAl2023} in real normed spaces still contains a separation condition (SSP), also under convexity of $A$ (ensured by a convexlikeness assumption on $f$). 

    It is worth mentioning that a novel aspect of our work in comparison to other conic scalarization related works in the normed space framework (e.g. \cite{Kasimbeyli2010, Kasimbeyli2013}, \cite{CastanoEtAl2023}) is the treatment of equivalent norms in the scalarization results (as done in Corollaries \ref{cor:main_scalarization_result_PEff_A_normed} and \ref{cor:main_scalarization_result_PEff_A_normed_finite_dim}).
\end{remark}

\begin{example} \label{ex:l2spaceBPtypeCone}
    As in Example \ref{ex:BPtypeCones} (b), consider the (not well-based) natural ordering cone $K := (\ell_2)_+$ in $\ell_2$, which is in fact a Bishop-Phelps cone $C_{\psi_y}(x^*_y, 1)$ based on $x^*_y \in C^\#$ and the continuous norm $\psi_y$. Note that $y = (y_n) \in \ell_2$ with $y_n > 0$ for all $n \in \mathbb{N}$. 
    Since $C$ is not well-based, we have $0 \in \cl\, S_{-K}(||\cdot||_2)$. Consequently, Corollary \ref{cor:main_scalarization_result_PEff_A_normed} can not be applied to this cone $K$. In view of Lemma \ref{lem:BPcone}, we know that \eqref{eq:sep_condition_9} is valid for $(x^*_y, 1)$ in the role of $(x^*, \alpha)$, and $\psi_y$ in the role of $\psi$. In particular, we have $0 \notin \cl\, S_{-K}(\psi_y)$. Consider $\bar x \in {\rm PEff}_{\mathbb{A}}(\Omega \mid f, K)$, and assume that $A := \mathbb{A}(\bar x)$ is a weakly closed cone in $Y$ with $\{0\} \neq f[\Omega] - f(\bar x) \subseteq A$ such that $\cl\, S_{A}^0(\psi_y)$ is weakly compact (hence $\cl_w\,B_A(\psi_y)\, (\subseteq \cl\, S_{A}^0(\psi_y))$ is weakly compact too). This ensures the validity of condition \eqref{ass:com3}.
    According to Corollary \ref{cor:main_scalarization_result_PEff_A_BP_2} (applied to $\ell_2$ with the weak topology), the conclusion of Theorem \ref{th:main_scalarization_result_PEff} is valid (with $\psi_y$ in the role of $\psi$). 
\end{example}

\subsection{Scalarization results for the concept of proper efficiency in the sense of Henig} \label{sec:scalarization_PEff_Henig}

Let us end the paper by stating some further scalarization results for the concept of proper efficiency in the sense of Henig by making special use of the structure of the dilating cones $D \in \mathcal{D}(K)$. From now on, we choose the cone $\bar D := Y \setminus (-\intt\, D)$ in the role of $A$.

\begin{lemma} \label{lem:dilating_cone_BP}
	Assume that $Y$ is a real locally convex space, $K \subseteq Y$ is a nontrivial cone, and $\bar D := Y \setminus (-\intt\, D)$ for some $D \in \mathcal{D}(K)$. Assume that $B_{K}$ and $B_{\bar D}$ are normlike-bases for $K$ and $\bar D$, respectively, constructed with respect to a continuous seminorm $\psi$. Suppose that ${\rm cl}\, S_{-K}$ is compact, and
	$$
	({\rm cl}\,S_{\bar D}^0) \cap ({\rm cl}\, S_{-K}) = \emptyset.
	$$
	Then, there exists $(x^*, \alpha) \in K^{a\#} \cap (Y^* \times \mathbb{P})$ such that
	$$
	K \setminus \{0\} \subseteq C^>_\psi(x^*, \alpha) = \intt\, C_\psi(x^*, \alpha) \subseteq \intt\, D.
	$$	
\end{lemma}

\begin{proof}
	First, note that $\emptyset \neq K \setminus \{0\} \subseteq C^>_\psi(x^*, \alpha) = \intt\, C_\psi(x^*, \alpha)$ for all $(x^*, \alpha) \in K^{a\#} \cap (Y^* \times \mathbb{P})$ taking into account the continuity of $\psi$. By Proposition \ref{prop:strict_cone_separation_if}, there exists $(x^*, \alpha) \in K^{a\#} \cap (Y^* \times \mathbb{P})$ such that
	\begin{equation}
		\label{eq:sepBarDand-K}
		x^*(d) + \alpha \psi(d) \geq  0 > x^*(-k) + \alpha \psi(k)   \quad \mbox{for all } d \in \bar D \mbox{ and all }k \in K \setminus \{0\}.
	\end{equation}	
	Any $x \in C^>_\psi(x^*, \alpha)$ satisfies $0 > x^*(-x) + \alpha \psi(x)$, hence $-x \notin \bar D$ in view of \eqref{eq:sepBarDand-K}, and so $x \in \intt\, D$.
\end{proof}

\begin{theorem} \label{th:main_scalarization_result_PEff_Henig_1}
	Assume that $Y$ is a real locally convex space, and $K \subseteq Y$ is a nontrivial cone. Let $\bar x \in {\rm PEff}_{He}(\Omega \mid f,K)$, i.e., there is $D \in \mathcal{D}(K)$ such that $\bar x \in {\rm Eff}(\Omega \mid f,D)$. Suppose that $D$ is closed. Define $\bar D := Y \setminus (-\intt\, D)$. Assume that $B_{K}$ and $B_{\bar D}$ are normlike-bases for $K$ and $\bar D$, respectively, constructed with respect to a continuous seminorm $\psi$. 
	Suppose that ${\rm cl}\, S_{-K}$ is compact, and
	$$
	({\rm cl}\,S_{\bar D}^0) \cap ({\rm cl}\, S_{-K}) = \emptyset.
	$$
	Then, the conclusion of Theorem \ref{th:main_scalarization_result_PEff} holds. 
\end{theorem}
\begin{proof}
	By Lemma \ref{lem:dilating_cone_BP} there exists $(x^*, \alpha) \in K^{a\#}  \cap (Y^* \times \mathbb{P})$ such that
	$$
	K \setminus \{0\} \subseteq C^>_\psi(x^*, \alpha) = \intt\, C_\psi(x^*, \alpha) \subseteq \intt\, D.
	$$
	Since $D$ is closed, we have $C_\psi(x^*, \alpha) \subseteq D$, hence $C_\psi(x^*, \alpha) \setminus \{0\}  \subseteq D \setminus \{0\}$. 
	We conclude that
	$$\bar x \in {\rm Eff}(\Omega \mid f,D) \subseteq {\rm Eff}(\Omega \mid f, C_\psi(x^*, \alpha)).$$ 
	Finally, noting that $C_\psi(x^*, \alpha) \neq \ell(C_\psi(x^*, \alpha))$ by Lemma \ref{lem:properties_BP} ($7^\circ$), Proposition \ref{prop:scal_result_Eff_BP} yields the rest of the conclusion.
\end{proof}

\begin{theorem}  \label{th:main_scalarization_result_PEff_Henig_2}
	Assume that $Y$ is a real locally convex space, and $K \subseteq Y$ is a nontrivial cone. Let $\bar x \in {\rm PEff}_{He}(\Omega \mid f,K)$, i.e., there is $D \in \mathcal{D}(K)$ such that $\bar x \in {\rm Eff}(\Omega \mid f,D)$.  Define $\bar D := Y \setminus (- \intt\, D)$. Assume that $B_{K}$ and $B_{\bar D}$ are normlike-bases for $K$ and $\bar D$, respectively, constructed with respect to a continuous seminorm $\psi$.  
	Suppose that ${\rm cl}\, S_{-K}$ is compact, and
	$$
	({\rm cl}\,S_{\bar D}^0) \cap ({\rm cl}\, S_{-K}) = \emptyset.
	$$
	Then, there exists $(x^*, \alpha) \in K^{a\#} \cap (Y^* \times \mathbb{P})$ such that	$\bar x \in {\rm WEff}(\Omega \mid f, C_\psi(x^*, \alpha))$.
	Moreover, if $\varphi$ is $-C_\psi(x^*, \alpha)$ representing, 
    then the other equivalent assertions given in Proposition \ref{prop:scal_result_WEff_BP} are valid whenever the corresponding assumptions in Proposition \ref{prop:scal_result_WEff_BP} are satisfied.
\end{theorem}

\begin{proof}
	By Lemma \ref{lem:dilating_cone_BP} there exists $(x^*, \alpha) \in K^{a\#}\cap (Y^* \times \mathbb{P})$ such that
	$$
	K \setminus \{0\} \subseteq C^>_\psi(x^*, \alpha) = \intt\, C_\psi(x^*, \alpha) \subseteq \intt\, D,
	$$
	hence $\bar x \in {\rm Eff}(\Omega \mid f,D) \subseteq {\rm WEff}(\Omega \mid f,D) \subseteq {\rm WEff}(\Omega \mid f, C_\psi(x^*, \alpha))$. 
	Proposition \ref{prop:scal_result_WEff_BP} finishes the proof.
\end{proof}

\begin{remark} \label{rem:optimality_cond_weak}
	Taking Remark \ref{rem:EffBPConeSubsetWEffBPConeSubsetPEffHenig} into account, in Theorems \ref{th:main_scalarization_result_PEff_Henig_1} and \ref{th:main_scalarization_result_PEff_Henig_2}, we derived necessary (and sufficient) conditions for the proper efficiency in the sense of Henig of $\bar x \in \Omega$.
    
	Proposition \ref{prop:separation_conditions_1} gives sufficient conditions for the equality $({\rm cl}\,S_{\bar D}^0) \cap ({\rm cl}\, S_{-K}) = \emptyset$, which is given as an assumption in Theorems \ref{th:main_scalarization_result_PEff_Henig_1} and \ref{th:main_scalarization_result_PEff_Henig_2}.
	Suppose that there is $(x^*, \alpha) \in Y^* \times \mathbb{P}$ such that
	\eqref{eq:sep_condition_9} is valid.
	Notice that the set $\bar D$ is a nontrivial, closed cone in $Y$ with $\bar D \cap (-K) = \{0\}$.
	Assume that the nontrivial, convex cone $K$ is closed and that either ${\rm conv}({\rm cl}(B_{\bar D}) \cup \{0\})$ is closed or ${\rm cl}\,B_{\bar D}$ is compact.
	Then, Proposition \ref{prop:separation_conditions_1} ensures $
	({\rm cl}\,S_{\bar D}^0) \cap ({\rm cl}\, S_{-K}) = \emptyset.
	$
\end{remark}

\begin{remark} 
	The approach in this Section \ref{sec:scalarization_PEff_Henig} follows basically some ideas by Kasimbeyli \cite[Th. 4.4 and 5.7]{Kasimbeyli2010}, who deals (in a normed setting) with $\varepsilon$-conic neighborhoods for $K$ (as dilating cones) and supposes a separation property for a whole family of dilating cones. 
	
	In Section \ref{sec:scalarization_PEff_A}, one has a separation condition for $A = \mathbb{A}(\bar x)$ and $-K$ (for a given $\bar x$), and one can further ensure that $A$ is convex. Here in Section \ref{sec:scalarization_PEff_Henig} we have a separation condition for $\bar D = Y \setminus (-\intt\, D)$ and $-K$. One has to notice that the complement cone $\bar D$ is usually not convex if $D$ is convex. Moreover, the cone $\mathbb{A}(\bar x)$, which has a particular structure for specific choices of $\mathbb{A}$, is easier to handle than the cone $\bar D = Y \setminus (-\intt\, D)$ based on a general dilating cone $D \in \mathcal{D}(K)$.
\end{remark}

\section{Conclusions}  \label{sec:conclusions}

It is well-known that several scalarizing functions used in sca\-larization methods in vector optimization and in the treatment of optimization problems under uncertainty admit certain cone-representation and cone-monotonicity properties (e.g., methods in the sense of Gerstewitz  \cite{Gerstewitz1983,Gerstewitz1984}; Kasimbeyli \cite{Kasimbeyli2010, Kasimbeyli2013}; Pascoletti and Serafini \cite{PascolettiSerafini1984}; Zaffaroni \cite{Zaffaroni2003}).  In the general framework of real topological-linear spaces, our paper provides new insights into the relationships between a vector optimization problem based on an original (convex) cone, an updated vector optimization problem based on an enlarging / dilating cone of Bishop-Phelps type, and scalar optimization problems involving (Bishop-Phelps type) cone-representing and cone-monotone scalarizing functions. Thus, our new scalarization results with respect to the concepts of weak efficiency and different types of proper efficiency  (e.g., in the sense of Benson; Borwein; Henig; Hurwicz) are useful for several established scalarization methods in vector optimization.

In future research, we would like to derive Bishop-Phelps type scalarization results for approximate solutions of vector optimization problems in real topological-linear spaces, which would extend recent scalarization results obtained by Garc\'{i}a-Casta\~{n}o, Melguizo-Padial and Parzanese \cite{CastanoEtAl2023}. Moreover, the derivation of duality statements for such vector optimization problems based on Bishop-Phelps type scalarization is of interest (see also Kasimbeyli and Karimi \cite{KasimbeyliKarimi2021}). 


\section*{Data Availability}
No datasets were generated or analyzed during the current study.

\section*{Acknowledgments}
The authors would like to express their gratitude to the editor and the anonymous reviewers for their valuable comments, which helped to improve the manuscript.

\bibliography{references}


\end{document}